\documentclass[final]{amsart}

\usepackage[utf8]{inputenc}
\usepackage[T1]{fontenc}
\usepackage[english]{babel}
\usepackage{amsmath}
\usepackage{amssymb}
\usepackage{amsthm}
\usepackage{mathtools}
\usepackage{mathrsfs}
\usepackage{enumerate}
\usepackage{hyperref}
\usepackage{commath}
\usepackage{cleveref}
\usepackage{tikz}
\usepackage{dsfont,esint}
\usepackage{comment}
\usepackage[shortlabels]{enumitem}

\usepackage[sort, numbers]{natbib}
\usepackage{doi}

\usepackage{scrextend}

\usepackage{subcaption}

\newcommand{\N}{\mathbb{N}}
\newcommand{\R}{\mathbb{R}}

\newcommand{\Leb}{\mathscr{L}}
\newcommand{\TrFun}{\mathcal{T}}
\newcommand{\E}{\mathcal{E}}
\newcommand{\gfun}{\mathscr{F}}
\newcommand{\hfun}{\mathscr{G}}
\newcommand{\TrDis}{{\widetilde{Wb}_{2}}}
\newcommand{\cone}{\mathscr{S}}
\newcommand{\const}{\mathfrak{c}}
\newcommand{\cost}{\mathcal{C}}
\DeclareMathOperator{\Id}{Id}
\newcommand{\ovom}{{\overline \Omega}}
\DeclareMathOperator{\Adm}{Adm}
\DeclareMathOperator{\Tr}{Tr}
\DeclareMathOperator{\Opt}{Opt}
\DeclareMathOperator{\Lip}{Lip}
\DeclareMathOperator{\dive}{div}
\DeclareMathOperator{\diam}{diam}
\newcommand{\Ent}{\mathcal{H}}
\newcommand{\KR}{\mathrm{\widetilde{KR}}}

\newcommand{\1}{\mathds{1}}

\newcommand{\comma}{\,\mathrm{,}\;\,}
\newcommand{\semicolon}{\,\mathrm{;}\;\,}
\newcommand{\fstop}{\,\mathrm{.}}
\DeclareMathOperator*{\essinf}{-ess\,inf}
\newcommand{\slope}[2]{\left\lvert \partial_{\, #2} #1 \right\rvert}
\newcommand{\slopesmall}[2]{\bigl\lvert \partial_{\, #2} #1 \bigr\rvert}
\DeclareMathOperator*{\argmin}{arg\,min}

\newtheorem{theorem}{Theorem}[section]
\newtheorem{proposition}[theorem]{Proposition}
\newtheorem{lemma}[theorem]{Lemma}
\newtheorem{corollary}[theorem]{Corollary}

\newtheorem{example}[theorem]{Example}%
\newtheorem{remark}[theorem]{Remark}%

\newtheorem{definition}[theorem]{Definition}%

\numberwithin{equation}{section}

\begin{document}

\title[Variational structures for Fokker--Planck with Dirichlet BC]{Variational structures \\ for the Fokker--Planck equation \\ with general Dirichlet boundary conditions}

\author{Filippo Quattrocchi}
\address{Institute of Science and Technology Austria, Am Campus 1, 3400 Klosterneuburg, Austria}
\curraddr{}
\email{filippo.quattrocchi@ista.ac.at}
\thanks{
	The author would like to thank Jan~Maas for suggesting this project and for many helpful comments, Antonio~Agresti, Lorenzo~Dello~Schiavo and Julian~Fischer for several fruitful discussions, Oliver~Tse for pointing out the reference~\cite{KimKooSeo23}, and the anonymous reviewer for carefully reading this manuscript and providing valuable suggestions. He also gratefully acknowledges support from the Austrian Science Fund (FWF) project \href{https://doi.org/10.55776/F65}{10.55776/F65}.
}
\thanks{This version of the article  has been accepted for publication, after peer review but is not the Version of Record and does not reflect post-acceptance improvements, or any corrections. The Version of Record is available online at: \url{http://dx.doi.org/10.1007/s00526-025-03193-1}.}



\date{\today}

\dedicatory{}

\begin{abstract}

	We prove the convergence of a modified Jordan--Kinderlehrer--Otto scheme to a solution to the Fokker--Planck equation in~$\Omega \Subset \R^d$ with general---strictly positive and temporally constant---Dirichlet boundary conditions. We work under mild assumptions on the domain, the drift, and the initial datum.

	In the special case where~$\Omega$ is an interval in~$\R^1$, we prove that such a solution is a gradient flow---curve of maximal slope---within a suitable space of measures, endowed with a modified Wasserstein distance.
	
	Our discrete scheme and modified distance draw inspiration from contributions by A.~Figalli and N.~Gigli [J.~Math.~Pures Appl.~94, (2010), pp.~107--130], and J.~Morales [J.~Math.~Pures Appl.~112, (2018), pp.~41--88] on an optimal-transport approach to evolution equations with Dirichlet boundary conditions. Similarly to these works, we allow the mass to flow from/to the boundary~$\partial \Omega$ throughout the evolution. However, our leading idea is to also keep track of the mass at the boundary by working with measures defined on the whole closure~$\ovom$.
	
	The driving functional is a modification of the classical relative entropy that also makes use of the information at the boundary.
	As an intermediate result, when~$\Omega$ is an interval in~$\R^1$, we find a formula for the descending slope of this geodesically nonconvex functional.
\end{abstract}

\maketitle

\setcounter{tocdepth}{1}

\section{Introduction}
The subject of this paper is the linear Fokker--Planck equation
\begin{equation} \label{eq:fp0}
	\frac{\dif}{\dif t} \rho_t 
	=
	\dive \left( \nabla \rho_t + \rho_t \nabla V \right)
\end{equation}
on a bounded Euclidean domain~$\Omega \subseteq \R^d$ combined with general---strictly positive and constant in time---Dirichlet boundary conditions, and with nonnegative initial data. We want to approach this problem by applying the theory of \emph{optimal transport}, which, since the seminal works of R.~Jordan, D.~Kinderlehrer, and F.~Otto~\cite{JordanKinderlehrerOtto98,Otto99,Otto01}, has proven effective in the study of a number of evolution equations.

Existence, uniqueness, and appropriate estimates are often consequence of a peculiar structure of the problem. Important instances are those PDEs which can be seen as \emph{gradient flows}. In fact, it has been proven that several equations, including Fokker--Planck on~$\R^d$, are gradient flows in a space of probability measures endowed with the~$2$-Wasserstein distance
\[
W_2(\mu,\nu) 
\coloneqq
\inf_{\gamma} \sqrt{\int \abs{x-y}^2 \dif \gamma(x,y) } \comma
\]
where the infimum is taken among all couplings~$\gamma$ between~$\mu$ and~$\nu$, i.e., measures with marginals~$\pi^1_\# \gamma = \mu$ and~$\pi^2_\# \gamma = \nu$.
For such PDEs, existence can be deduced from the convergence of the discrete-time approximations given by the Jordan--Kinderlehrer--Otto variational scheme (also known, in a more general metric setting, as De~Giorgi's minimizing movement scheme~\cite{DeGiorgi93})
\begin{equation} \label{eq:jko}
	\rho_{(n+1)\tau}^\tau \dif x
	\in
	{\arg\min}_{\mu} \left(\mathcal F(\mu) + \frac{W_2^2(\mu, \rho_{n \tau}^\tau \dif x)}{2\tau} \right) \comma \qquad n \in \N_0\comma
\end{equation}
where~$\mathcal F$ is a functional that depends on the equation, and~$\tau > 0$ is the time step.

When applied on a bounded Euclidean domain, this approach produces solutions \emph{with Neumann boundary conditions}. This fact is inherent in the choice of the metric space (probability measures with the distance~$W_2$) in which the flow evolves. Intuitively, Neumann boundary conditions are natural because a curve of probability measures, by definition, conserves the total mass; see also the discussion in~\cite{Santambrogio17}.

In order to deal with Dirichlet boundary conditions, A.~Figalli and N.~Gigli defined in~\cite{FigalliGigli10} a modified Wasserstein distance~$Wb_2$ that gives a special role to the boundary~$\partial \Omega$. Despite measuring a distance between nonnegative measures on~$\Omega$, the metric~$Wb_2$ is defined as an infimum over measures~$\gamma$ on the product of the topological closures~$\ovom \times \ovom$, and only the restrictions of the marginals~$\pi^1_\# \gamma$ and~$\pi^2_\# \gamma$ to~$\Omega$ are prescribed (see the original paper \cite{FigalliGigli10} or \Cref{sec:figalli-gigli}). In this sense, the boundary~$\partial \Omega$ can be interpreted as an infinite reservoir, where mass can be taken and deposited freely. The main result in~\cite{FigalliGigli10} is the convergence of the scheme
\[
\rho^\tau_{(n+1)\tau}
\in
{\arg\min}_{\rho} \left( \int_{\Omega} \bigl(\rho \log \rho -\rho +1\bigr) \dif x  + \frac{Wb_2^2(\rho\dif x, \rho^\tau_{n \tau} \dif x )}{2\tau}\right)
\comma \qquad n \in \N_0 \comma 
\]
as~$\tau \downarrow 0$, to a solution to the heat equation with the \emph{constant} Dirichlet boundary condition~$\rho|_{\partial \Omega} = 1$. More generally, it was observed in~\cite[Section~4]{FigalliGigli10} that the same scheme with a suitably modified entropy functional %
converges to solutions to the linear Fokker--Planck equation~\eqref{eq:fp0} with the boundary condition~$\rho|_{\partial \Omega} = e^{-V}$. In particular, this theory covers the heat equation with \emph{any constant and strictly positive} Dirichlet boundary condition.

In a more recent contribution, J.~Morales~\cite{Morales18} proved convergence of a similar discrete scheme for a family of reaction-diffusion equations with drift, subject to rather \emph{general} Dirichlet boundary conditions. In this scheme, the distance between measures is replaced by~$\tau$-dependent transportation costs. %
Morales' work, together with~\cite{FigalliGigli10}, is the starting point of the present paper.

\subsubsection*{Related literature}

The case of the heat flow with \emph{vanishing} Dirichlet boundary conditions was studied by A.~Profeta and K.-T.~Sturm in~\cite{ProfetaSturm20}. They defined `charged probabilities' and a suitable distance on them. This metric is built upon the idea that mass can touch the boundary and be reflected, as with the classical Wasserstein distance, but possibly changing the charge (positive to negative or vice versa). One of their results is the \emph{Evolution Variational Inequality} (see~\cite{AmbrosioGigliSavare08}) for such a heat flow.

D.~Kim, D.~Koo and G.~Seo~\cite{KimKooSeo23} adapted the setting of~\cite{FigalliGigli10} to porous medium equations~$\partial_t \rho_t = \Delta \rho^\alpha$ ($\alpha>1$) with \emph{constant} boundary conditions.

M.~Erbar and G.~Meglioli~\cite{ErbarMeglioli24} generalized the result of~\cite{KimKooSeo23} to a larger class of diffusion equations with constant boundary conditions. They also established a dynamical characterization of~$Wb_2$, in the spirit of the Benamou--Brenier formula for~$W_2$~\cite{BenamouBrenier00}.

J.-B.~Casteras, L.~Monsaingeon, and F.~Santambrogio~\cite{CasterasMonsaingeonSantambrogio24} found the Wasserstein gradient flow structure for the equation arising from the so-called Sticky Brownian Motion, i.e.,~the Fokker--Planck equation together with boundary conditions of Dirichlet type that also evolve in time subject to diffusion and drift on the boundary. Namely, denoting by~$\partial_{\boldsymbol n}$ the outer normal derivative,
\begin{equation} \label{eq:CastMonSant}
	\begin{cases}
		\partial_t \rho = \Delta \rho &\text{in } \Omega \comma \\
		\rho = \gamma &\text{on } \partial \Omega \comma \\
		\partial_t \gamma = \Delta_{\partial \Omega} \gamma - \partial_{\boldsymbol n} \rho &\text{in } \partial \Omega \fstop
	\end{cases}
\end{equation}

M.~Bormann, L.~Monsaingeon, D.~R.~M.~Renger, and M.~von~Renesse~\cite{BormannMonsaingeonRengerVonRenesse25} recently proved a {negative} result. If we modify~\eqref{eq:CastMonSant} by weakening the diffusion on the boundary (i.e.,~we multiply~$\Delta_{\partial \Omega} \gamma$ by a factor~$a \in (0,1)$) the resulting problem is \emph{not} a gradient flow of the entropy in the $2$-Wasserstein space built from any reasonably regular metric on~$\overline \Omega$.

\subsection*{Our contribution}

In this work, we present two novel results:
\begin{enumerate}
	\item We prove convergence of a modified Jordan--Kinderlehrer--Otto scheme to a solution to the Fokker--Planck equation with general Dirichlet boundary conditions under mild regularity assumptions. To do this, we adopt a \emph{different point of view} compared to~\cite{FigalliGigli10,Morales18,KimKooSeo23}: our scheme is defined on a subset~$\cone$ of the signed measures \emph{on the closure~$\ovom$}, rather than on measures on~$\Omega$. %
	\item In dimension~$d=1$, we determine that this solution is also a \emph{curve of maximal slope} for a functional~$\Ent$ in an appropriate metric space~$(\cone,\TrDis)$.
\end{enumerate}
Let us now explain in detail the extent of these contributions and provide precise statements.

\subsubsection*{Convergence of a modified JKO scheme}
We look at the boundary-value problem
\begin{equation} \label{eq:fpStrong}
	\begin{cases}
		\displaystyle \frac{\dif}{\dif t} \rho_t 
		=
		\dive \left( \nabla \rho_t + \rho_t \nabla V \right) &\text{in } \Omega \comma \\
		\rho_t|_{\partial \Omega} = e^{\Psi-V} &\text{on } \partial \Omega \comma \\
		\rho_{t=0} = \rho_0 \fstop
	\end{cases}
\end{equation}
Here,~$\Omega \subseteq \R^d$ is a \emph{bounded} open set and~$\rho_0,\Psi,V$ are given functions, with~$\rho_0 \ge 0$. The function~$\Psi$ can be tuned to obtain the desired boundary condition.

We introduce the set~$\cone$ of all signed measures on~$\ovom$ with
\begin{equation} \label{eq:S}
	\mu|_\Omega \ge 0 \quad \text{and} \quad \mu(\ovom) = 0 \fstop
\end{equation}
We also define
\begin{equation} \label{eq:E}
	\E(\rho) \coloneqq \int_\Omega \bigl(\rho \log \rho + (V-1) \rho + 1 \bigr) \dif x \comma \qquad \rho \colon \Omega \to \R_+ \comma
\end{equation}
and, for~$\mu \in \cone$,
\begin{equation} \label{eq:Ent} \Ent (\mu) \coloneqq \begin{cases}
		\displaystyle \E(\rho)  + \int \Psi \dif \mu|_{\partial \Omega} &\text{if } \mu|_\Omega = \rho \dif x \comma \\
		\infty &\text{otherwise.}
\end{cases} \end{equation}
In \Cref{sec:functionals}, we will define a transportation-cost functional~$\TrFun$ on~$\cone$. With it, we can consider the scheme 
\begin{equation} \label{eq:jkoTrFun0}
	\mu_{(n+1) \tau}^{\tau} \in \argmin_{\mu \in \cone} \, \left(\Ent(\mu) + \frac{\TrFun^2(\mu, \mu_{n\tau}^\tau)}{2\tau} \right) \comma \qquad n \in \N_0 \comma \tau > 0 \comma
\end{equation}
starting from some~$\mu_0^\tau = \mu_0 \in \cone$, independent of~$\tau$, such that the restriction~$\mu_0|_\Omega$ is absolutely continuous with density~$\rho_0$. These sequences are extended to maps~$t \mapsto \mu^\tau_t$, constant on the intervals~$\bigl[n\tau,(n+1)\tau\bigr)$ for every~$n \in \N_0$, namely:
\begin{equation} \label{eq:defextension}
	\mu_t^\tau \coloneqq \mu_{\lfloor t/\tau \rfloor \tau}^\tau \comma \qquad t \in [0,\infty) \fstop
\end{equation}

\begin{theorem}
	\label{Theorem_1.1}
	Assume that~$\int_\Omega \rho_0 \log \rho_0 \dif x < \infty$, that~$\Psi \colon \ovom \to \R$ is Lipschitz continuous, and that\footnote{By~$V \in W^{1,d+}_\mathrm{loc}(\Omega)$ we mean that for every~$\omega \Subset \Omega$ open there exists~$p = p(\omega) > d$ such that~$V \in W^{1,p}(\omega)$; see also \Cref{sec:functions}.} $V \in W^{1,d+}_\mathrm{loc}(\Omega) \cap L^\infty(\Omega)$.
	Then:
	\begin{enumerate}
		\item \emph{Well-posedness:} The maps~$( t \mapsto \mu_t^\tau )_\tau$ resulting from the scheme~\eqref{eq:jkoTrFun0} are well-defined and uniquely defined: for every~$n$ and~$\tau$, there exists a minimizer in~\eqref{eq:jkoTrFun0} and it is unique.
		\item \emph{Convergence:} When~$\tau \to 0$, up to subsequences, the maps~$\bigl(t \mapsto \mu^\tau_t|_\Omega\bigr)_\tau$ converge pointwise w.r.t.~the Figalli--Gigli distance~$Wb_2$ to a curve of absolutely continuous measures~$t \mapsto \rho_t \dif x$. For every~$q \in [1,\frac{d}{d-1})$, convergence holds also in~$L^1_\mathrm{loc}\bigl( (0,\infty); L^q(\Omega))$.
		\item \emph{Equation:} This limit curve is a weak solution to the Fokker--Planck equation~\eqref{eq:fp0}; see \Cref{sec:defFokker}.
		\item \emph{Boundary condition:} The function~$t \mapsto \left(\sqrt{\rho_t e^V} - e^{\Psi/2}\right)$ belongs to the space~$L^2_\mathrm{loc}\bigl([0,\infty ); W^{1,2}_0(\Omega)\bigr)$.
	\end{enumerate}
\end{theorem}

\begin{remark}
	We assume that~$\Psi$ is defined on the whole set~$\ovom$ in order to make sense of the inclusion~$\sqrt{\rho_t e^V} - e^{\Psi/2} \in W^{1,2}_0(\Omega)$ also when~$\partial \Omega$ is not smooth enough to have a trace operator. Note that, if we are given a Lipschitz continuous function~$\Psi_0 \colon \partial \Omega \to \R$, we can extend it to a Lipschitz function on~$\ovom$ via
	\[
	\Psi(x) \coloneqq \inf_{y \in \partial \Omega} \left( \Psi_0(y) + (\Lip \Psi_0 ) \abs{x-y} \right) \fstop
	\]
\end{remark}

\begin{remark}
	If~$V$ is Lipschitz continuous \emph{only in a neighborhood of~$\partial \Omega$}, then it is possible to find~$\Psi$, Lipschitz as well, in order for~$e^{\Psi-V}$ to match \emph{any} uniformly positive and Lipschitz boundary condition.
\end{remark}

\begin{remark}
	Throughout the proof of \Cref{Theorem_1.1}, we also show:
	\begin{itemize}
		\item time contractivity of suitably truncated and weighted~$L^q$ norms of~$\mu_t^\tau|_\Omega$ (see \Cref{prop:Lq}),
		\item upper bounds on the~$L^q$ norms of~$\mu^\tau_t|_\Omega$, for every~$t > 0$ (see \Cref{lemma:lebesgueBound}),
		\item upper bounds on time averages of the~$W^{1,2}$ norm of~$\sqrt{\rho^\tau_t e^V}$, where~$\rho^\tau_t$ is the density of~$\mu^\tau_t|_\Omega$ (see \Cref{lemma:sobolev}).
	\end{itemize}
	Furthermore, these estimates (assuming~$q \in [1,\frac{d}{d-1})$ in the first two) pass to the limit as~$\tau \to 0$, i.e.,~analogous properties hold for the curve~$t \mapsto \rho_t$.
\end{remark}

As mentioned, the conceptual difference between the present work and~\cite{FigalliGigli10,Morales18,KimKooSeo23} is that we make use of signed measures on the full closure~$\ovom$. In this regard, our approach is similar to those of~\cite{CasterasMonsaingeonSantambrogio24,Monsaingeon21}. The idea is that, due to the boundary condition we have to match, it is convenient to keep track of the mass at the boundary and to consider a functional that makes use of this information (namely,~$\Ent$).

On a more technical note, although \Cref{Theorem_1.1} is similar to~\cite[Theorem 4.1]{Morales18}, the latter is not applicable to the Fokker--Planck equation~\eqref{eq:fp0} without reaction term due to \cite[Assumptions~(C1)-(C9)]{Morales18} (see in particular~(C7)). Furthermore, we achieve significant improvements in the hypotheses:
\begin{itemize}
	\item The boundary~$\partial \Omega$ does not need to have \emph{any} regularity, as opposed to Lipschitz and with the interior ball condition. 
	\item There is no uniform bound on~$\rho_0$ from above or below by positive constants. Only nonnegativity and the integrability of~$\rho_0 \log \rho_0$ are assumed.
	\item The function~$V$ is not necessarily Lipschitz continuous. Rather, it is required to be bounded and to have suitable local Sobolev regularity.
\end{itemize}
These weak assumptions make it more involved to prove Lebesgue and Sobolev bounds for~$\mu_t^\tau$, as well as the strong convergence of the scheme, which in turn allows us to characterize the limit. Indeed:
\begin{itemize}
	\item When~$\rho_0$ is bounded, or lies in some~$L^q$, it is possible to propagate these properties along~$t \mapsto \mu^\tau_t|_\Omega$; see~\cite[Proposition~5.3]{Morales18} and \Cref{prop:Lq}. With our weak assumptions on~$\rho_0$, we are still able to propagate the~$L^1$ bound, but also need to establish suitable Sobolev estimates (see \Cref{prop:sobolevReg} and \Cref{lemma:sobolev}) and make use of the Sobolev embedding theorem in order to get stronger integrability (see \Cref{lemma:lebesgueBound}) and convergence in~$L^1_\mathrm{loc}\bigl((0,\infty);L^q(\Omega)\bigr)$ (see \Cref{lemma:improvedconv}).
	\item If~$\partial \Omega$ is not regular enough, we cannot directly apply the Sobolev embedding theorem for~$W^{1,2}$ functions. Since the Sobolev continuous embedding holds for~$W^{1,2}_0$ functions regardless of the domain regularity, we are still able to apply it after establishing suitable boundary conditions for~$\mu^\tau_t|_\Omega$; see \Cref{prop:sobolevReg}.
	\item When~$V$ is not Lipschitz, we need an extra approximation procedure to prove that~$\mu^\tau_t|_\Omega$ is Sobolev regular and satisfies a precursor of the Fokker--Planck equation; see \Cref{prop:sobolevReg} and \Cref{lem:sobReg}.
	\item Another issue with~$\partial \Omega$ not being regular is in applying (a variant of) the Aubin--Lions lemma to prove convergence of the scheme. One of its assumptions is a compact embedding of functional spaces, which would follow from the Rellich--Kondrachov theorem if~$\Omega$ were regular enough. To overcome it, we use the Rellich--Kondrachov theorem on \emph{smooth subdomains} and take advantage of the integrability estimates to promote local~$L^q$ convergence to convergence in~$L^q(\Omega)$; see \Cref{lemma:improvedconv}.
\end{itemize}

\subsubsection*{Curve of maximal slope}

Our second main result is a strengthened version of \Cref{Theorem_1.1} in the case where~$\Omega$ is an interval in~$\R^1$ and~$V \in W^{1,2}(\Omega)$. In this setting, we are able to define a \emph{true} metric~$\TrDis$ on~$\cone$, construct piecewise constant maps with the scheme 
\begin{align} \label{eq:jkoTrDis0}
	\begin{split}
		\mu_{(n+1) \tau}^{\tau} &\in \argmin_{\mu \in \cone} \, \left( \Ent(\mu) + \frac{\TrDis^2(\mu, \mu_{n\tau}^\tau)}{2\tau} \right) \comma \qquad n \in \N_0 \comma \tau > 0 \comma \\
		\mu^\tau_0 &= \mu_0 \comma
	\end{split}
\end{align}
for a fixed~$\mu_0$ with~$\mu_0|_\Omega = \rho_0 \dif x$, show that they \emph{coincide} with those of \Cref{Theorem_1.1}, and prove that their limit is a \emph{curve of maximal slope} in~$(\cone,\TrDis)$.

\begin{theorem} \label{Theorem_1.5}
	Assume that~$\Omega = ( 0,1 )$, that~$\int_0^1 \rho_0 \log \rho_0 \dif x < \infty$, and that~$V \in W^{1,2}(0,1)$. Then:
	\begin{enumerate}
		\item \label{st:main201} If~$\tau$ is sufficiently small, the maps~$( t \mapsto \mu_t^\tau )_\tau$ resulting from the scheme \eqref{eq:jkoTrDis0} are well-defined, uniquely defined, and \emph{coincide with those of~\Cref{Theorem_1.1}}.
		\item \label{st:main202} When~$\tau \to 0$, up to subsequences, the maps~$( t \mapsto \mu_t^\tau )_\tau$ converge pointwise w.r.t.~$\TrDis$ to a curve~$t \mapsto \mu_t$.
		\item \label{st:main203} The convergence~$\mu^\tau|_\Omega \to_\tau \mu|_\Omega$ also holds in~$L^1_\mathrm{loc} \bigl( (0,\infty); L^q(0,1) \bigr)$ for every~$q \in [1,\infty)$. The curve~$t \mapsto \mu_t|_\Omega$ is a weak solution to the Fokker--Planck equation. Denoting by~$\rho_t$ the density of~$\mu_t|_\Omega$, the map~$t \mapsto \left(\sqrt{\rho_t e^V} - e^{\Psi/2}\right)$ belongs to~$L^2_\mathrm{loc}\bigl([0,\infty ); W^{1,2}_0(0,1)\bigr)$.
		\item \label{st:main204} The map~$t \mapsto \mu_t$ is a curve of maximal slope for the functional~$\Ent$ in the metric space~$(\cone, \TrDis)$, with respect to the descending slope~$\slope{\Ent}{\TrDis}$;
		see \Cref{sec:metricGF}.
	\end{enumerate}
\end{theorem}

Within the general theory of gradient flows in metric spaces developed by L.~Ambrosio, N.~Gigli, and G.~Savar\'e in~\cite{AmbrosioGigliSavare08} (see~\cite{Santambrogio17} for an overview), the `curve of maximal slope' is one of the metric counterparts of the gradient flow in the Euclidean space. In the context of PDEs with Dirichlet boundary conditions, other proofs of this metric characterization in a (Wasserstein-like) space of measures are given in~\cite{ProfetaSturm20,KimKooSeo23,ErbarMeglioli24}. To be precise, the result of~\cite[Proposition~1.20]{ProfetaSturm20} is an `Evolution Variational Inequality' (EVI) characterization, which \emph{implies} a formulation as curve of maximal slope by~\cite[Proposition~4.6]{AmbrosioGigli13}. By \Cref{prop:notconv}, our functional~$\Ent$ is not semiconvex and, therefore, we do not expect an EVI characterization in our setting; see \cite[Theorem~3.2]{DaneriSavare08}. Let us also point out that the `curve of maximal slope' characterizations in~\cite{KimKooSeo23,ErbarMeglioli24} %
use the \emph{relaxed} descending slope (see~\cite[Equation~(2.3.1)]{AmbrosioGigliSavare08}), which yields a weaker notion of gradient flow compared to ours. In fact, establishing that the descending slope is lower semicontinuous is the main difficulty in proving \Cref{Theorem_1.5}. Indeed, the lower semicontinuity of the slope is usually derived from the geodesic (semi)convexity of the functional via \cite[Corollary~2.4.10]{AmbrosioGigliSavare08}, but~$\Ent$ is not geodesically semiconvex by \Cref{prop:notconv}.%

Nonetheless, in dimension~$d=1$, we are able to find an \emph{explicit formula} for the descending slope of~$\Ent$ in~$(\cone,\TrDis)$ without resorting to geodesic convexity. As a corollary, we also give an answer, again in dimension~$d=1$, to the problem left open in~\cite{FigalliGigli10} of identifying the descending slope~$\slope{\E}{Wb_2}$ of~$\E$ %
with respect to the Figalli--Gigli distance~$Wb_2$.%

\begin{theorem}[see \Cref{cor:main3}] \label{thm:main30}
	Assume that~$V \in W^{1,2}(0,1)$. For every~$\rho \in L^1_{+}(0,1)$, we have the formula
	\begin{equation}
		\slope{\E}{Wb_2}^2 (\rho)
		=
		\begin{cases}
			4 \displaystyle \int_0^1 \left( \partial_x \sqrt{\rho e^V} \right)^2 e^{-V} \dif x &\text{if }  \sqrt{\rho e^V}-1 \in W^{1,2}_0(0,1) \comma \\ 
			\infty &\text{otherwise.}
		\end{cases}
	\end{equation}
	Additionally,~$\slope{\E}{Wb_2}$ is lower semicontinuous with respect to~$Wb_2$.
\end{theorem}

We believe that the same formula should hold true also in higher dimension. A similar open problem is~\cite[Conjecture~2]{CasterasMonsaingeonSantambrogio24}.

\subsection*{Plan of the work}
In \Cref{sec:heuristics}, we formally derive the objects (entropy and transportation functionals) that appear in the schemes~\eqref{eq:jkoTrFun0} and~\eqref{eq:jkoTrDis0}.%

In \Cref{sec:preliminaries}, we introduce notation, terminology, and assumptions that are in place throughout the paper, we recall some definitions from the theory of gradient flows in metric spaces, as well as the Figalli--Gigli distance of~\cite{FigalliGigli10}, and we define rigorously the transportation functionals~$\TrFun$ and~$\TrDis$.

In \Cref{sec:properties}, we gather the main properties of these functionals and of the corresponding admissible transport plans. In particular, we show that~$\TrDis$ is a true metric when~$\Omega$ is a finite union of one-dimensional intervals.

In \Cref{sec:main10}, we prove \Cref{Theorem_1.1}.

In Sections~\ref{sec:slope}-\ref{sec:main20}, we focus on the case where~$\Omega = (0,1) \subseteq \R^1$. In \Cref{sec:slope}, we find a formula for the slope of~$\Ent$ in the metric space~$(\cone, \TrDis)$ and prove, as a corollary, \Cref{thm:main30}. In \Cref{sec:main20}, making use of \Cref{Theorem_1.1} and of the slope formula, we prove \Cref{Theorem_1.5}.

Appendix~\ref{sec:appendix} contains some additional results on~$\TrDis$. Particularly, we prove the lack of geodesic $\lambda$-convexity for~$\Ent$ when~$\Omega = (0,1)$.

\section{Formal derivation} \label{sec:heuristics}
Let us work at a completely formal level and postulate that a solution to the Fokker--Planck equation~\eqref{eq:fpStrong} is the ``Wasserstein-like'' gradient flow of some functional~$\mathcal F$. By this we mean the following:
\begin{enumerate}
	\item the motion of~$\rho_t$ in~$\Omega$ is governed by the continuity equation%
	\begin{equation} \label{eq:ce}
		\frac{\dif}{\dif t} \rho_t = -\dive(\rho_t \boldsymbol v_t) \comma
	\end{equation}
	for some velocity field~$\boldsymbol v_t$,
	\item the time-derivative of~$\rho_t$ equals the inverse of the Wasserstein gradient of~$\mathcal F$ at~$\rho_t$ for every~$t$, in the sense that for every sufficiently nice curve~$s \mapsto f_s$ of functions on~$\Omega$ starting at~$f_0 = \rho_t$ we have
	\begin{equation} \label{eq:heuristics1}
		\frac{\dif}{\dif s} \mathcal F(f_s \dif x)\Big|_{s=0} = -\int_\Omega \langle \boldsymbol v_t, \nabla \psi \rangle \rho_t \dif x \comma \quad \text{where } \frac{\dif}{\dif s} f_s \Big|_{s=0} = -\dive ( \rho_t \nabla \psi ) \fstop
	\end{equation}
\end{enumerate}
As we want to retrieve the Fokker--Planck equation, a reasonable choice for~$\mathcal F$ seems to be
\begin{equation} \label{eq:ent0}
	\mathcal F_0(\rho \dif x) \coloneqq \int_{\Omega} \bigl( \rho \log \rho + (V-1) \rho + 1 \bigr) \dif x \fstop
\end{equation}
For a fixed~$t \ge 0$ and a curve~$s \mapsto f_s$, we have
\[ \frac{\dif}{\dif s} \mathcal F_0 (f_s \dif x) = \int_{\Omega} (V + \log f_s) \frac{\dif}{\dif s} f_s \dif x \comma \]
and, therefore,
\begin{align*}
	\frac{\dif}{\dif s} \mathcal F_0 (f_s \dif x)\Big|_{s=0}
	&=
	-\int_{\Omega} (V + \log \rho_t) \dive( \rho_t \nabla \psi ) \dif x \\
	&=
	\int_{\Omega} \langle (\nabla V + \nabla \log \rho_t), \nabla \psi \rangle \rho_t \dif x - \int_{\partial \Omega} \Psi \rho_t \langle \nabla \psi, \boldsymbol{n} \rangle \dif \mathscr H^{d-1} \comma
\end{align*}
where, in the last identity, we used the boundary conditions in~\eqref{eq:fpStrong}. Let us choose
\[
\boldsymbol v_t \coloneqq -\nabla V - \nabla \log \rho_t \comma
\]
which makes the continuity equation~\eqref{eq:ce} true, since~$\rho_t$ solves~\eqref{eq:fpStrong}. Then,
\[ 
\frac{\dif}{\dif s} \mathcal F_0 (f_s)\Big|_{s=0} = -\int_\Omega \langle \boldsymbol v_t, \nabla \psi \rangle \rho_t \dif x - \int_{\partial \Omega} \Psi \rho_t \langle \nabla \psi, \boldsymbol{n} \rangle \dif \mathscr H^{d-1} \comma
\]
and we see that~$\mathcal F_0$ is not the right functional because of the integral on the boundary. The measure~$\langle \nabla \psi, \boldsymbol n \rangle \rho_t  \mathscr H^{d-1}$ on~$\partial \Omega$ can be seen as the flux of mass (coming from~$f_0=\rho_t$) that is moving away from~$\Omega$ along the flow~$s \mapsto f_s$ at~$s=0$. Thus, if we let this mass settle on the boundary,~$\langle \nabla \psi, \boldsymbol n \rangle \rho_t  \mathscr H^{d-1}$ is the time-derivative of the mass on~$\partial \Omega$. For this reason, it makes sense to consider not just measures on~$\Omega$, but rather on the closure~$\ovom$, and to define
\[
\mathcal F (\mu) \coloneqq \mathcal F_0(\mu|_\Omega) + \int \Psi \dif \mu|_{\partial \Omega} \fstop
\]
Our entropy functional~$\Ent$ is defined precisely like this, and, as we will see in \Cref{sec:preliminaries}, the transportation functionals~$\TrFun$ and~$\TrDis$ are extensions of~$Wb_2$ to the subset~$\cone$ of the signed measures on~$\ovom$, constructed so as to encode the idea that mass can leave~$\Omega$ to settle on~$\partial \Omega$ (and vice versa).

This argument is simple, but let us also emphasize the hidden difficulties:
\begin{itemize}
	\item we assume low regularity on~$\partial \Omega$ and on the functions~$\rho_0$ and~$V$;
	\item the transportation-cost functionals~$\TrDis$ and~$\TrFun$ will not be, in general, distances;
	\item the functional~$\Ent$ is not bounded from below on~$\cone$ (if~$\Psi$ is nonconstant), nor it is strictly convex. Indeed, it is linear along lines of the form~$\R \ni l \mapsto \mu + l \eta$ with~$\mu,\eta \in \cone$ and~$\eta$ concentrated on~$\partial \Omega$;
	\item when~$(\cone,\TrDis)$ is a geodesic metric space, the functional~$\Ent$ is \emph{not} geodesically semiconvex; see~\cite[Remark~3.4]{FigalliGigli10} and \Cref{sec:geod}.
\end{itemize}

\section{Preliminaries} \label{sec:preliminaries}

\subsection{Setting} \label{sec:setting}

Throughout the paper,~$\Omega$ is an open, bounded, and nonempty subset of $\R^d$. Without loss of generality, we assume that~$0 \in \Omega$. No assumption is made on the regularity of its boundary. %

Three functions are given: the initial datum~$\rho_0 \colon \Omega \to \R_+$, the potential~$V \colon \Omega \to \R$, and the function~$\Psi : \ovom \to \R$ that determines the boundary condition. We assume that~$\Psi$ is Lipschitz continuous and that the integral~$\int_{\Omega} \rho_0 \log \rho_0 \dif x$ is finite. In addition, we suppose that~$V$ is bounded (i.e.,~in~$L^\infty(\Omega)$) and in the set of locally Sobolev functions~$W^{1,d+}_\mathrm{loc}(\Omega)$.\footnote{In particular,~$V \in C(\Omega)$.}

\begin{definition} \label{sec:functions}
	We say that~$V \in W^{1,d+}_\mathrm{loc}(\Omega)$ if, for every~$\omega \Subset \Omega$ open, there exists~$p = p(\omega) > d$ such that~$V \in W^{1,p}(\omega)$.
\end{definition}

The set~$\cone$ is the convex cone of all finite and signed Borel measures~$\mu$ on $\ovom$ such that~\eqref{eq:S} holds.

\begin{proposition}
	The set~$\cone$ is closed w.r.t.~the weak convergence, i.e., in duality with continuous and bounded functions on~$\ovom$.
\end{proposition}

\begin{proof}
	If~$\cone \ni \mu^n \to_n \mu$, then
	$ \mu(\ovom) = \lim_{n \to \infty} \mu^n(\ovom)= 0 $
	and, for every~$f \colon \ovom \to \R_+$ continuous and compactly supported in~$\Omega$,
	\[ \int f \dif \mu_\Omega = \int f \dif \mu = \lim_{n \to \infty} \int f \dif \mu^n = \lim_{n \to \infty} \int f \dif \mu^n_\Omega \ge 0 \fstop \]
	The conclusion follows from the Riesz--Markov--Kakutani  theorem.
\end{proof}

The {entropy functionals}~$\E \colon L^1_+(\Omega) \to \R \cup \set{\infty}$ and~$\Ent \colon \cone \to \R \cup \set{\infty}$ are defined in~\eqref{eq:E} and~\eqref{eq:Ent}, respectively.

\subsection{Convention on constants}
The symbol~$\const$ is reserved for strictly positive real constants. The number it represents \emph{may change from formula to formula} and possibly depends on the dimension~$d$, the set~$\Omega$, the functions~$V$ and~$\Psi$, and the initial datum~$\rho_0$. We also allow~$\const$ to depend on other quantities, which are, in case, explicitly displayed as a subscript.

\subsection{Measures} \label{subs:not:measures}
For every signed Borel measure~$\mu$ and Borel set~$A$, we write $\mu_A = \mu|_A$ for the restriction of~$\mu$ to~$A$. Similarly, and following the notation of~\cite{FigalliGigli10,Morales18}, if~$\gamma$ is a measure on a product space and $A,B$ are Borel, we write $\gamma_A^B = \gamma_{A \times B}$ for the restriction of~$\gamma$ to~$A \times B$. We use the notation~$\mu_+, \mu_-$ for the positive and negative parts of a given measure~$\mu$, and ~$\norm{\mu}$ for the total-variation norm of~$\mu$, i.e., the total mass of~$\mu_+ + \mu_-$.

For every Borel function~$f$ and signed Borel measure~$\mu$, we denote by~$\mu(f)$ the integral~$\int f \dif \mu$.

On the set of the finite signed Borel measures on~$\ovom$, we also consider the (modified) Kantorovich--Rubinstein norm (see~\cite[Section~8.10(viii)]{Bogachev07})
\begin{equation}
	\norm{\mu}_{\KR} \coloneqq \abs{\mu(\ovom)} + \sup \set{\mu(f) \, : \, f \colon \ovom \to \R \comma \Lip(f) \le 1 \text{ and } f(0)=0} \fstop
\end{equation}

We write~$F_\# \mu$ for the push-forward of a (signed) Borel measure~$\mu$ via a Borel map~$F$. Often, we use as~$F$ the projection onto some coordinate: we write~$\pi^i$ for the projection on the~$i^\mathrm{th}$ coordinate (or~$\pi^{ij}$ for the projection on the two coordinates~$i$ and~$j$).

We denote by~$\Leb^d$ the $d$-dimensional Lebesgue measure on~$\R^d$. We also use the notation~$\abs{A} \coloneqq \Leb^d(A)$ when~$A \subseteq \R^d$ is a Borel set. We write~$\delta_x$ for the Dirac delta measure at~$x$.

\subsection{Weak solution to the Fokker--Planck equation} \label{sec:defFokker}
We say that a family of nonnegative measures~$(\mu_t)_{t \ge 0}$ on~$\Omega$ is a weak solution to  the Fokker--Planck equation if:
\begin{enumerate}
	\item it is continuous in duality with the space of continuous and compactly supported functions~$C_c(\Omega)$;
	\item for every open set~$\omega \Subset \Omega$, both~$t\mapsto  \mu_t(\omega)$ and~$t \mapsto \int \abs{\nabla V} \dif  \mu_t|_\omega $ belong to~$L^1_\mathrm{loc}\bigl( [0,\infty) \bigr)$, i.e.,~their restrictions to~$(0,\bar t \, )$ are integrable for every~$\bar t > 0$;
	\item for every~$\varphi \in C_c^2(\Omega)$ and~$0 \le s \le t$, the following identity holds:
	\begin{equation} \label{eq:fp} 
		\int \varphi \dif  \mu_t - \int \varphi \dif  \mu_s = \int_s^t \int \bigl( \Delta \varphi - \langle \nabla \varphi, \nabla V \rangle \bigr) \dif  \mu_r \dif r \fstop
	\end{equation}
\end{enumerate}

\subsection{Metric gradient flows} \label{sec:metricGF}
The general theory of gradient flows in metric spaces was developed  in~\cite{AmbrosioGigliSavare08}; we refer to this book and to the survey~\cite{Santambrogio17} for a comprehensive exposition of the topic. We collect here only the definitions we need from this theory.

Let~$(X, \mathsf d)$ be a metric space, let~$[0,\infty) \ni t \mapsto x_t$ be an~$X$-valued map, and let~$f \colon X \to \R \cup \set{\infty}$ be a function.

\begin{definition}[{Metric derivative~\cite[Theorem~1.1.2]{AmbrosioGigliSavare08}}]
	We say that~$(x_t)_{t \in [0,\infty) }$ is~\emph{locally absolutely continuous} if there exists a function~$m \in L^1_\mathrm{loc}\bigl([0,\infty)\bigr)$ such that
	\begin{equation} \label{eq:metricDer}
		\mathsf d(x_s,x_t) \le \int_s^t m(r) \dif r
	\end{equation}
	for every~$0 \le s < t$. If~$(x_t)_{t \in [0,\infty)}$ is locally absolutely continuous, for~$\Leb^1_{[0,\infty)}$-a.e.~$t$ there exists the limit
	\begin{equation}
		\abs{\dot x_t } \coloneqq \lim_{s \to t} \frac{\mathsf{d}(x_s,x_t)}{\abs{s-t}} \comma
	\end{equation}
	and this function, called~\emph{metric derivative}, is the~$\Leb^1_{[0,\infty)}$-a.e. minimal function~$m$ that satisfies \eqref{eq:metricDer}; see~\cite[Theorem 1.1.2]{AmbrosioGigliSavare08}.\footnote{In~\cite[Theorem 1.1.2]{AmbrosioGigliSavare08}, the completeness of the space is assumed but not necessary, as can be easily checked.}
\end{definition}

\begin{definition}[{Descending slope~\cite[Definition~1.2.4]{AmbrosioGigliSavare08}}]
	The \emph{descending slope} of~$f$ at~$x \in X$ is the number
	\begin{equation}
		\slopesmall{f}{}(x) = \slopesmall{f}{\mathsf d}(x) \coloneqq \limsup_{y \stackrel{\mathsf d}{\to} x} \frac{\bigl(f(x)-f(y)\bigr)_+}{\mathsf d(x,y)} \comma
	\end{equation}
	where~$a_+ \coloneqq \max\set{0,a}$ is the positive part of~$a \in \R \cup \set{\pm \infty}$. The slope is conventionally set equal to~$\infty$ if~$f(x) = \infty$, and to~$0$ if~$x$ is isolated and~$f(x) < \infty$.
\end{definition}

\begin{definition}[{Curve of maximal slope~\cite[Definition~1.3.2]{AmbrosioGigliSavare08}}] \label{def:maxSlope}
	We say that a locally absolutely continuous~$X$-valued map~$(x_t)_{t \in [0,\infty)}$ is a \emph{curve of maximal slope} (with respect to~$\slope{f}{\mathsf d}$) if~$t \mapsto f (x_t)$ is a.e.~equal to a nonincreasing map~$\phi \colon [0,\infty) \to \R$ such that
	\begin{equation} \label{eq:ediGeneral}
		\dot \phi(t) \le - \frac{1}{2} \abs{\dot x_t}^2 - \frac{1}{2} \slopesmall{f}{\mathsf d}^2(x_t) \quad \text{for } \Leb^1_{[0,\infty)} \text{-a.e.~} t \fstop
	\end{equation}
\end{definition}

\Cref{def:maxSlope} is motivated by the observation that, when~$(X,\mathsf d)$ is a Euclidean space and~$f$ is smooth, the inequality~\eqref{eq:ediGeneral} is equivalent to the gradient-flow equation
\[
\frac{\dif}{\dif t} x_t = -\nabla f(x_t) \comma \qquad t \ge 0 \comma
\]
see for instance~\cite[Section 2.2]{Santambrogio17}. As noted in~\cite[Remark~1.3.3]{AmbrosioGigliSavare08},\footnote{Once again, completeness is not necessary.} even in the general metric setting,~\eqref{eq:ediGeneral} actually implies the identities
\[
-\dot \phi(t) = \abs{\dot  x_t}^2 = \slopesmall{f}{\mathsf d}^2(x_t) \quad \text{for a.e.~$t \ge 0$} \fstop
\]

\subsection{The Figalli--Gigli distance} \label{sec:figalli-gigli}
We briefly recall the definition and some properties of the distance $Wb_2$ introduced in~\cite{FigalliGigli10}.
We denote by~$\mathcal M_2(\Omega)$ the set of nonnegative Borel measures~$ \mu$ on~$\Omega$ such that
\begin{equation}
	\int \inf_{y \in \partial \Omega} \abs{x-y}^2 \dif  \mu(x) < \infty \comma
\end{equation}
and, for every nonnegative Borel measure~$\gamma$ on~$\ovom \times \ovom$, define the cost functional
\begin{equation}
	\cost(\gamma) \coloneqq \int \abs{x-y}^2 \dif \gamma(x,y) \fstop
\end{equation}

\begin{definition}[{\cite[Problem 1.1]{FigalliGigli10}}]
	Let~$ \mu,  \nu \in \mathcal M_2(\Omega)$. We say that a nonnegative Borel measure~$\gamma$ on $\ovom \times \ovom$ is a \emph{$Wb_2$-admissible transport plan} between~$ \mu$ and~$ \nu$, and write~$\gamma \in \Adm_{Wb_2}( \mu,  \nu)$, if
	\begin{equation} \label{eq:projGF} \bigl( \pi^1_\# \gamma \bigr)_\Omega =  \mu \quad \text{and} \quad \bigl( \pi^2_\# \gamma \bigr)_\Omega =  \nu \fstop \end{equation}
	The distance~$Wb_2(\mu,\nu)$ is then defined as
	\begin{equation} \label{eq:defGF} Wb_2( \mu,  \nu) \coloneqq \inf\set{\sqrt{\cost(\gamma)} \, : \, \gamma \in \Adm_{Wb_2}( \mu,  \nu)} \fstop \end{equation}
\end{definition}

In~\cite[Section~2]{FigalliGigli10}, it was observed that for every~$ \mu,  \nu \in \mathcal M_2(\Omega)$ there exists at least one $Wb_2$-optimal transport plan, that is, a measure~$\gamma \in \Adm_{Wb_2}( \mu,  \nu)$ that attains the infimum in~\eqref{eq:defGF}.

Later, we will make use of the following consequences of~\cite[Proposition 2.7]{FigalliGigli10}: the convergence w.r.t.~the metric~$Wb_2$ implies the convergence in duality with~$C_c(\Omega)$, and it is implied by the convergence in duality with~$C_b(\Omega)$.

\subsection{Transportation functionals} \label{sec:functionals}

We now define the transportation functionals~$\TrFun$ and~$\TrDis$ that appear in the schemes~\eqref{eq:jkoTrFun0} and~\eqref{eq:jkoTrDis0}.

\begin{definition} \label{def:TrDis}
	For every~$\mu,\nu \in \cone$, let~$\Adm_{\TrDis}(\mu,\nu)$ be the set of all finite nonnegative Borel measures~$\gamma$ on $\ovom \times \ovom$ such that
	\begin{enumerate}[(1)]
		\item \label{(1)}$\bigl(\pi^1_\# \gamma \bigr)_\Omega = \mu_\Omega$,
		\item \label{(2)} $\bigl(\pi^2_\# \gamma \bigr)_\Omega = \nu_\Omega$,
		\item \label{(3)}$\pi^1_\# \gamma - \pi^2_\# \gamma = \mu - \nu$.
	\end{enumerate} 
	We call such measures \emph{$\TrDis$-admissible transport plans} between~$\mu$ and~$\nu$.
	We set
	\begin{equation} \label{eq:TrDis1}
		\TrDis(\mu,\nu) \coloneqq \inf\set{\sqrt{\cost(\gamma)} \, : \, \gamma \in \Adm_{\TrDis}(\mu,\nu)} \comma
	\end{equation}
	and write
	\begin{equation} \label{eq:TrDis2}
		\Opt_\TrDis(\mu,\nu) \coloneqq \argmin_{\gamma \in \Adm_{\TrDis}(\mu,\nu)} \cost(\gamma)
	\end{equation}
	for the set of all \emph{$\TrDis$-optimal tranport plans} between~$\mu$ and~$\nu$.
\end{definition}

\begin{remark}
	There is some redundancy in the properties~\ref{(1)}-\ref{(3)}, indeed,
	\[
	\ref{(1)} + \ref{(3)} \Rightarrow \ref{(2)} \quad \text{and} \quad \ref{(2)} + \ref{(3)} \Rightarrow \ref{(1)} \fstop
	\]
\end{remark}

\begin{definition} \label{def:TrFun}
	For every~$\mu,\nu \in \cone$, let~$\Adm_{\TrFun}(\mu,\nu)$ be the set of all measures~$\gamma \in \Adm_{\TrDis}(\mu,\nu)$ such that, additionally,
	\begin{enumerate}[(4)]
		\item \label{(4)} $\gamma_{\partial \Omega}^{\partial \Omega} = 0$.
	\end{enumerate}
	We define the $\TrFun$-admissible/optimal tranport plans as in~\eqref{eq:TrDis1} and~\eqref{eq:TrDis2}, by replacing~$\TrDis$ with~$\TrFun$.
\end{definition}

\begin{remark} \label{rmk:boundGamma}
	If~$\gamma \in \Adm_{\TrFun}(\mu, \nu)$ for some~$\mu, \nu \in \cone$, then
	\begin{equation} \label{eq:boundGamma}
		\norm{\gamma} \le \norm{\gamma_\Omega^\ovom} + \norm{\gamma_\ovom^\Omega} = \norm{\mu_\Omega} + \norm{\nu_\Omega} \fstop
	\end{equation}
\end{remark}

\begin{figure}
	\centering
	\subfloat[][$Wb_2$-admissible]
	{\includegraphics[width=.3\textwidth, trim={17mm 26mm 65mm 16mm},clip]{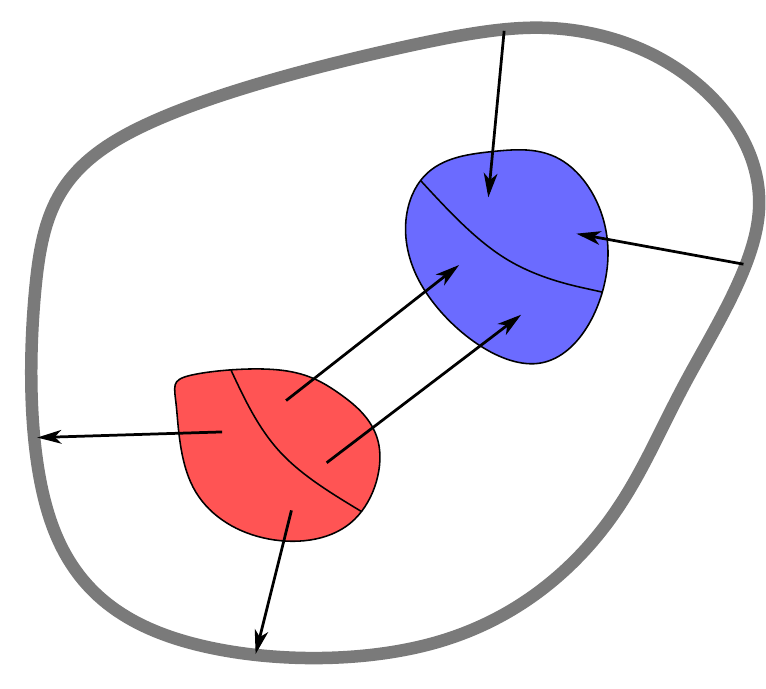}}
	\hspace{10pt}
	\subfloat[][$\TrDis$-admissible]
	{\includegraphics[width=.3\textwidth, trim={17mm 26mm 65mm 16mm},clip]{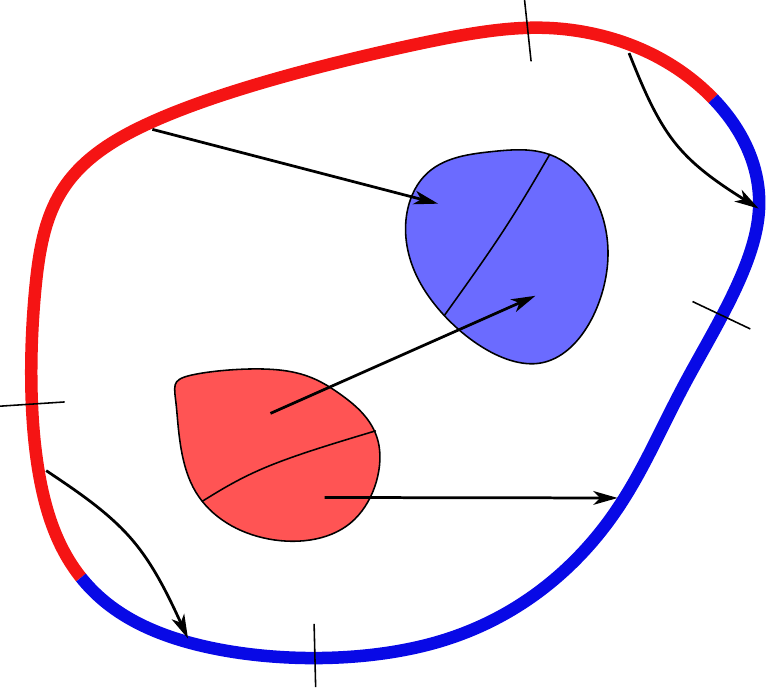}} \hspace{10pt}
	\subfloat[][$\TrFun$-admissible]
	{\includegraphics[width=.3\textwidth, trim={17mm 26mm 65mm 16mm},clip]{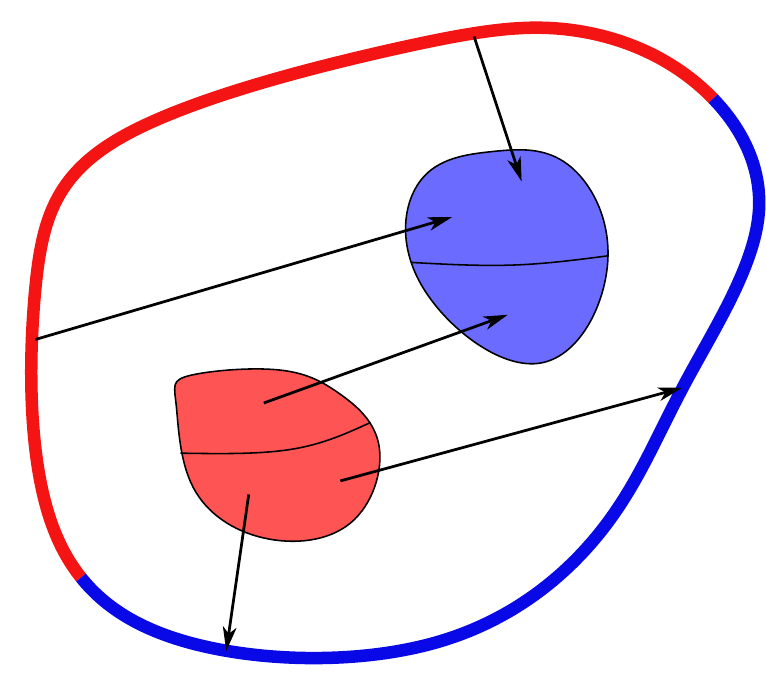}}
	\caption{Examples of admisssible plans. Red (resp.~blue) regions are those with an abundance of initial (resp.~final) mass~$\mu$ (resp.~$\nu$). Admissible plans for~$Wb_2$ do not have any restriction on the mass departing from and arriving to~$\partial \Omega$. Admissible plans for~$\TrDis$ must agree---in the sense of Condition~\ref{(3)}---with the configurations~$\mu,\nu$ also on~$\partial \Omega$. Admissible plans for~$\TrFun$ are~$\TrDis$-admissible and, additionally, do not move mass from~$\partial \Omega$ to~$\partial \Omega$.}
	\label{fig:plans}
\end{figure}

\begin{remark} \label{rmk:changeBoundary}
	Fix~$\mu, \nu \in \cone$. For every~$\eta \in \cone$ concentrated on~$\partial \Omega$, it is easy to check that
	\[
	\Adm_{\TrDis}(\mu+\eta, \nu+\eta) = \Adm_{\TrDis}(\mu, \nu) \quad \text{and} \quad \Adm_{\TrFun}(\mu+\eta, \nu+\eta) = \Adm_{\TrFun}(\mu, \nu) \fstop
	\]
	Hence,
	\begin{equation}
		\TrDis(\mu+\eta,\nu+\eta) = \TrDis(\mu,\nu) \quad \text{and} \quad \TrFun(\mu+\eta,\nu+\eta) = \TrFun(\mu,\nu) \fstop
	\end{equation}
\end{remark}

Let us briefly comment on these definitions. Conditions~\ref{(1)} and~\ref{(2)} are precisely the same as~\eqref{eq:projGF}. They are needed to ensure that the mass that departs from (resp.~arrives in) $\Omega$ is precisely $\mu_\Omega$ (resp.~$\nu_\Omega$). Condition~\ref{(3)} is needed to also keep track of the mass that is exchanged with the boundary. Namely, it ensures that, on each subregion of~$A \subseteq \overline \Omega$ (possibly including part of the boundary, which was neglected by Conditions~\ref{(2)}-\ref{(3)}), the mass~$\nu(A)$ after the transportation equals the initial mass~$\mu(A)$, plus the imported mass~$\gamma(\overline \Omega \times A)$, minus the exported mass~$\gamma(A \times \overline \Omega)$. Observe that, since~$\mu$ and~$\nu$ normally have a negative mass on some subregions of $\partial \Omega$, it does not make sense to naively impose $\pi^1_\# \gamma = \mu$ and $\pi^2_\# \gamma = \nu$.

The difference between~$\TrDis$ and~$\TrFun$ is Condition~\ref{(4)}:~$\TrFun$-admissible transport plans cannot move mass from~$\partial \Omega$ to~$\partial \Omega$. This results in the loss of the triangle inequality.

\begin{example} \label{ex:notri}
	Consider, for the domain~$\Omega \coloneqq (0,1)$, the measures
	\[ \mu_1 \coloneqq \delta_0 - \delta_1 \in \cone \comma \quad \mu_2 \coloneqq \delta_{1/2} - \delta_1 \in \cone \comma \quad \mu_3 \coloneqq 0 \in \cone \fstop \]
	The transport plans~$\gamma_{12} \coloneqq \delta_{( 0,1/2 ) }$ and~$\gamma_{23} \coloneqq \delta_{(1/2,1)}$ are $\TrFun$-admissible, between~$\mu_1$ and~$\mu_2$, and between~$\mu_2$ and~$\mu_3$, respectively.
	Thus, both~$\TrFun(\mu_1,\mu_2)$ and~$\TrFun(\mu_2,\mu_3)$ are bounded above by~$1/2$. However, there is no~$\gamma_{13} \in \Adm_{\TrFun}(\mu_1,\mu_3)$, whence $\TrFun(\mu_1,\mu_3) = \infty$. Indeed, Conditions~\ref{(1)} and~\ref{(2)} in \Cref{def:TrDis} would imply~$(\gamma_{13})_\Omega^{\ovom} = (\gamma_{13})_\ovom^\Omega = 0$. Together with~\ref{(4)} in \Cref{def:TrFun}, this means that~$\gamma_{13}$ equals the zero measure, which contradicts~\ref{(3)} in \Cref{def:TrDis}.
\end{example}

Nonetheless, it is shown in~\Cref{prop:informloss} that Condition~\ref{(4)} is needed in dimension~$d \ge 2$, because the information about $\mu_{\partial \Omega}$ and $\nu_{\partial \Omega}$ may otherwise be lost. This does not happen when~$\Omega$ is just a finite union of intervals in~$\R^1$, because points in~$\partial \Omega$ are distant from each other. We will see that, in this case, \Cref{def:TrDis} defines a distance.

These remarks reveal part of the difficulties in building cost functionals for signed measures that behave like~$W_2$. See~\cite{Mainini12} for further details. However, it seems at least convenient to use signed measures, given that a modified JKO scheme that mimics~\cite{FigalliGigli10} should allow for a virtually unlimited amount of mass to be taken from points of~$\partial \Omega$, step after step.

\section{Properties of the transportation functionals} \label{sec:properties}

We gather some useful properties of~$\TrFun$ and~$\TrDis$.

\subsection{Relation with the Figalli--Gigli distance}
For every~$\mu, \nu \in \cone$, we have the inclusions
\begin{equation*}
	\Adm_{\TrFun}(\mu,\nu) \subseteq \Adm_{\TrDis}(\mu,\nu) \subseteq \Adm_{Wb_2}(\mu_\Omega,\nu_\Omega) \fstop
\end{equation*}
As a consequence,
\begin{equation} \label{eq:distIneq}
	Wb_2(\mu_\Omega, \nu_\Omega) \le \TrDis(\mu,\nu) \le \TrFun(\mu,\nu) \comma \qquad \mu, \nu \in \cone \fstop
\end{equation}

In fact,~$\TrDis$ and~$\TrFun$ can be seen as extensions of~$Wb_2$ in the following sense.

\begin{lemma} \label{lemma:proj}
	Let~$\mu, \nu$ be finite nonnegative Borel measures on~$\Omega$. For every~$\tilde \mu \in \cone$ with~$\tilde \mu_\Omega = \mu$, we have the identities
	\begin{align}
		Wb_2( \mu,  \nu) &= \inf_{\tilde \nu \in \cone} \set{ \TrDis(\tilde \mu, \tilde \nu) \, : \, \tilde \nu_\Omega = \nu } = \inf_{\tilde \nu \in \cone} \set{ \TrFun(\tilde \mu, \tilde \nu) \, : \,  \tilde \nu_\Omega = \nu }  \fstop
	\end{align}
\end{lemma}

\begin{proof}
	In light of~\eqref{eq:distIneq}, it suffices to prove that
	\[
	\inf_{\tilde \nu \in \cone} \set{ \TrFun(\tilde \mu, \tilde \nu) \, : \, \tilde \nu_\Omega = \nu } \le Wb_2(\mu, \nu) \fstop
	\]
	Let~$\gamma \in \Adm_{Wb_2}(\mu, \nu)$. Define~$\tilde \gamma \coloneqq \gamma - \gamma_{\partial \Omega}^{\partial \Omega}$ and
	\[
	\tilde \nu \coloneqq \tilde \mu + \pi^2_\# \tilde \gamma - \pi^1_\# \tilde \gamma \fstop
	\]
	It is easy to check that~$\tilde \nu_\Omega = \nu$, that~$\tilde \gamma \in \Adm_{\TrFun}(\tilde \mu, \tilde \nu)$, and that~$\cost(\tilde \gamma) \le \cost(\gamma)$. As a consequence,
	\[
	\inf_{\tilde \nu \in \cone}\set{\TrFun(\tilde \mu, \tilde \nu) \, : \, \tilde \nu_\Omega = \nu} \le \sqrt{\cost(\gamma)} \comma
	\]
	and we conclude by arbitrariness of~$\gamma$.
\end{proof}

\subsection{Relation with the Kantorovich--Rubinstein norm}
Interestingly, an inequality relates~$\TrDis$ and~$\norm{\cdot}_\KR$.
\begin{lemma} \label{lemma:TrDisKR}
	For every~$\mu, \nu \in \cone$, we have
	\begin{equation} \label{eq:TrDisKR}
		\TrDis^2(\mu, \nu) \le \diam(\Omega) \norm{\mu - \nu}_\KR \fstop
	\end{equation}
\end{lemma}
\begin{proof}
	Define the nonnegative measures
	\[
	\tilde \mu \coloneqq \mu_\Omega + (\mu_{\partial \Omega}-\nu_{\partial \Omega})_{+} \comma \quad \tilde \nu \coloneqq \nu_\Omega + (\mu_{\partial \Omega}-\nu_{\partial \Omega})_{-} \comma
	\]
	and note that~$\tilde\mu-\tilde\nu = \mu - \nu$. In particular,~$\tilde \mu(\ovom) = \tilde \nu(\ovom)$.
	
	Let~$\gamma$ be a coupling between~$\tilde \mu$ and~$\tilde \nu$, i.e.,~$\gamma$ is a nonnegative Borel measure on~$\ovom \times \ovom$ such that~$\pi^1_\# \gamma = \tilde \mu$ and~$\pi^2_\# \gamma = \tilde \nu$. Notice that~$\gamma$ is $\TrDis$-admissible between~$\mu$ and~$\nu$. Therefore,
	\[
	\TrDis^2(\mu, \nu) \le \cost(\gamma) =  \int \abs{x-y}^2 \dif \gamma \le \diam(\Omega) \int \abs{x-y} \dif \gamma \fstop
	\]
	After taking the infimum over~$\gamma$, the Kantorovich--Rubinstein duality~\cite[Theorem~8.10.45]{Bogachev07} implies
	\[
	\TrDis^2(\mu, \nu) \le \diam(\Omega) \norm{\tilde\mu-\tilde\nu}_\KR = \diam(\Omega) \norm{\mu-\nu}_\KR \fstop \qedhere
	\]
\end{proof}

\subsection{$\TrFun$ is an extended semimetric}
The functional~$\TrFun$ may take the value infinity and does not satisfy the triangle inequality; see \Cref{ex:notri}. Nonetheless, we have the following proposition, which we prove together with two useful lemmas.

\begin{proposition} \label{prop:semimetric}
	The functional~$\TrFun$ is an extended semimetric, i.e., it is nonnegative, symmetric, and we have
	\begin{equation} \label{eq:distancePropTrFun} \TrFun(\mu,\nu) = 0 \quad \Longleftrightarrow \quad \mu = \nu \fstop \end{equation}
\end{proposition}

\begin{lemma} \label{lemma:admTrFunclosed}
	Let~$(\mu^n)_{n \in {\N_0}}$ and~$(\nu^n)_{n \in {\N_0}}$ be two sequences in~$\cone$, and let~$\gamma^n \in \Adm_\TrFun(\mu^n,\nu^n)$ for every~$n \in {\N_0}$. Assume that
	\begin{enumerate}[(a)]
		\item $\mu^n \to_n \mu$ and~$\nu^n \to_n \nu$ weakly for some~$\mu,\nu$,
		\item $\mu^n_\Omega \to_n \mu_\Omega$ and~$\nu^n_\Omega \to_n \nu_\Omega$ \emph{setwise}, i.e., on all Borel sets,
		\item $\gamma^n \to_n \gamma$ weakly.
	\end{enumerate}
	Then~$\mu,\nu \in \cone$ and $\gamma \in \Adm_\TrFun(\mu,\nu)$.
	
	In particular, for any~$\mu,\nu \in \cone$, the set~$\Adm_\TrFun(\mu,\nu)$ is sequentially closed with respect to the weak convergence.
\end{lemma}

The proof of this lemma is inspired by part of that of \cite[Lemma~3.1]{Morales18}.

\begin{proof}
	The total mass of~$\gamma^n$ is bounded and, therefore, the same can be said for the total mass of~$(\gamma^n)_\Omega^\Omega, (\gamma^n)_{\Omega}^{\partial \Omega}, (\gamma^n)_{\partial \Omega}^\Omega$. Hence, up to taking a subsequence, we may assume that
	\begin{align*} &(\gamma^n)_{\Omega}^{\Omega} \to_n \sigma_1 \quad \text{in duality with } C(\ovom \times \ovom) \comma \\
		&(\gamma^n)_{\Omega}^{\partial \Omega} \to_n \sigma_2 \quad \text{in duality with } C(\ovom \times \partial \Omega) \comma \\
		&(\gamma^n)_{\partial \Omega}^{\Omega} \to_n \sigma_3 \quad \text{in duality with } C(\partial \Omega \times \ovom)
	\end{align*}
	for some~$\sigma_1,\sigma_2,\sigma_3$. In particular,~$\gamma^n \to_n \gamma \coloneqq \sigma_1 + \sigma_2 + \sigma_3$.
	
	We \emph{claim} that~$\sigma_1,\sigma_2, \sigma_3$ are concentrated on~$\Omega \times \Omega, \Omega \times \partial \Omega, \partial \Omega \times \Omega$ respectively. If this is true, then Condition~\ref{(4)} in \Cref{def:TrFun} for~$\gamma$ is obvious, and those in \Cref{def:TrDis} follow by testing them with a function $f \in C_b(\overline \Omega)$ for every $n$ and passing to the limit. For instance, to prove~Condition~\ref{(1)} in \Cref{def:TrDis}, we write the chain of equalities
	\begin{align*} \mu_\Omega(f) &= \lim_{n \to \infty} \mu_\Omega^n(f) = \lim_{n \to \infty} \int f(x) \dif \, (\gamma^n)_{\Omega}^{\ovom}(x,y) \\ &= \int f(x) \dif \, (\sigma_1+\sigma_2)(x,y)  = \int f(x) \dif \gamma_\Omega^{\ovom}(x,y) = \bigl(\pi^1_{\#} \gamma_{\Omega}^{\ovom}\bigr)(f) \fstop \end{align*}
	
	Let us prove the claim. Let~$A \subseteq \ovom$ be an open set, in the relative topology of~$\ovom$, that contains~$\partial \Omega$. We have
	\begin{align*} \sigma_1(\partial \Omega \times \overline \Omega) &\le \sigma_1(A \times \overline \Omega) \le \liminf_{n \to \infty} (\gamma^n)_{\Omega}^{\Omega}(A \times \overline \Omega)\\
		&\le \liminf_{n \to \infty} (\gamma^n)_{\Omega}^{\overline \Omega}(A \times \overline \Omega) = \liminf_{n \to \infty} \mu_\Omega^n(A) =  \mu_{\Omega}(A) \comma \end{align*}
	where the second inequality follows from the semicontinuity of the mass on open sets (in the topology of~$\ovom \times \ovom$) and the last equality from the setwise convergence.
	Since~$\mu_{\Omega}$ has finite total mass and~$\mu_{\Omega}(\partial \Omega) = 0$, we have~$\sigma_1(\partial \Omega \times \ovom) = 0$. Analogously, using Condition~\ref{(2)} in place of Condition~\ref{(1)}, we obtain~$\sigma_1(\ovom \times \partial \Omega) = 0$. For~$\sigma_2$ and~$\sigma_3$, the proof is similar.
\end{proof}

\begin{lemma} \label{lemma:optTrFun}
	If~$\TrFun(\mu,\nu) < \infty$, then~$\Opt_\TrFun(\mu,\nu) \neq \emptyset$.
\end{lemma}

\begin{proof}
	It suffices to prove that~$\Adm_{\TrFun}(\mu,\nu)$ is nonempty and weakly sequentially compact. It is nonempty if~$\TrFun(\mu,\nu) < \infty$. It is sequentially compact because
	\[ \gamma \in \Adm_{\TrFun}(\mu,\nu) \quad \stackrel{\eqref{eq:boundGamma}}{\Longrightarrow} \quad \norm{\gamma} \le \norm{\mu_\Omega} + \norm{\nu_\Omega} \comma \]
	and thanks to \Cref{lemma:admTrFunclosed}.
\end{proof}

\begin{proof}[Proof of \Cref{prop:semimetric}]
	Only the implication~$\Rightarrow$ in~\eqref{eq:distancePropTrFun} is not immediate. Let us assume that~$\TrFun(\mu,\nu) = 0$ and let~$\gamma \in \Opt_\TrFun(\mu,\nu)$. Since~$\cost(\gamma) = 0$,
	the measure~$\gamma$ is concentrated on the diagonal of~$\ovom \times \ovom$. Thus, the equality~$\mu = \nu$ follows from Condition~\ref{(3)} in \Cref{def:TrDis}.
\end{proof}

We conclude with a corollary of \Cref{lemma:admTrFunclosed}: a semicontinuity property of~$\TrFun$.

\begin{corollary} \label{cor:sciT}
	Let~$(\mu^n)_{n \in {\N_0}}$ and~$(\nu^n)_{n \in {\N_0}}$ be two sequences in~$\cone$. Assume that
	\begin{enumerate}[(a)]
		\item $\mu^n \to_n \mu$ and~$\nu^n \to_n \nu$ weakly for some~$\mu,\nu$,
		\item $\mu^n_\Omega \to_n \mu_\Omega$ and~$\nu^n_\Omega \to_n \nu_\Omega$ \emph{setwise}, i.e., on all Borel sets.
	\end{enumerate}
	Then
	\begin{equation} \label{eq:sciT}
		\TrFun(\mu,\nu) \le \liminf_{n \to \infty} \TrFun(\mu^n, \nu^n) \fstop
	\end{equation}
\end{corollary}

\begin{proof}
	We may assume that the right-hand side in~\eqref{eq:sciT} exists as a finite limit and that, for every~$n \in \N_0$, there exists~$\gamma^n \in \Adm_{\TrFun}(\mu,\nu)$ such that
	\[ \cost(\gamma^n) \le \TrFun^2(\mu^n,\nu^n) + \frac{1}{n} \fstop \]
	The total variation of each measure~$\gamma^n$ is bounded by $\norm{\mu^n_\Omega} + \norm{\nu^n_\Omega}$, which is in turn bounded thanks to the assumption. Therefore, we can extract a subsequence~$(\gamma^{n_k})_{k \in \N_0}$ that converges weakly to a measure~$\gamma$. We know from \Cref{lemma:admTrFunclosed} that~$\gamma \in \Adm_\TrFun(\mu,\nu)$; thus,
	\[ \TrFun^2(\mu,\nu) \le \cost(\gamma) = \lim_{k \to \infty} \cost(\gamma^{n_k}) = \lim_{k \to \infty} \TrFun^2(\mu^{n_k},\nu^{n_k}) = \lim_{n \to \infty} \TrFun^2(\mu^n,\nu^n) \fstop \qedhere \]
\end{proof}

\subsection{$\Ent$ is ``semicontinuous w.r.t~$\TrFun$''} \label{sec:semicont1}
Albeit not being a distance, the transportation functional~$\TrFun$ makes~$\Ent$ lower semicontinuous, in the following sense.

\begin{proposition} \label{prop:lscI}
	Let~$(\mu^n)_{n \in \N_0}$ be a sequence in~$\cone$ and suppose that
	\begin{equation} \lim_{n \to \infty} \TrFun(\mu^n, \mu) = 0 \end{equation}
	for some~$\mu \in \cone$. Then
	\begin{equation} \label{eq:lscI} \Ent(\mu) \le \liminf_{n \to \infty} \Ent(\mu^n) \fstop \end{equation}
\end{proposition}

For the proof we need a lemma, to which we will also often refer later. This lemma, inspired by~\cite[Lemma 5.8]{Morales18} allows to control $(\mu-\nu)_{\partial \Omega}$ in terms of~$\TrFun(\mu,\nu)$ and of the restrictions~$\mu_\Omega$ and~$\nu_\Omega$. This fact is convenient for two reasons:
\begin{itemize}
	\item the part of the functional~$\Ent$ that depends on~$\mu_\Omega$ is superlinear,
	\item we will see (\Cref{rmk:massbound}) that \emph{the restrictions to $\Omega$} of the measures produced by the scheme~\eqref{eq:jkoTrFun0} have bounded (in time) mass.
\end{itemize}

\begin{lemma} \label{lemma:boundaryVSinterior}
	Let~$\tau > 0$, let~$\mu,\nu \in \cone$, and let~$\Phi \colon \ovom \to \R$ be Lipschitz continuous. Then,
	\begin{equation}
		\abs{\mu(\Phi) - \nu(\Phi)} \le \tau (\Lip \Phi)^2 \bigl( \norm{\mu_\Omega} + \norm{\nu_\Omega} \bigr) + \frac{\TrFun^2(\mu,\nu)}{4\tau} \fstop
	\end{equation}
	In particular,
	\begin{equation}
		\mu_{\partial \Omega}(\Phi) - \nu_{\partial \Omega}(\Phi) \le \nu_\Omega(\Phi) - \mu_\Omega(\Phi) + \tau (\Lip \Phi)^2  \bigl( \norm{\mu_\Omega} + \norm{\nu_\Omega} \bigr) + \frac{\TrFun^2(\mu,\nu)}{4\tau} \fstop
	\end{equation}
\end{lemma}

\begin{proof}
	Let~$\gamma \in \Opt_\TrFun(\mu,\nu)$. By \Cref{def:TrDis} and \Cref{def:TrFun}, we have
	\begin{align*}
		\abs{\mu(\Phi) - \nu(\Phi)} &= \abs{(\pi^1_\# \gamma - \pi^2_\# \gamma) (\Phi)} = \abs{\int \bigl( \Phi(x)-\Phi(y) \bigr) \dif \gamma(x,y)} \\
		&\le \int \sqrt{2\tau}(\Lip \Phi) \cdot \frac{\abs{x-y}}{\sqrt{2\tau}} \dif \gamma(x,y) \\
		&\le \tau (\Lip\Phi)^2 \norm{\gamma} + \frac{1}{4\tau} \int \abs{x-y}^2 \dif \gamma(x,y) \\
		&\le \tau (\Lip \Phi)^2 \bigl( \norm{\mu_\Omega} + \norm{\nu_\Omega} \bigr) + \frac{\TrFun^2(\mu,\nu)}{4\tau} \fstop \qedhere
	\end{align*}
\end{proof}

\begin{proof}[Proof of \Cref{prop:lscI}]
	We may assume that the right-hand side in~\eqref{eq:lscI} exists as a finite limit and that~$\Ent(\mu^n)$ is finite for every~$n$. In particular,~$\mu^n_\Omega$ is absolutely continuous w.r.t.~$\Leb^d_\Omega$. Denote by~$\rho^n$ its density. Owing to \Cref{lemma:boundaryVSinterior}, for every~$\tau > 0$ and~$n$, we have
	\begin{align*}
		\Ent(\mu^n) &= \E(\rho^n) + \mu^n_{\partial \Omega}(\Psi) \\ 
		&\ge \int_{\Omega} (\log \rho^n + V-1-\const \tau -\Psi)\rho^n \dif x + \abs{\Omega} + \mu(\Psi) - \const \tau \norm{\mu_\Omega} - \frac{\TrFun^2(\mu^n, \mu)}{4\tau} \fstop
	\end{align*}
	It follows that the sequence~$(\rho^n)_{n}$ is uniformly integrable. By the Dunford--Pettis theorem, it admits a (not relabeled) subsequence that converges, weakly in~$L^1(\Omega)$, to some function~$\rho$. From~\eqref{eq:distIneq} and \cite[Proposition 2.7]{FigalliGigli10}, we infer that~$\mu^n_\Omega \to \mu_\Omega$ in duality with~$C_c(\Omega)$ and, therefore,~$\rho$ is precisely the density of~$\mu_\Omega$. The functional~$\E$ is convex and lower semicontinuous on~$L^1(\Omega)$ (by Fatou's lemma), hence weakly lower semicontinuous. Thus, we are only left with proving that
	\[ \mu_{\partial \Omega}(\Psi) \le \liminf_{n \to \infty} \mu^n_{\partial \Omega}(\Psi) \fstop \]
	Once again, we make use of \Cref{lemma:boundaryVSinterior} and of the weak convergence in~$L^1(\Omega)$ to write, for every~$\tau > 0$,
	\[
	\limsup_{n \to \infty} (\mu-\mu^n)_{\partial \Omega}(\Psi) \le \limsup_{n \to \infty} \const \tau \bigl(\norm{\mu^n_\Omega} + \norm{\mu_\Omega} \bigr) + \limsup_{n \to \infty} \frac{\TrFun^2(\mu^n,\mu)}{4\tau} \le \const \tau \norm{\mu_\Omega} \fstop
	\]
	We conclude by arbitrariness of~$\tau$.
\end{proof}

\subsection{$\TrDis$ is a pseudodistance}
The functional~$\TrDis$ is a pseudodistance on~$\cone$, meaning that it fulfills the properties of a distance, except, possibly,~$\mu = \nu$ when $\TrDis(\mu,\nu) = 0$. As before, nonnegativity, symmetry, and the implication \[ \mu = \nu \quad \Longrightarrow \quad  \TrDis(\mu,\nu) = 0\] are obvious. To prove finiteness, it suffices to produce a single~$\gamma \in \Adm_{\TrDis}(\mu,\nu)$ for every~$\mu,\nu \in \cone$. Let us arbitrarily fix a probability measure~$\zeta$ on~$\partial \Omega$ and set
\[ \eta \coloneqq \mu_{\partial\Omega}-\nu_{\partial \Omega} + \bigl(\norm{\mu_\Omega}-\norm{\nu_\Omega}\bigr)\zeta \fstop \] The following is $\TrDis$-admissible:
\[ \gamma \coloneqq \begin{cases}
	\mu_\Omega \otimes \zeta + \zeta \otimes \nu_\Omega + \frac{\eta_+ \otimes \eta_-}{\norm{\eta_+}} &\text{if } \eta \neq 0 \comma \\
	\mu_\Omega \otimes \zeta + \zeta \otimes \nu_\Omega &\text{if } \eta = 0 \fstop
\end{cases}
\]

Only the triangle inequality is still missing.

\begin{proposition}
	The functional~$\TrDis$ satisfies the triangle inequality. Hence, it is a pseudodistance.
\end{proposition}

\begin{proof}
	Let~$\mu_1,\mu_2,\mu_3 \in \cone$, and let us view them as measures on three different copies of~$\ovom$, that we denote by~$\ovom_1,\ovom_2,\ovom_3$, respectively. We write~$\pi^2$ for both the projections from~$\ovom_1 \times \ovom_2$ and~$\ovom_2 \times \ovom_3$ onto~$\ovom_2$.
	
	Choose two transport plans~$\gamma_{12} \in \Adm_\TrDis(\mu_1,\mu_2)$ and~$\gamma_{23} \in \Adm_\TrDis(\mu_2,\mu_3)$. Let~$\eta \coloneqq (\pi^2_\# \gamma_{23} - \pi^2_\# \gamma_{12})_{\partial \Omega}$ and consider
	\[ \tilde \gamma_{12} \coloneqq \gamma_{12} + (\Id,\Id)_\# \eta_+, \quad \tilde \gamma_{23} \coloneqq \gamma_{23} + (\Id, \Id)_\# \eta_- \fstop \]
	It is easy to check that these are admissible too, i.e.,~$\tilde \gamma_{12} \in \Adm_\TrDis(\mu_1,\mu_2)$ and~$\tilde \gamma_{23} \in \Adm_\TrDis(\mu_2,\mu_3)$, as well as that~$\cost(\gamma_{12}) = \cost(\tilde \gamma_{12})$ and $\cost(\gamma_{23}) = \cost(\tilde \gamma_{23})$. Furthermore,~$\pi^2_\# \tilde \gamma_{12}$ equals~$\pi^2_\# \tilde \gamma_{23}$. The gluing lemma \cite[Lemma 5.3.2]{AmbrosioGigliSavare08} supplies a nonnegative Borel measure~$\tilde\gamma_{123}$ such that
	\[ \pi^{12}_\# \tilde \gamma_{123} = \tilde\gamma_{12} \quad \text{and} \quad \pi^{23}_\# \tilde \gamma_{123} = \tilde\gamma_{23} \fstop \] \pagebreak
	The measure~$\gamma \coloneqq \pi^{13}_\# \tilde \gamma_{123}$ is $\TrDis$-admissible between~$\mu_1$ and~$\mu_2$. By the Minkowski inequality,
	\begin{equation*}
		\TrDis(\mu_1,\mu_2) \le  \sqrt{\cost(\gamma)} \le \sqrt{\cost(\tilde \gamma_{12})} + \sqrt{\cost(\tilde \gamma_{23})} = \sqrt{\cost(\gamma_{12})} + \sqrt{\cost(\gamma_{23})} \comma
	\end{equation*}
	from which, by arbitrariness of~$\gamma_{12}$ and~$\gamma_{23}$, the triangle inequality follows.
\end{proof}

In general,~$\TrDis$ is \emph{not} a true metric on~$\cone$. This is proven in~\Cref{prop:informloss}. However, an analogue of~\Cref{lemma:admTrFunclosed} holds (proof omitted).

\begin{lemma} \label{lemma:admDisFunclosed}
	Let~$(\mu^n)_{n \in {\N_0}}$ and~$(\nu^n)_{n \in {\N_0}}$ be two sequences in~$\cone$, and let~$\gamma^n \in \Adm_\TrDis(\mu^n,\nu^n)$ for every~$n \in {\N_0}$. Assume that
	\begin{enumerate}[(a)]
		\item $\mu^n \to_n \mu$ and~$\nu^n \to_n \nu$ weakly for some~$\mu,\nu$,
		\item $\mu^n_\Omega \to_n \mu_\Omega$ and~$\nu^n_\Omega \to_n \nu_\Omega$ \emph{setwise}, i.e., on all Borel sets,
		\item $\gamma^n \to_n \gamma$ weakly.
	\end{enumerate}
	Then~$\mu,\nu \in \cone$ and $\gamma \in \Adm_\TrDis(\mu,\nu)$.
	
	In particular, for any~$\mu,\nu \in \cone$, the set~$\Adm_\TrDis(\mu,\nu)$ is sequentially closed with respect to the weak convergence.
\end{lemma}

\subsection{When~$\Omega$ is a finite union of intervals,~$\TrDis$ is a distance} \label{sec:distance}

When~$\Omega$ is a finite union of~$1$-dimensional intervals (equivalently, when~$\partial \Omega$ is a finite set) we also have
\[ \TrDis(\mu,\nu) = 0 \quad \Longleftrightarrow \quad \mu = \nu \fstop \]

\begin{proposition} \label{prop:distance}
	If~$d=1$ and~$\Omega$ is a finite union of intervals, then~$(\cone,\TrDis)$ is a metric space.
\end{proposition}

This proposition is an easy consequence of the following remark and lemma, analogous to \Cref{rmk:boundGamma} and \Cref{lemma:optTrFun}, respectively.

\begin{remark} \label{rmk:boundGamma2}
	Fix~$\mu, \nu \in \cone$ and pick~$\gamma \in \Adm_{\TrDis}(\mu,\nu)$. If~$\partial \Omega$ is finite and the diagonal of~$\partial \Omega \times \partial \Omega$ is~$\gamma$-negligible, then
	\begin{align} \begin{split}
			\norm{\gamma} &\le \norm{\gamma_\Omega^\ovom} + \norm{\gamma_\ovom^\Omega} + \norm{\gamma_{\partial \Omega}^{\partial \Omega}} \le \norm{\mu_\Omega} + \norm{\nu_\Omega} + \frac{1}{\min_{\stackrel{x,y \in \partial \Omega}{x \neq y}} \abs{x-y}^2} \int \abs{x-y}^2 \dif \gamma(x,y) \\
			&\le \norm{\mu_\Omega} + \norm{\nu_\Omega} + \const \, \cost(\gamma) \fstop
		\end{split}
	\end{align}
\end{remark}

\begin{lemma} \label{lemma:optTrDis}
	Assume that~$d=1$ and that $\Omega$ is a finite union of intervals. Then the set~$\Opt_\TrDis(\mu,\nu)$ is nonempty for every~$\mu,\nu \in \cone$.
\end{lemma}

\begin{proof}
	We already know that~$\Adm_\TrDis(\mu,\nu) \neq \emptyset$. Let us take a minimizing sequence~$(\gamma^n)_{n \in {\N_0}} \subseteq \Adm_{\TrDis}(\mu,\nu)$ for the cost functional~$\cost$. Let~$\Delta$ be the diagonal of $\partial \Omega \times \partial\Omega$. It is easy to see that~$(\gamma^n - \gamma^n|_\Delta)_{n}$ is still an admissible and minimizing sequence. Therefore, we can assume that~$\gamma^n|_\Delta = 0$. By \Cref{rmk:boundGamma2}, the total variation of~$\gamma^n$ is bounded. Therefore, there exists a subsequence of~$(\gamma^n)_n$ that converges weakly to a limit~$\gamma$ and, by \Cref{lemma:admDisFunclosed},~$\gamma \in \Adm_\TrDis(\mu,\nu)$. Since the sequence is minimizing,~$\gamma$ is also~$\TrDis$-optimal.
\end{proof}

Two further useful facts about~$\TrDis$ are the counterparts of \Cref{lemma:boundaryVSinterior} and \Cref{prop:lscI} in the case where~$\Omega$ is a finite union of intervals.

\begin{lemma} \label{lemma:boundaryVSinteriorTrDis}
	Assume that~$d=1$ and that $\Omega$ is a finite union of intervals. Let~$\mu, \nu \in \cone$ and let~$\Phi \colon \ovom \to \R$ be Lipschitz continuous. Then,
	\begin{equation} \label{boundaryVSinteriorTrDis0}
		\abs{\mu(\Phi) - \nu(\Phi)} \le \const_{\Phi} \TrDis(\mu, \nu) \sqrt{\norm{\mu_\Omega} + \norm{\nu_\Omega} + \TrDis^2(\mu, \nu)} \fstop
	\end{equation}
\end{lemma}

\begin{proof}
	By Condition~\ref{(3)} in \Cref{def:TrDis}, for every~$\mu,\nu \in \cone$ and every~$\gamma \in \Opt_{\TrDis}(\mu,\nu)$, we have
	\begin{align*} %
		\begin{split}
			\abs{\mu(\Phi) - \nu(\Phi)} &= \abs{\int \bigl(\Phi(x)-\Phi(y) \bigr) \dif \gamma(x,y)} 
			\le (\Lip \Phi) \int \abs{x-y} \dif \gamma(x,y) \\ &\le (\Lip \Phi)  \sqrt{\cost(\gamma) \, \norm{\gamma}}  = (\Lip \Phi) \TrDis(\mu,\nu) \sqrt{\norm{\gamma}} \fstop
		\end{split}
	\end{align*}
	We can assume that the diagonal of~$\partial \Omega \times \partial \Omega$ is~$\gamma$-negligible; hence, we conclude by \Cref{rmk:boundGamma2}.
\end{proof}

\begin{proposition} \label{prop:lscIbis}
	Assume that~$d=1$ and that $\Omega$ is a finite union of intervals. Then~$\Ent$ is lower semicontinuous w.r.t.~$\TrDis$.
\end{proposition}

\begin{proof}
	Similar to the proof of \Cref{prop:lscI}, making use of \Cref{lemma:boundaryVSinteriorTrDis} in place of \Cref{lemma:boundaryVSinterior}.
\end{proof}

When~$\TrDis$ defines a metric, a natural question is whether or not this metric is complete. In general, the answer is \emph{no}; this is proven in \Cref{prop:notComplete}. %
Nonetheless, we prove in \Cref{lemma:completeSublev} that the \emph{sublevels} of~$\Ent$ are complete for~$\TrDis$. %

Another interesting problem is to find a convergence criterion for~$\TrDis$. Exploiting \Cref{lemma:TrDisKR}, we find a simple sufficient condition for convergence in the~$1$-dimensional setting.

\begin{lemma} \label{lemma:convergCrit}
	Assume that~$d=1$ and that $\Omega$ is a finite union of intervals. If~$(\mu^n)_{n \in \N_0} \subseteq \cone$ converges weakly to~$\mu \in \cone$, then~$\mu^n \stackrel{\TrDis}{\to}_n \mu$.
\end{lemma}
\begin{proof}
	The idea is to use \Cref{lemma:TrDisKR} together with the measure-theoretic result~\cite[Theorem~8.3.2]{Bogachev07}: the metric induced by~$\norm{\cdot}_{\KR}$ metrizes the weak convergence\footnote{In~\cite{Bogachev07}, two Kantorovich--Rubinstein norms are defined. Here, we implicitly use that they are equivalent on measures on a bounded metric space; see~\cite[Section~8.10(viii)]{Bogachev07}.} of \emph{nonnegative} Borel measures on~$\ovom$. For every~$x \in \partial \Omega$, let~$a_x \coloneqq -\inf_{n} \mu_n(x)$. Every number~$a_x$ is finite because, by the uniform boundedness principle, the total variation of~$\mu^n$ is bounded. By the considerations above, we have
	\begin{multline*}
		\mu^n \to_n \mu \text{ weakly} \quad \Longrightarrow \quad \mu^n + \sum_{x \in \partial \Omega} a_x \delta_x \to_n \mu + \sum_{x \in \partial \Omega} a_x \delta_x \text{ weakly} \\
		\Longrightarrow \quad \norm{\mu^n - \mu}_\KR \to_n 0 \quad \stackrel{\eqref{eq:TrDisKR}}{\Longrightarrow} \quad \TrDis(\mu^n, \mu) \to_n 0 \fstop\qedhere
	\end{multline*}
\end{proof}

\begin{remark}
	The converse of \Cref{lemma:convergCrit} is not true: in the case~$\Omega \coloneqq (0,1)$, consider the sequence
	\[
	\mu^n \coloneqq n(\delta_{1/n}-\delta_0) \comma \qquad n \in \N_1 \comma
	\]
	which converges to~$\mu \coloneqq 0$ w.r.t.~$\TrDis$.
\end{remark}

\subsection{Estimate on the directional derivative}
The following lemma will be used in \Cref{prop:sobolevReg} to characterize the solutions of the variational problem~\eqref{eq:jkoTrFun0}. We omit its simple proof, almost identical to that of~\cite[Proposition~2.11]{FigalliGigli10}.

\begin{lemma} \label{lemma:dirDer}
	Let~$\mu,\nu \in \cone$ and~$\gamma \in \Opt_{\TrFun}(\mu,\nu)$. Let~$\boldsymbol w\colon \Omega \to \R^d$ be a bounded and Borel vector field with compact support. For~$t > 0$ sufficiently small, define~$\mu_t \coloneqq (\Id + t\boldsymbol{w})_\# \mu$. Then
	\begin{equation} \label{eq:dirDer1}
		\limsup_{t \to 0^+} \frac{\TrFun^2(\mu_t,\nu) - \TrFun^2(\mu, \nu)}{t} \le -2 \int \langle \boldsymbol{w}(x), y-x \rangle \dif \gamma(x,y) \fstop
	\end{equation}
\end{lemma}

\subsection{Existence of transport maps} \label{sec:exTranspPlans}

\begin{proposition} \label{prop:trPlans}
	Let~$\mu,\nu \in \cone$, let~$A,B \subseteq \ovom \times \ovom$ be Borel sets, and let~$\gamma$ be a nonnegative Borel measure on~$\ovom \times \ovom$. If
	\begin{enumerate}[(a)]
		\item $\gamma \in \Opt_\TrDis(\mu,\nu)$,
		\item or:~$\gamma \in \Opt_{\TrFun}(\mu,\nu)$ and~$(A \times B) \cap (\partial \Omega \times \partial \Omega) = \emptyset$,
	\end{enumerate}
	then~$\gamma_A^B$ is optimal \emph{for the classical $2$-Wasserstein distance} between its marginals.
\end{proposition}

Consequently: under the assumptions of this proposition, if one of the two marginals of~$\gamma_A^B$ is absolutely continuous, we can apply Brenier's theorem \cite{Brenier87} and deduce the existence of an optimal transport map. For instance, whenever~$\mu_\Omega$ is absolutely continuous, there exists a Borel map~$T \colon \Omega \to \ovom$ such that~$\gamma_\Omega^\ovom = (\Id,T)_\# \mu_\Omega$.

\begin{proof}[Proof of \Cref{prop:trPlans}]
	Let~$\tilde \gamma$ be any nonnegative Borel coupling between~$\pi^1_\# \gamma_A^B$ and~$\pi^2_\# \gamma_A^B$. In particular,~$\tilde \gamma$ is concentrated on~$A \times B$. Define the nonnegative measure
	\[
	\gamma' \coloneqq \gamma - \gamma_A^B + \tilde \gamma \fstop
	\]
	Note that
	\[
	\pi^1_\# \gamma' = \pi^1_\# \gamma \quad \text{and} \quad \pi^2_\# \gamma' = \pi^2_\# \gamma \comma
	\]
	which yields
	\[
	\gamma \in \Adm_{\TrDis}(\mu,\nu) \quad \Longrightarrow \quad \gamma' \in \Adm_{\TrDis}(\mu,\nu) \fstop
	\]
	Furthermore, if~$\gamma_{\partial \Omega}^{\partial \Omega} = 0$, then~$ (\gamma')_{\partial \Omega}^{\partial \Omega} = \tilde \gamma_{\partial \Omega}^{\partial \Omega}$. Thus,
	\[
	\bigl[\gamma \in \Adm_{\TrFun}(\mu,\nu) \text{ and } (A \times B) \cap (\partial \Omega \times \partial \Omega) = \emptyset \bigr] \quad \Longrightarrow \quad \gamma' \in \Adm_{\TrFun}(\mu,\nu) \fstop
	\]
	Hence, if~$\gamma \in \Opt_\TrDis(\mu,\nu)$, or~$\gamma \in \Opt_{\TrFun}(\mu,\nu)$ and~$(A \times B) \cap (\partial \Omega \times \partial \Omega) = \emptyset$, then, by optimality,~$\cost(\gamma) \le \cost( \gamma')$, and we infer that~$\cost(\gamma_A^B) \le \cost(\tilde \gamma)$. We conclude by arbitrariness of~$\tilde \gamma$.
\end{proof}

In~\cite[Proposition~2.3]{FigalliGigli10} and~\cite[Proposition~3.2]{Morales18}, the authors give more precise characterizations of the optimal plans for their respective transportation functionals in terms of suitable $c$-cyclical monotonicity of the support, as in the classical optimal transport theory; see, e.g.,~\cite[Lecture~3]{AmbrosioBrueSemola21}. Existence of transport plans is then derived as a consequence. We believe that a similar analysis can be carried out for the transport plans in~$\Opt_\TrFun$ and~$\Opt_\TrDis$, but it is not necessary for the purpose of this work.

\section[Proof of Theorem 1.1]{Proof of \Cref{Theorem_1.1}} \label{sec:main10}

Recall the scheme~\eqref{eq:jkoTrFun0}: we first fix a measure~$\mu_0 \in \cone$ such that its restriction to~$\Omega$ is absolutely continuous (w.r.t.~the Lebesgue measure) with density equal to~$\rho_0$. Then, for every~$\tau > 0$ and~$n \in {\N_0}$, we iteratively choose
\begin{equation*} \mu_{(n+1) \tau}^{\tau} \in \argmin_{\mu \in \cone} \, \left(\Ent(\mu) + \frac{\TrFun^2(\mu, \mu_{n\tau}^\tau)}{2\tau} \right) \fstop \end{equation*}
For all~$\tau > 0$, these sequences are extended to maps~$t \mapsto \mu^\tau_t$, constant on the intervals~$\bigl[n\tau,(n+1)\tau\bigr)$ for every~$n \in \N_0$.

\begin{remark}
	The choice of~$(\mu_0)_{\partial \Omega}$ is inconsequential, in the sense that, for every~$t$ and~$\tau$ the restriction~$(\mu_t^\tau)_\Omega$ does not depend on it. In fact, from \Cref{rmk:changeBoundary} and the uniqueness of the minimizer in~\eqref{eq:jkoTrFun0} (i.e.,~\Cref{prop:uniq}), it is possible to infer the following proposition (proof omitted).
\end{remark}

\begin{proposition}
	Fix~$\tau > 0$, and let~$\mu_0, \tilde \mu_0 \in \cone$ be such that~$(\mu_0)_\Omega = (\tilde \mu_0)_\Omega$. Let~$t \mapsto \mu_t^\tau$ and~$t \mapsto \tilde \mu_t^\tau$ be the maps constructed with the scheme~\eqref{eq:jkoTrFun0}, starting from~$\mu_0$ and~$\tilde \mu_0$, respectively. Then, for every~$t \ge 0$,
	\begin{equation}
		\mu_t^\tau - \tilde \mu_t^\tau = \mu_0 - \tilde \mu_0 = (\mu_0)_{\partial \Omega} - (\tilde \mu_0)_{\partial \Omega} \fstop
	\end{equation}
\end{proposition}

We are going to prove \Cref{Theorem_1.1} in seven steps, corresponding to as many (sub)sections:
\begin{enumerate}[1.]
	\item Existence: The scheme is well-posed, in the sense that there exists a minimizer for the variational problem~\eqref{eq:jkoTrFun0}.
	\item Boundary condition: The minimizers of~\eqref{eq:jkoTrFun0} approximately satisfy the boundary condition~$\rho|_{\partial \Omega} = e^{\Psi-V}$.
	\item Sobolev regularity: There are minimizers such that their restriction to $\Omega$ enjoy some Sobolev regularity, with quantitative estimates, and satisfy a ``precursor'' of the Fokker--Planck equation.
	\item Uniqueness: There is only one minimizer for~\eqref{eq:jkoTrFun0} (given~$\mu^\tau_{n\tau}$).
	\item Contractivity: Suitably truncated~$L^q$ norms decrease in time along~$t \mapsto \mu_t^\tau$. This result is useful in proving convergence of the scheme, both w.r.t.~$Wb_2$ and in~$L^1_\mathrm{loc} \bigl( (0,\infty); L^q(\Omega) \bigr)$.
	\item Convergence w.r.t.~$Wb_2$.
	\item Fokker--Planck with Dirichlet boundary conditions: The limit solves the Fokker--Planck equation with the desired Dirichlet boundary conditions. Moreover, the convergence holds in~$L^1_\mathrm{loc} \bigl( (0,\infty); L^q(\Omega) \bigr)$ for~$q \in [1,\frac{d}{d-1})$.
\end{enumerate}
Each (sub)section starts with the precise statement of the corresponding main proposition and ends with its proof. When needed, some preparatory lemmas precede the proof.

\subsection{One step of the scheme} \label{sec:onestep}
In this section, we gather together the subsections corresponding to the first five steps of our plan for \Cref{Theorem_1.1}. The reason is that they all involve only one step of the discrete scheme. \pagebreak

Throughout this section,~$\bar \mu$ is any measure in~$\cone$ whose restriction to~$\Omega$ is absolutely continuous and such that, denoting by~$\bar \rho$ the density of~$\bar \mu_\Omega$, the quantity~$\E(\bar \rho)$ is finite. We also fix~$\tau > 0$. We aim to find one/all minimizer(s) of
\begin{equation} \label{eq:singlestep} \Ent(\cdot) + \frac{\TrFun^2(\cdot,\bar \mu)}{2\tau} \, : \, \cone \to \R \end{equation}
and determine some of its/their properties.

\subsubsection{Existence}

\begin{proposition} \label{prop:existence}
	There exists at least one minimizer of the function in~\eqref{eq:singlestep}. Every minimizer~$\mu$ satisfies the following:
	\begin{enumerate}
		\item Both~$\Ent(\mu)$ and~$\TrFun(\mu,\bar \mu)$ are finite. In particular,~$\mu_\Omega$ admits a density~$\rho$.
		\item The total variation of~$\mu$ and the integral~$\int_\Omega \rho \log \rho \dif x$ can be bounded by a constant~$\const_{\tau, \bar \mu}$ that depends on~$V$ only through~$\norm{V}_{L^\infty}$.
		\item The following inequality holds:
		\begin{equation} \label{eq:existence:inequality}
			\frac{\TrFun^2(\mu,\bar \mu)}{4\tau} \le \E(\bar \rho) - \E(\rho) + \mu_\Omega(\Psi) - \bar \mu_\Omega(\Psi) + \const \tau \bigl(\norm{\mu_\Omega} + \norm{\bar \mu_\Omega} \bigr) \fstop
		\end{equation}
	\end{enumerate}
\end{proposition}

The proof of this proposition, partially inspired by \cite[Propositions~4.3~\&~5.9]{Morales18}, is essentially an application of the \emph{direct method in the calculus of variations}, although some care is needed due to the unboundedness of~$\Ent$ from below.

\begin{proof}[Proof of \Cref{prop:existence}]
	Let~$(\mu^n)_{n \in {\N_1}} \subseteq \cone$ be a minimizing sequence for~\eqref{eq:singlestep}. We may assume that
	\begin{equation} \label{eq:optimalitymun} \Ent(\mu^n) + \frac{\TrFun^2(\mu^n,\bar \mu)}{2\tau} \le \Ent(\bar \mu) + \frac{\TrFun^2(\bar \mu,\bar \mu)}{2\tau} + \frac{1}{n} = \Ent(\bar \mu) + \frac{1}{n} < \infty \comma \qquad n \in \N_1 \comma \end{equation}
	where the finiteness of~$\Ent(\bar \mu)$ is consequence of~$\E(\bar \rho) < \infty$. For every~$n$, let~$\rho^n$ be the density of~$\mu^n_\Omega$ and let~$\gamma^n \in \Opt_\TrFun(\mu^n,\bar \mu)$.
	
	\emph{Step 1 (preliminary bounds).} Firstly, we shall do some work towards the proof of~\eqref{eq:existence:inequality} and establish uniform integrability for~$\set{\rho^n}_n$. By~\eqref{eq:optimalitymun} and \Cref{lemma:boundaryVSinterior},
	\begin{multline} \label{eq:preIneq}
		\frac{\TrFun^2(\mu^n,\bar \mu)}{2\tau} \le \Ent(\bar \mu) - \Ent(\mu^n) + \frac{1}{n} = \E(\bar \rho) - \E(\rho^n) + \bar \mu_{\partial \Omega}(\Psi)-\mu^n_{\partial \Omega}(\Psi) +  \frac{1}{n} \\
		\le \E(\bar \rho) - \E(\rho^n) + \mu^n_\Omega(\Psi) - \bar \mu_\Omega(\Psi) + \tau (\Lip \Psi)^2 \bigl( \norm{\mu_\Omega^n} + \norm{\bar \mu_\Omega} \bigr) + \frac{\TrFun^2(\mu^n,\bar\mu)}{4\tau} + \frac{1}{n} \comma
	\end{multline}
	from which,
	\begin{equation} \label{eq:onestep:massbound}
		\int_\Omega \rho^n \log \rho^n \le \int_\Omega \left( \bar \rho \log \bar \rho + (\norm{V}_{L^\infty} + \norm{\Psi}_{L^\infty} + 1 + \tau(\Lip \Psi)^2) (\bar \rho + \rho^n) \right) \dif x + \frac{1}{n} \fstop
	\end{equation}
	Since~$l \mapsto l \log l$ is superlinear, we have uniform integrability of~$\set{\rho^n}_n$. In particular,~$\norm{\mu^n_\Omega}$ is bounded.
	
	Also the total variation~$\norm{\mu^n}$ is bounded. Indeed,
	\begin{equation} \label{eq:tvbound}
		\norm{\mu^n} \le 2\norm{\gamma^n} + \norm{\bar \mu} \le 2\norm{\mu^n_\Omega} + 3 \norm{\bar \mu} \comma
	\end{equation}
	where the first inequality follows from Condition~\ref{(3)} in \Cref{def:TrDis}, and the second one from \Cref{rmk:boundGamma}.%
	
	\emph{Step 2 (existence).} We can extract a (not relabeled) subsequence such that:
	\begin{enumerate} \item $\mu^n_{\partial \Omega} \to_n \eta$ for some~$\eta$ weakly in duality with $C(\partial \Omega)$,
		\item $\rho^n \rightharpoonup_n \rho$ for some~$\rho$ weakly in~$L^1(\Omega)$,
		\item $\mu^n \to_n  \mu \coloneqq \rho \dif x + \eta$ weakly in duality with~$C(\ovom)$, and~$ \mu \in \cone$.
	\end{enumerate}
	Since the functional~$\E$ is sequentially lower semicontinuous w.r.t.~the weak convergence in~$L^1(\Omega)$, and sum of lower semicontinuous functions is lower semicontinuous, \Cref{cor:sciT} yields
	\begin{equation*} \Ent( \mu) + \frac{\TrFun^2( \mu,\bar \mu)}{2\tau}
		\le \liminf_{n \to \infty} \left( \Ent(\mu^n) + \frac{\TrFun^2(\mu^n,\bar\mu)}{2
			\tau} \right) = \inf \left( \Ent(\cdot) + \frac{\TrFun^2(\cdot,\bar\mu)}{2
			\tau} \right) \fstop
	\end{equation*}
	
	\emph{Step 3 (inequalities).} If~$\mu$ is \emph{any} minimizer for~\eqref{eq:singlestep}, the inequality~\eqref{eq:existence:inequality}, and the bounds on~$\norm{\mu}$ and~$\int_\Omega \rho \log \rho \dif x$ directly follow from~\eqref{eq:preIneq},~\eqref{eq:onestep:massbound}, and~\eqref{eq:tvbound} by taking the constant sequence equal to~$\mu$ in place of~$(\mu^n)_n$.
\end{proof}

\subsubsection{Boundary condition}

Pick any minimizer~$\mu$ for~\eqref{eq:singlestep} and denote by~$\rho$ the density of~$\mu_\Omega$. Let~$\gamma \in \Opt_\TrFun(\mu, \bar \mu)$ and let~$S \colon \Omega \to \ovom$ be such that~$\gamma_\Omega^\ovom = (\Id,S)_\# \mu_\Omega$.

\begin{proposition} \label{prop:boundary}
	There exists a $\Leb^d$-negligible set~$N \subseteq \Omega$ such that:
	\begin{enumerate}
		\item For all~$x \in \Omega \setminus N$ and~$y \in \partial \Omega$, the inequalities
		\begin{equation} \label{eq:prop:boundary:00}
			-\frac{\abs{x-y}^2}{2\tau} \le \log \rho(x) - \Psi(y) + V(x) \le \const \frac{\abs{x-y}}{\tau} + \const \tau
		\end{equation}
		hold. The constant~$\const$ can be chosen independent of~$V$.
		\item For all~$x \in \Omega \setminus N$ such that~$S(x) \in \partial \Omega$, we have the identity
		\begin{equation}
			\label{eq:prop:boundary:01}
			\log \rho(x) = \Psi(S(x)) - V(x) - \frac{\abs{x-S(x)}^2}{2\tau} \fstop
		\end{equation}
	\end{enumerate}
\end{proposition}

\begin{remark} \label{rmk:equivmeasur}
	\Cref{prop:boundary} implies in particular that~$\rho \in L^\infty(\Omega)$ and that~$\rho$ is bounded from below by a positive constant (depending on~$\tau$). In particular, the measure~$\mu_\Omega$ is equivalent to the Lebesgue measure on~$\Omega$.
\end{remark}

\begin{remark} \label{rmk:g}
	Define
	\[
	g \coloneqq \sqrt{\rho e^V}-e^{\Psi/2} \comma \quad g^{{(\kappa)}} \coloneqq (g-\kappa)_+-(g+\kappa)_- \comma \qquad \kappa > 0 \fstop
	\]
	It follows from~\eqref{eq:prop:boundary:00} that, when~$\kappa \ge c (e^{c \tau} - 1)$, for a suitable constant~$c$ independent of~$V$ and~$\tau$, the function~$g^{{(\kappa)}}$ is compactly supported in~$\Omega$ (up to changing its value on a Lebesgue-negligible set).
\end{remark}

\begin{remark}
	The term~$\const \tau$ at the right-hand side of~\eqref{eq:prop:boundary:00} can be removed when~$\Psi$ is constant. This fact can be easily checked in the proof of \Cref{prop:boundary} and is consistent with \cite[Proposition~3.7 (27)]{AmbrosioGigliSavare08}. However, the following example proves that, in general, this extra term is necessary, i.e.,~the boundary condition need not be satisfied \emph{exactly} by the map~$t \mapsto \mu_t^\tau$ (even for~$t\ge \tau$).
\end{remark}

\begin{example}
	Let~$\Omega \coloneqq (0,1)$ and~$V \equiv 0$, and choose~$\bar \mu = 0$. Since~$\bar \mu = 0$, we necessarily have~$S(x) \in \partial \Omega = \set{0,1}$ for~$\mu_\Omega$-a.e.~$x$, hence for~$\Leb^1$-a.e.~$x \in \Omega$ by \Cref{rmk:equivmeasur}. Additionally, by \Cref{prop:boundary}, for~$\Leb^1$-a.e.~$x \in S^{-1}(0)$ we have
	\[
	\Psi(1) - \frac{\abs{1-x}^2}{2\tau} \stackrel{\eqref{eq:prop:boundary:00}}{\le} \log \rho(x) \stackrel{\eqref{eq:prop:boundary:01}}{=} \Psi(0) - \frac{\abs{x}^2}{2\tau}
	\]
	and, after rearranging,
	\[
	x \le \frac{1}{2} + \tau\bigl(\Psi(0)-\Psi(1)\bigr) \fstop
	\]
	Therefore, when~$\Psi$ and~$\tau$ are such that~$\tau\bigl(\Psi(0)-\Psi(1)\bigr) < -\frac{1}{2}$, the set~$S^{-1}(0)$ is negligible, i.e.,~$S(x) = 1$ for~$\Leb^1$-a.e.~$x \in \Omega$. Then, \eqref{eq:prop:boundary:01} gives
	\[
	\log \rho(x) = \Psi(1) - \frac{\abs{1-x}^2}{2\tau} \qquad \text{for~$\Leb^1$-a.e.~$x \in \Omega$} \comma
	\]
	and, therefore, the trace of~$\rho$ at~$0$ is~$\exp\left(\Psi(1) - \frac{1}{2\tau}\right) > \exp\bigl(\Psi(0)\bigr)$.
\end{example}

\Cref{prop:boundary} is analogous to~\cite[Proposition~3.7~(27)~\&~(28)]{FigalliGigli10} and~\cite[Proposition~5.2~(5.39)~\&~(5.40)]{Morales18}. Like those, ours is proven by taking suitable variations of the minimizer~$\mu$. %

\begin{proof}[Proof of~\Cref{prop:boundary}]
	We shall prove the inequalities in the statement for~$x$ out of negligible sets~$N_y$ \emph{that depend on~$y$}. This is sufficient because the set~$\partial \Omega$ is separable and all the functions in the statement are continuous in the variable~$y$.
	Fix~$y \in \partial \Omega$.
	
	\emph{Step 1 (first inequality in~\eqref{eq:prop:boundary:00}).} Let~$\epsilon > 0 $, take a Borel set~$A \subseteq \Omega$, and define
	\[ \tilde \mu_1 \coloneqq \mu + \epsilon \Leb^d_{A} - \epsilon \abs{A} \delta_y \in \cone \comma \quad \tilde \gamma_1 \coloneqq \gamma + \epsilon \Leb^d_{A} \otimes \delta_y \in \Adm_\TrFun(\tilde \mu_1, \bar \mu) \fstop \]
	By the minimality property of~$\mu$ and the optimality of~$\gamma$,
	\[ 0 \le \int_{A} \left( \frac{(\rho + \epsilon) \log (\rho + \epsilon) - \rho \log \rho }{\epsilon} + V - 1 - \Psi(y) + \frac{\abs{x-y}^2}{2\tau} \right) \dif x \fstop \]
	Since the function~$l \mapsto l \log l$ is convex, we can use the monotone convergence theorem (``downwards'') to find
	\[ 0 \le \int_{A} \left( \log \rho  + V  - \Psi(y) + \frac{\abs{x-y}^2}{2\tau} \right) \dif x \fstop \]
	By arbitrariness of~$A$, we have the first inequality in~\eqref{eq:prop:boundary:00} for~$x$ out of a~$\Leb^d$-negligible set (possibly dependent on~$y$). In particular,~$\rho > 0$.
	
	\emph{Step 2 (second indequality in~\eqref{eq:prop:boundary:00} on~$S^{-1}(\Omega)$).} Let~$\epsilon \in (0,1)$, take a Borel set~$A \subseteq S^{-1}(\Omega)$, define
	\begin{align*} \tilde \mu_2 &\coloneqq \mu + \epsilon \mu(A) \delta_y - \epsilon \mu_A \in \cone \comma \\ \tilde \gamma_2 &\coloneqq \gamma - \epsilon (\Id,S)_\# \mu_A + \epsilon \delta_y \otimes S_\# \mu_A \in \Adm_{\TrFun}(\tilde \mu_2,\bar \mu) \fstop \end{align*}
	Note that~$A \subseteq S^{-1}(\Omega)$ is needed to ensure that~$(\tilde \gamma_2)_{\partial \Omega}^{\partial \Omega} = 0$. This time, the minimality property gives
	\[ 0 \le \int \left( \frac{(1-\epsilon)\log(1-\epsilon)}{\epsilon} - \log \rho - V + 1 + \Psi(y) + \frac{\langle y-\Id, y+\Id-2S \rangle }{2\tau} \right) \dif \mu_A \fstop \]
	We conclude by arbitrariness of~$A$, after letting~$\epsilon \to 0$, that
	\[ \log \rho(x) + V(x) - \Psi(y) \le \frac{\langle y-x, y+x-2S(x)\rangle }{2\tau} \le \diam(\Omega)\frac{\abs{x-y}}{\tau} \]
	for~$\mu$-a.e.~$x \in S^{-1}(\Omega)$. Since~$\rho > 0$, the same is true~$\Leb^d_{S^{-1}(\Omega)}$-a.e.
	
	\emph{Step 3 (identity~\eqref{eq:prop:boundary:01}).} Let~$\epsilon \in (0,1)$, take a Borel set~$A \subseteq S^{-1}(\partial \Omega)$, define
	\begin{align*} \tilde \mu_3 &\coloneqq \mu + \epsilon S_\# \mu_A  - \epsilon \mu_A \in \cone \comma \\ \tilde \gamma_3 &\coloneqq \gamma - \epsilon (\Id,S)_\# \mu_A \in \Adm_{\TrFun}(\tilde \mu_3,\bar \mu) \fstop \end{align*}
	By the minimality property,
	\[ 0 \le \int \left( \frac{(1-\epsilon)\log(1-\epsilon)}{\epsilon} - \log \rho - V + 1 + \Psi \circ S - \frac{\abs{\Id - S}^2}{2\tau} \right) \dif \mu_A \comma \]
	from which, by arbitrariness of~$\epsilon$ and~$A$, we infer the inequality~$\le$ in~\eqref{eq:prop:boundary:01}~$\Leb^d_{S^{-1}(\partial \Omega)}$-a.e. The inequality~$\ge$ follows from the first inequality in~\eqref{eq:prop:boundary:00}.
	
	\emph{Step 4 (second inequality in~\eqref{eq:prop:boundary:00} on~$S^{-1}(\partial \Omega)$).} We make use of~\eqref{eq:prop:boundary:01}, the Lipschitz continuity of~$\Psi$, the triangle inequality, and the inequality~$2ab-b^2 \le a^2$:
	\begin{align*} \log \rho(x) - \Psi(y) + V(x) &\stackrel{\eqref{eq:prop:boundary:01}}{=} \Psi(S(x))-\Psi(y)-\frac{\abs{x-S(x)}^2}{2\tau} \\
		&\le (\Lip \Psi) \abs{S(x)-y}-\frac{\abs{x-S(x)}^2}{2\tau} \\
		&\le (\Lip \Psi) \abs{x-S(x)}-\frac{\abs{x-S(x)}^2}{2\tau} + (\Lip \Psi) \abs{x-y} \\
		&\le \frac{\tau (\Lip \Psi)^2}{2} + (\Lip \Psi) \abs{x-y} \fstop \end{align*}
	Eventually, we conclude with the estimate
	\[ \abs{x-y} \le \frac{\abs{x-y}}{2\tau} + \frac{\tau\abs{x-y}}{2} \le \frac{\abs{x-y}}{2\tau} + \frac{\tau\diam(\Omega)}{2} \fstop \qedhere \]
\end{proof}

\subsubsection{Sobolev regularity}
\begin{proposition} \label{prop:sobolevReg}
	Let~$\mu$ be a minimizer of~\eqref{eq:singlestep} and denote by~$\rho$ the density of~$\mu_\Omega$.
	\begin{enumerate}
		\item The function~$\rho$ belongs to~$W^{1,(2 \wedge d)}_\mathrm{loc}(\Omega)$, and~$\sqrt{ \rho e^V}$ belongs to~$W^{1,2}(\Omega)$. We have the estimates
		\begin{equation} \label{eq:reg2}
			\norm{\nabla\sqrt{\rho e^V}}_{L^2} \le \const \frac{\TrFun( \mu,\bar \mu)}{ \tau} \comma
		\end{equation}
		and, for every~$q \in [1,\infty)$ such that~$q(d-2) \le d$,
		\begin{equation} \label{eq:reg2bis}
			\norm{\rho}_{L^q} \le \const_{q} \left( e^{\const \tau} +  \norm{\nabla\sqrt{\rho e^V}}_{L^2}^2 + \norm{\rho}_{L^1} \right) \fstop
		\end{equation}
		If~$d=1$, the same is true with~$q = \infty$ too.
		\item For every~$\gamma \in \Opt_\TrFun(\mu,\bar \mu)$, writing~$\gamma_{\Omega}^\ovom = (\Id,S)_\# \mu_\Omega$, we have \begin{equation} \label{eq:onestep:transportMap2}
			\frac{S-\Id}{\tau} \rho = \nabla \rho + \rho \nabla V = e^{-V} \nabla(\rho e^V) \qquad \Leb^d\text{-a.e.~on~} \Omega \fstop
		\end{equation}
	\end{enumerate}
\end{proposition}

The core idea to prove \Cref{prop:sobolevReg} is to compute the first variation of the functional~\eqref{eq:singlestep} at a minimizer and exploit \Cref{lemma:dirDer}, like in \cite[Proposition~3.6]{FigalliGigli10}. However, the proof is complicated by the weak assumptions on~$V$ and the lack of regularity of the boundary~$\partial \Omega$. To manage~$V$, we rely on an approximation argument (in the next lemma). The issue with~$\partial \Omega$ is that the the Sobolev embedding theorem is not available for functions in~$W^{1,2}(\Omega)$. Nonetheless, we can still apply it to functions in~$W^{1,2}_0(\Omega)$. To do this, we leverage the approximate boundary conditions of \Cref{prop:boundary}.

\begin{lemma} \label{lem:sobReg}
	Let~$\mu$ be a minimizer of~\eqref{eq:singlestep} and denote by~$\rho$ the density of~$\mu_\Omega$. Let~$\boldsymbol w \colon \Omega \to \R^d$ be a~$C^\infty$-regular vector field with compact support. For~$\epsilon > 0$ sufficiently small, define~$\mu^\epsilon \coloneqq (\Id + \epsilon \boldsymbol w)_\# \mu$. Then
	\begin{equation}
		\lim_{\epsilon \to 0^+} \frac{\Ent(\mu) - \Ent(\mu^\epsilon)}{\epsilon} = \int_\Omega \bigl(\dive \boldsymbol w - \langle \nabla V, \boldsymbol w \rangle \bigr) \rho \dif x \fstop
	\end{equation}
\end{lemma}

\begin{proof}
	Let~$R_\epsilon (x) \coloneqq x + \epsilon \boldsymbol{w}(x)$. Fix~$\epsilon$ sufficiently small and an open set~$\omega \Subset \Omega$ so that~$R_{s\epsilon}$ is a diffeomorphism from~$\omega$ to itself and equals the identity on~$\Omega \setminus \omega$ for every~$s \in (0,1)$, and~$\inf_{s \in (0,1), x \in \Omega} \abs{\det \nabla R_{s\epsilon}(x)} > 0$. It can be easily checked that the density~$\rho^{\epsilon}$ of~$\mu^{\epsilon}_\Omega$ satisfies
	\[ \rho^{\epsilon} \circ R_\epsilon = \frac{\rho}{\det \nabla R_\epsilon} \quad \Leb^d\text{-a.e. on } \Omega \semicolon \]
	therefore,
	\begin{equation} \label{eq:ratioH}
		\begin{aligned}
			\frac{\Ent(\mu) - \Ent(\mu^\epsilon)}{\epsilon} &= \int_\Omega \frac{\log \rho - \log(\rho^{\epsilon} \circ R_\epsilon) + V - V \circ R_\epsilon }{\epsilon} \dif \mu_\Omega \\
			&= \int_\Omega \frac{\log \det \nabla R_\epsilon}{\epsilon} \dif \mu_\Omega + \int_{\Omega} \frac{V-V \circ R_\epsilon}{\epsilon} \dif \mu_\Omega \fstop
		\end{aligned}
	\end{equation}
	By the dominated convergence theorem,
	\begin{equation*}
		\lim_{\epsilon \to 0^+} \int_\Omega \frac{\log \det \nabla R_\epsilon}{\epsilon} \dif \mu_\Omega = \int_\Omega (\dive \boldsymbol w) \rho \dif x \fstop
	\end{equation*}
	To deal with the last term in~\eqref{eq:ratioH}, we choose an open set~$\tilde \omega$ such that~$\omega \Subset \tilde \omega \Subset \Omega$. By \Cref{sec:functions}, we have~$V \in W^{1,p}(\tilde \omega)$ for some~$p > d$ and, by Friedrichs’ theorem~\cite[Theorem~9.2]{Brezis11},
	the function~$V|_\omega$ is the limit in~$W^{1,p}(\omega)$ and a.e.~of (the restriction to~$\omega$ of) a sequence of equibounded functions~$(V_k)_{k \in \N_0} \subseteq C_c^\infty(\R^d)$. For every~$k$, we have
	\begin{align*}
		\int \frac{V-V \circ R_\epsilon}{\epsilon} \dif \mu_\Omega &= \int_\omega \frac{V-V_k}{\epsilon} \rho \dif x + \int_\omega \frac{V_k \circ R_\epsilon -V \circ R_\epsilon}{\epsilon} \rho \dif x - \int_\omega \langle \nabla V_k, \boldsymbol w \rangle \rho \dif x \\
		&\quad- \int_0^1 \int_\omega \langle(\nabla V_k) \circ R_{s\epsilon} - \nabla V_k, \boldsymbol w \rangle \rho \dif x \dif s \fstop %
	\end{align*}
	With a change of variables, we rewrite the last integral as
	\[
	\int_0^1 \int_\omega \langle(\nabla V_k) \circ R_{s\epsilon} - \nabla V_k, \boldsymbol w \rangle \rho \dif x \dif s
	=
	\int_\omega \Bigl\langle \nabla V_k,  \int_0^1 \frac{(\boldsymbol w \rho) \circ R_{s\epsilon}^{-1}}{\det \nabla R_{s\epsilon} \circ R_{s\epsilon}^{-1}} \dif s - \boldsymbol w \rho \Bigr\rangle \dif x \fstop
	\]
	Recall that~$\rho \in L^\infty(\Omega)$ by \Cref{rmk:equivmeasur}. Passing to the limit in~$k$, we find that
	\begin{equation*}
		\int \frac{V-V \circ R_\epsilon}{\epsilon} \dif \mu_\Omega + \int_\Omega \langle \nabla V, \boldsymbol{w} \rangle \rho \dif x = \int_\omega \Bigl\langle \nabla V,  \int_0^1 \frac{(\boldsymbol w \rho) \circ R_{s\epsilon}^{-1}}{\det \nabla R_{s\epsilon} \circ R_{s\epsilon}^{-1}} \dif s - \boldsymbol w \rho \Bigr\rangle \dif x \fstop
	\end{equation*}
	It only remains to prove that the right-hand side in the latter is negligible as~$\epsilon \to 0$. Let~$(\rho_l)_{l \in \N_0}$ be a sequence of continuous and equibounded functions that converge to~$\rho$ almost everywhere (hence in~$L^{p'}$). Using the triangle inequality and Minkowski's integral inequality, for~$l \in  \N_0$, we write
	\begin{multline*}
		\norm{\int_0^1 \frac{(\boldsymbol w \rho) \circ R_{s\epsilon}^{-1}}{\det \nabla R_{s\epsilon} \circ R_{s\epsilon}^{-1}} \dif s - \boldsymbol w \rho}_{L^{p'}} \\ \le
		\int_0^1 \norm{ \frac{(\boldsymbol w \rho - \boldsymbol w \rho_l) \circ R_{s\epsilon}^{-1}}{\det \nabla R_{s\epsilon} \circ R_{s\epsilon}^{-1}} }_{L^{p'}} \dif s + \norm{\boldsymbol w \rho_l - \boldsymbol w \rho }_{L^{p'}} 
		\\ + \int_0^1	\norm{ \frac{(\boldsymbol w \rho_l) \circ R_{s\epsilon}^{-1}}{\det \nabla R_{s\epsilon} \circ R_{s\epsilon}^{-1}}  - \boldsymbol w \rho_l}_{L^{p'}} \dif s \fstop
	\end{multline*}
	A change of variables yields
	\[
	\norm{ \frac{(\boldsymbol w \rho - \boldsymbol w \rho_l) \circ R_{s\epsilon}^{-1}}{\det \nabla R_{s\epsilon} \circ R_{s\epsilon}^{-1}} }_{L^{p'}} = \norm{\frac{\boldsymbol w \rho - \boldsymbol w\rho_l}{\abs{\det \nabla R_{s\epsilon}}^{1/p}}}_{L^{p'}} \fstop
	\]
	Hence, when we let~$\epsilon \to 0$, using that~$\rho_l$ is continuous, we find
	\[
	\limsup_{\epsilon \to 0} \norm{\int_0^1 \frac{(\boldsymbol w \rho) \circ R_{s\epsilon}^{-1}}{\det \nabla R_{s\epsilon} \circ R_{s\epsilon}^{-1}} \dif s - \boldsymbol w \rho}_{L^{p'}} \le 2\norm{\boldsymbol w \rho - \boldsymbol w\rho_l}_{L^{p'}} \comma
	\]
	and we conclude by arbitrariness of~$l$.
\end{proof}

\begin{proof}[Proof of \Cref{prop:sobolevReg}]
	\emph{Step 1 (inequality~\eqref{eq:reg2}). } Let~$\boldsymbol w \colon \Omega \to \R^d$ be a~$C^\infty$-regular vector field with compact support. For~$\epsilon > 0$ sufficiently small, define~$\mu^\epsilon \coloneqq (\Id + \epsilon \boldsymbol w)_\# \mu \in \cone$. Since~$\mu$ is optimal for~\eqref{eq:singlestep},
	\[
	\frac{\Ent(\mu) - \Ent(\mu^\epsilon)}{\epsilon} \le \frac{\TrFun^2(\mu^{\epsilon},\bar \mu) - \TrFun^2(\mu,\bar \mu)}{2\epsilon \tau} \fstop
	\]
	We can pass to the limit~$\epsilon \to 0$ using \Cref{lemma:dirDer} and \Cref{lem:sobReg} to find that
	\begin{equation} \label{eq:preTransportMapOld}
		\int_\Omega \bigl(\dive \boldsymbol{w} - \langle \nabla V, \boldsymbol{w}\rangle \bigr) \rho \dif x  \le -\frac{1}{\tau} \int \langle \boldsymbol{w}(x), y-x \rangle \dif  \gamma(x,y)
		\le \norm{\boldsymbol w}_{L^2 (\rho)} \frac{\TrFun(\mu, \bar \mu)}{\tau} \comma
	\end{equation}
	for any~$\gamma \in \Opt_{\TrFun}(\mu,\bar \mu)$. By the Riesz representation theorem, this means that there exists a vector field~$\boldsymbol u \in L^2(\rho;\R^d)$ such that
	\begin{equation} \label{eq:boundu}
		\norm{\boldsymbol u}_{L^2(\rho)} \le \frac{\TrFun(\mu, \bar \mu)}{\tau} \comma
	\end{equation}
	and
	\[
	\int_\Omega \bigl(\dive \boldsymbol{w} - \langle \nabla V, \boldsymbol{w}\rangle \bigr) \rho \dif x = \int_\Omega \langle \boldsymbol u, \boldsymbol{w} \rangle \rho \dif x \comma
	\]
	for all smooth and compactly supported vector fields~$\boldsymbol w$.
	In other words,~$-\rho (\boldsymbol u + \nabla V)$ is the distributional gradient of~$\rho$. Since~$\rho \in L^\infty(\Omega)$ (see \Cref{rmk:equivmeasur}) and~$V \in W^{1,d+}_\mathrm{loc}(\Omega)$, we now know that~$\rho \in W^{1,(2 \wedge d)}_\mathrm{loc}(\Omega)$. Hence, for every smooth~$\boldsymbol w$ that is compactly supported,
	\begin{align*}
		\int_{\Omega} \sqrt{\rho e^{V}} \dive \boldsymbol w &\dif x = \lim_{\epsilon \downarrow 0} \int_{\Omega} \sqrt{\rho e^{V} + \epsilon} \dive \boldsymbol w \dif x = \lim_{\epsilon \downarrow 0} \int_{\Omega} \frac{\rho e^{V}}{2\sqrt{\rho e^{V} + \epsilon}} \langle \boldsymbol u, \boldsymbol w \rangle \dif x \\
		&\le \frac{\norm{\boldsymbol u}_{L^2(\rho)}}2 \liminf_{\epsilon \downarrow 0} \sqrt{\int_\Omega \frac{\rho e^{2V} \abs{\boldsymbol w}^2}{\rho e^{V} + \epsilon} \dif x } = \frac{ \norm{\boldsymbol u}_{L^2(\rho)} \norm{\boldsymbol w}_{L^2(e^{V})} }{2} \comma
	\end{align*}
	where, for the second equality, we used a standard property of the composition of Sobolev functions (cf.~\cite[Proposition~9.5]{Brezis11}) and, in the last one, the monotone convergence theorem. It follows that that~$\sqrt{\rho e^{V}} \in W^{1,2}(\Omega)$ with
	\begin{equation} \label{eq:sobBoundOld}
		\int_{\Omega} \abs{\nabla \sqrt{\rho e^{V}}}^2 e^{-V} \dif x \le \left( \frac{\norm{\boldsymbol{u}}_{L^2(\rho)}}{2} \right)^2 \stackrel{\eqref{eq:boundu}}{\le} \frac{\TrFun^2(\mu, \bar \mu)}{4\tau^2} \comma
	\end{equation}
	which, since~$V$ is bounded, yields~\eqref{eq:reg2}.
	
	\emph{Step 2 (inequality~\eqref{eq:reg2bis}).} Pick~$q$ as in the statement, i.e.,~$1 \le q < \infty$~with~$q(d-2) \le d$ or, if~$d=1$,~$q \in [1,\infty]$. Inequality~\eqref{eq:reg2bis} would follow from the Sobolev embedding theorem~\cite[Corollary~9.14]{Brezis11} if~$\partial \Omega$ were regular enough. Nonetheless, by~\cite[Remark~20, Chapter~9]{Brezis11}, even with no regularity on~$\partial \Omega$, we still have that the inclusion~$W^{1,2}_0(\Omega) \hookrightarrow L^q(\Omega)$ is continuous. Consider the functions~$g$ and~$g^{(\kappa)}$ of \Cref{rmk:g} and fix~$\kappa = c(e^{c\tau}-1)$ for a suitable constant~$c$ independent of~$\tau$ (and~$q$), so that~$g^{(\kappa)}$ is compactly supported, hence in~$W^{1,2}_0(\Omega)$. From the Sobolev embedding theorem we obtain~$\norm{ g^{(\kappa)}}_{L^{2q}} \le \const_q \norm{ g^{(\kappa)}}_{W^{1,2}}$ and, therefore,
	\begin{align*}
		\norm{\sqrt{ \rho e^{V}}}_{L^{2q}} &\le \const_q + \norm{ g}_{L^{2q}} \le \const_q (1+\kappa) + \norm{ g^{(\kappa)}}_{L^{2q}} \le \const_{q} \left( 1+\kappa+ \norm{ g^{(\kappa)}}_{W^{1,2}} \right) \\
		&\le \const_{q} \left( 1+\kappa+ \norm{ g}_{W^{1,2}} \right) \le \const_{q} \left( 1+\kappa+ \norm{\sqrt{ \rho e^{V}}}_{W^{1,2}} \right) \\
		&\le \const_{q} \left( 1+\kappa + \norm{\nabla\sqrt{\rho e^V}}_{L^2} + \sqrt{\norm{ \rho}_{L^1}} \right) \comma
	\end{align*}
	which can be easily transformed into~\eqref{eq:reg2bis}.

	\emph{Step 3 (identity~\eqref{eq:onestep:transportMap2}).} Let~$\gamma \in \Opt_{\TrFun}(\mu,\bar \mu)$ and let~$S$ be such that~$\gamma_\Omega^\ovom = (\Id,S)_\# \mu_\Omega$. From~\eqref{eq:preTransportMapOld} we infer that
	\begin{multline*} \label{eq:pretransportmapnew}
		-2\int_\Omega  \sqrt{\rho e^{-V}} \left \langle \nabla \sqrt{\rho e^{V}}, \boldsymbol w \right\rangle \dif x  \le -\frac{1}{\tau} \int \langle \boldsymbol{w}(x), y-x \rangle \dif  \gamma(x,y) \\ = -\frac{1}{\tau} \int \langle \boldsymbol{w}, S-\Id \rangle \rho \dif x  \fstop
	\end{multline*}
	By arbitrariness of~$\boldsymbol w$,~\eqref{eq:onestep:transportMap2} follows.
\end{proof}

\subsubsection{Uniqueness}
Let us assume that~$\mu$ and~$\mu'$ are two minimizers for~\eqref{eq:singlestep}
such that their restrictions to~$\Omega$ are absolutely continuous; let~$\rho$ and~$\rho'$ be their respective densities. Let~$\gamma \in \Opt_\TrFun(\mu,\bar \mu)$ and~$\gamma' \in \Opt_\TrFun(\mu', \bar \mu)$. By \Cref{prop:trPlans}, we can write
\begin{align*}
	&\gamma_\Omega^\ovom = (\Id,S)_\# \mu_\Omega \comma \quad( \gamma')_\Omega^\ovom = (\Id,S')_\# \mu_\Omega \comma \\ &\gamma_\ovom^\Omega = (T,\Id)_\# \bar \mu_\Omega \comma \quad (\gamma')_\ovom^\Omega = (T',\Id)_\# \bar \mu_\Omega \comma
\end{align*}
for some appropriate Borel maps. %

\begin{proposition} \label{prop:uniq}
	The two measures~$\mu$ and~$\mu'$ are equal.
\end{proposition}

Note that uniqueness is not immediate, given that the functional~$\Ent$ is not strictly convex. This setting is different from that of~\cite{Morales18} and~\cite{FigalliGigli10}: therein, measures are defined \emph{only on~$\Omega$}. Instead, we claim here that the measure~$\mu$, on the \emph{whole}~$\ovom$, is uniquely determined.

The proof of \Cref{prop:uniq} is preceded by three lemmas: the first one concerns the identification of~$S$ and~$S'$; the second one, similar to~\cite[Proposition~A.3~(A.5)]{Morales18}, shows that~$T|_{T^{-1}(\partial \Omega)}$ and~$T'|_{(T')^{-1}(\partial \Omega)}$ enjoy one same property, inferred from the minimality of~$\mu$ and~$\mu'$; the third one ensures that this property identifies uniquely~$T$ (i.e.,~$T=T'$) on~$T^{-1}(\partial \Omega) \cap (T')^{-1}(\partial \Omega)$.

\begin{lemma} \label{lemma:uniqS}
	If~$\mu_\Omega = \mu_\Omega'$, then~$S(x) = S'(x)$ for~$\Leb^d_\Omega$-a.e~$x$.
\end{lemma}

\begin{proof}
	This statement immediately follows from~\eqref{eq:onestep:transportMap2} in \Cref{prop:sobolevReg}.
\end{proof}

\begin{lemma} \label{lemma:argmin}
	For~$\bar \mu$-a.e.~point~$x \in \Omega$ such that~$T(x) \in \partial \Omega$, we have
	\begin{equation}
		T(x) \in \argmin_{y \in \partial \Omega} \left( \Psi(y) + \frac{\abs{x-y}^{2}}{2\tau} \right) \fstop
	\end{equation}
	An analogous statement holds for~$T'$.
\end{lemma}

\begin{proof}
	Set
	\begin{equation} \label{eq:deff} f(x,y) \coloneqq \Psi(y) + \frac{\abs{x-y}^2}{2\tau} \comma \qquad x \in \Omega \comma y \in \partial \Omega \fstop \end{equation}
	By~\cite[Theorem 18.19]{AliprantisBorder06} there exists a Borel function~$R \colon \Omega \to \partial \Omega$ such that
	\[ R(x) \in \argmin_{y \in \partial \Omega} f(x,y) \]
	for all~$x \in \Omega$. Let~$A \subseteq T^{-1}(\partial \Omega)$ be a Borel set and consider the measure
	\[ \tilde \mu \coloneqq \mu - T_\# \bar \mu_A + R_\# \bar \mu_A \comma  \]
	which lies in~$\cone$. Additionally define
	\[ \tilde \gamma \coloneqq \gamma - (T,\Id)_\#  \bar \mu_A + (R,\Id)_\#  \bar \mu_A \]
	and notice that~$\tilde \gamma \in \Adm_{\TrFun}(\tilde \mu, \bar \mu)$. By the minimality property of~$\mu$ and the optimality of~$\gamma$, we must have
	\[ \Ent(\mu) + \frac{1}{2\tau} \cost(\gamma) \le \Ent(\tilde \mu) + \frac{1}{2\tau} \cost({\tilde \gamma}) \comma \]
	which, after rearranging the terms, gives
	\begin{equation*} \int f(x,T(x)) \dif  \bar \mu_A(x) \le \int f(x,R(x)) \dif  \bar \mu_A(x) = \int \min_{y \in \partial \Omega} f(x,y) \dif \bar \mu_A(x) \fstop \end{equation*}
	We conclude the proof by arbitrariness of~$A$.
\end{proof}

\begin{lemma} \label{lemma:uniqT}
	For~$\bar \mu$-a.e.~point~$x \in \Omega$ such that~$T(x) \in \partial \Omega$ and~$T'(x) \in \partial \Omega$, we have
	\[ T(x) = T'(x) \fstop \]
\end{lemma}

\begin{proof}
	We can resort to~\cite[Lemma 1]{Cox20} by~G.~Cox. Adopting the notation of this lemma, we set
	\[
	Q(t,z) \coloneqq \Psi(t) + \frac{\abs{z-t}^2}{2\tau} \comma \quad P \coloneqq c \, \bar \mu|_{T^{-1}(\partial \Omega) \cap (T')^{-1}(\partial \Omega)} \comma
	\]
	for some constant~$c$ that makes~$P$ a probability distribution. Four assumptions are made therein and need to be checked:
	\begin{itemize}
		\item Absolute Continuity: It follows from ~$\E(\bar \mu) < \infty$ that~$\bar \mu_\Omega$ is absolutely continuous. Hence, so is the probability~$P$.
		\item Continuous Differentiability: Conditions~(a) and~(b) are easy to check. Condition~(c) is vacuously true by setting~$A(t) \coloneqq \emptyset$ for every~$t$.
		\item Generic: Condition~(d) is true and easy to check.
		\item Manifold: This condition is not true if~$\partial \Omega$ does not enjoy any kind of regularity. However, one can check that that~$\partial \Omega$ does not need to be a union of manifolds if the condition Generic holds with~$A(t) \coloneqq \emptyset$ for every~$t$. The other topological properties, namely second-countability and Hausdorff, are trivially true, since~$\partial \Omega \subseteq \R^d$. \qedhere
	\end{itemize}
\end{proof}

\begin{proof}[Proof of \Cref{prop:uniq}]
	\emph{Step 1 (uniqueness of~$\rho$ and~$S$).} The identity~$\rho = \rho'$ follows from the strict convexity of the function~$l \mapsto l \log l$. To see why, notice that~$\frac{\gamma+\gamma'}{2} \in \Adm_\TrFun(\frac{\mu+\mu'}{2},\bar \mu)$; therefore, by minimality,
	\[ \frac{\Ent(\mu) + \frac{1}{2\tau}\cost(\gamma) + \Ent(\mu') + \frac{1}{2\tau}\cost(\gamma')}{2} \le \Ent\left( \frac{\mu + \mu'}{2} \right) + \frac{1}{2\tau} \cost\left(\frac{\gamma + \gamma'}{2} \right) \fstop \]
	Most of the terms simplify by linearity. What remains is
	\[ \int_\Omega \frac{\rho \log \rho + \rho' \log \rho'}{2} \dif x \le \int_\Omega \left( \frac{\rho + \rho'}{2} \right) \log \left( \frac{\rho + \rho'}{2} \right) \dif x \comma \]
	which implies~$\rho(x) = \rho'(x)$ for~$\Leb^d$-a.e.~$x \in \Omega$. The identity~$S=S'$ out of a~$\Leb^d_\Omega$-negligible set follows from \Cref{lemma:uniqS}.
	
	\emph{Step 2 (uniqueness of~$\gamma_{\partial \Omega}^\Omega$).} We can write
	\[ \gamma = \gamma_\Omega^\ovom + \gamma_{\partial \Omega}^\Omega \quad \text{and} \quad \gamma'=(\gamma')_\Omega^\ovom + (\gamma')_{\partial \Omega}^\Omega \fstop \]
	Because of the uniqueness of~$\mu_\Omega$ and~$S$, we have the equality~$\gamma_\Omega^\ovom = (\gamma')_\Omega^\ovom$. If we combine this fact with Condition~\ref{(2)} in~\Cref{def:TrDis}, we find
	\begin{align*}
		0 &= \left(\pi^2_\# (\gamma- \gamma')\right)_\Omega = \pi^2_\# \left(\gamma_{\partial \Omega}^\Omega- (\gamma')_{\partial \Omega}^\Omega\right) \\
		&= \pi^2_\# \left( (T,\Id)_\# \bar \mu_{T^{-1}(\partial \Omega)} - (T',\Id)_\# \bar\mu_{(T')^{-1}(\partial\Omega)} \right) = \bar \mu_{T^{-1}(\partial \Omega)} - \bar \mu_{(T')^{-1}(\partial \Omega)} \fstop \end{align*}
	This proves that~$T^{-1}(\partial \Omega)$ and~$(T')^{-1}(\partial \Omega)$ are~$\bar \mu$-essentially equal. Together with \Cref{lemma:uniqT}, this gives
	\[ \gamma_{\partial \Omega}^\Omega = (T,\Id)_\# \bar \mu_{T^{-1}(\partial \Omega)} = (T',\Id)_\#\bar \mu_{(T')^{-1}(\partial \Omega)} = (\gamma')_{\partial \Omega}^\Omega \fstop \]
	
	\emph{Step 3 (conclusion).} We have determined that~$\gamma = \gamma'$. Condition~\ref{(3)} in \Cref{def:TrFun} gives
	\[ \mu = \pi^1_\# \gamma - \pi^2_\# \gamma + \bar \mu = \pi^1_\# \gamma' - \pi^2_\# \gamma' + \bar \mu = \mu' \comma \]
	which is what we wanted to prove.
\end{proof}

\subsubsection{Contractivity}

In this section, we establish time monotonicity for some truncated and weighted~$L^q$ norm ($q \ge 1$) of the densities~$\rho_t^\tau$.

Here, too, only one step of the scheme is involved. %
We let~$\mu$ be the unique minimimum point of~\eqref{eq:singlestep} and~$\rho$ be the density of its restriction to~$\Omega$.

\begin{proposition} \label{prop:Lq}
	Let~$q \ge 1$. For every~$\vartheta \ge \vartheta_0 \coloneqq \max_{\partial \Omega} e^\Psi$, the following inequality holds (possibly, with one or both sides being infinite):
	\begin{equation} \label{eq:prop:Lq:0}
		\int_\Omega \max\set{\rho, \vartheta e^{-V}}^q e^{(q-1)V} \dif x \le \int_\Omega \max\set{\bar \rho, \vartheta e^{-V}}^q e^{(q-1)V} \dif x \fstop
	\end{equation}
\end{proposition}

\begin{remark}
	For a solution to the Fokker--Planck equation~\eqref{eq:fpStrong}, a monotonicity property like~\eqref{eq:prop:Lq:0} is expected. Indeed, \emph{formally}:
	\begin{multline*}
		\frac{\dif}{\dif t} \int_\Omega \max\set{\rho_t, \vartheta e^{-V}}^q e^{(q-1)V} \dif x = q \int_{\set{\rho_t > \vartheta e^{-V}}} (\rho_te^V)^{q-1} \dive(\nabla \rho_t + \rho_t \nabla V)  \dif x \\
		\quad = q \int_{\partial\set{\rho_t > \vartheta e^{-V} }} (\rho_t e^V)^{q-1} e^{-V} \langle \nabla (\rho_t e^V) , \boldsymbol n \rangle \dif \mathscr{H}^{d-1} \\
		\underbrace{-q(q-1) \int_{\set{\rho_t > \vartheta e^{-V}}} (\rho_t e^V)^{q-2} e^V \abs{\nabla \rho_t + \rho_t \nabla V}^2 \dif x}_{\le 0} \fstop
	\end{multline*}
	If~$\vartheta \gneq \vartheta_0$, the boundary condition forces the set $\partial\set{\rho_t > \vartheta e^{-V} } \cap \partial \Omega$ to be negligible. Moreover, on~$\partial\set{\rho_t > \vartheta e^{-V} } \cap \Omega$, the scalar product~$\langle \nabla (\rho_t e^V) , \boldsymbol n \rangle$ is nonpositive. The case~$\vartheta = \vartheta_0$ can be deduced by approximation.
\end{remark}

\begin{remark}[Mass bound] \label{rmk:massbound}
	Note that \Cref{prop:Lq} implies that the mass of~$(\mu_t^\tau)_\Omega$ is bounded by a constant~$\const$ indepentent of~$t$ and~$\tau$. Indeed,
	\begin{align*} \int_{\Omega} \rho_t^\tau \dif x &\le \int_\Omega \max\set{\rho_t^\tau, \vartheta_0 e^{-V}} \dif x \le \cdots \le \int_\Omega \max\set{\rho_0, \vartheta_0 e^{-V}} \dif x \\
		&\le \int_\Omega \rho_0 \dif x + \vartheta_0 \int_\Omega e^{-V} \dif x \fstop \end{align*}
\end{remark}

The proof of the first Step in \Cref{prop:Lq}, i.e.,~the case~$q=1$, and of the preliminary lemma \Cref{lemma:Lq} follow the lines of \cite[Proposition~3.7~(24)]{FigalliGigli10} and \cite[Proposition~5.3]{Morales18}. In all these proofs, the key is to leverage the optimality of~$\mu$ by constructing small variations. In the proof of Step~2, i.e., the case~$q > 1$, instead, our idea is to take the inequality for~$q=1$, multiply it by a suitable power of~$\vartheta$, and integrate it w.r.t.~the variable~$\vartheta$ itself. This is the reason why, while \Cref{prop:Lq} will later be used only with~$\vartheta = \vartheta_0$---or in the form of \Cref{rmk:massbound}---it is convenient to have it stated and proven (at least for~$q=1$) for a continuum of values of~$\vartheta$.

\begin{lemma} \label{lemma:Lq}
	For~$\mu$-a.e.~$x \in \Omega$ such that~$S(x) \in \Omega$, we have
	\begin{equation} \label{eq:lemma:preL1}
		\log \rho(x) + V(x) \le \log \rho(S(x)) + V(S(x)) - \frac{\abs{x-S(x)}^2}{2\tau} \fstop
	\end{equation}
\end{lemma}

\begin{proof}
	Let~$\epsilon \in (0,1)$ and let~$A \subseteq S^{-1}(\Omega)$ be a Borel set. We define
	\begin{align*} \tilde \mu &\coloneqq \mu + \epsilon S_\# \mu_A  - \epsilon \mu_A \in \cone \comma \\ \tilde \gamma &\coloneqq \gamma - \epsilon (\Id,S)_\# \mu_A + \epsilon(S,S)_\# \mu_A \in \Adm_{\TrFun}(\tilde \mu,\bar \mu) \fstop \end{align*}
	Let~$\hat \rho$ be the density of~$S_\# \mu_A$ and note that~$\hat \rho \le \bar \rho$.
	By the minimality of~$\mu$, we have
	\begin{multline*}
		0 \le \underbrace{\int_{\Omega} \frac{\bigl(\rho + \epsilon(\hat \rho - \1_A \rho)\bigr) \log\bigl(\rho + \epsilon(\hat \rho - \1_A \rho)\bigr) - \rho \log \rho }{\epsilon} \dif x}_{\coloneqq I_1} \\
		+ \int \left(V \circ S -V - \frac{\abs{\Id - S}^2}{2\tau} \right) \dif\mu_A \fstop
	\end{multline*}
	We use the convexity of~$l \mapsto l \log l$ to write
	\begin{align*} I_1 &\le \int_\Omega (\hat\rho - \1_A\rho) \left(1+\log\bigl(\rho + \epsilon(\hat \rho - \1_A \rho)\bigr)\right) \dif x \\
		&= \int_\Omega (\hat\rho - \1_A\rho) \log\bigl(\rho + \epsilon(\hat \rho - \1_A \rho)\bigr) \dif x \\
		&= \int_\Omega \hat \rho \log\bigl(\rho + \epsilon(\hat \rho - \1_A \rho) \bigr) \dif x - \int_A \rho \log\bigl( (1-\epsilon)\rho + \epsilon \hat\rho \bigr) \dif x \\
		&\le \int_\Omega \hat \rho \log(\rho + \epsilon\hat \rho \bigr) \dif x - \int_A \rho \bigl( \log \rho + \log (1-\epsilon) \bigr) \dif x \fstop \end{align*}
	On the first integral on the last line, we use the monotone convergence theorem (``downwards''): its hypotheses are satisfied because~$\hat \rho \le \bar \rho$.
	By passing to the limit~$\epsilon \to 0$, we obtain
	\begin{multline*} 0 \le \int_{\Omega} \hat \rho \log \rho \dif x + \int \left(-\log \rho + V \circ S -V - \frac{\abs{\Id-S}^2}{2\tau} \right) \dif\mu_A \\
		= \int \left(\log \rho \circ S -\log \rho + V\circ S-V - \frac{\abs{\Id-S}^2}{2\tau} \right) \dif\mu_A \comma \end{multline*}
	and we conclude by arbitrariness of~$A$.
\end{proof}

\begin{proof}[Proof of \Cref{prop:Lq}]
	\emph{Step 1 ($q = 1$).} Consider the case~$q=1$. Let
	\begin{equation} \label{eq:lemma:L1:1} A \coloneqq \set{x \in \Omega \, \colon \, \rho e^V > \vartheta} \fstop \end{equation}
	Thanks to~\eqref{eq:prop:boundary:01}, we know that~$A \cap S^{-1}(\partial \Omega)$ is~$\Leb^d$-negligible. Therefore, we can extract a~$\Leb^d_A$-full-measure Borel subset~$\tilde A$ of~$A \cap S^{-1}(\Omega)$ where~\eqref{eq:lemma:preL1} holds (recall that~$\Leb^d_\Omega \ll \mu_\Omega$). It is easy to check that~$S(\tilde A) \subseteq A$. Therefore, we have
	\begin{multline} \label{eq:lemma:L1:2}
		\int_A \max\set{\rho, \vartheta e^{-V}} \dif x \stackrel{\eqref{eq:lemma:L1:1}}{=} \int_A \rho \dif x = \int_{\tilde A} \rho \dif x \le \int_{S^{-1}(A)} \rho \dif x = S_\# \mu_\Omega(A) \\
		= \pi^2_\# \gamma_\Omega^\ovom (A) \stackrel{(A \subseteq \Omega)}{=} \pi^2_\# \gamma_\Omega^\Omega (A) \le \pi^2_\# \gamma_\ovom^\Omega (A) = \bar \mu_\Omega(A) \le \int_A \max\set{\bar \rho, \vartheta e^{-V}} \dif x \fstop
	\end{multline}
	On the other hand,
	\begin{equation} \label{eq:lemma:L1:3} \int_{\Omega \setminus A} \max\set{\rho, \vartheta e^{-V}} \dif x \stackrel{\eqref{eq:lemma:L1:1}}{=} \int_{\Omega \setminus A} \vartheta e^{-V} \dif x \le \int_{\Omega \setminus A} \max\set{\bar \rho, \vartheta e^{-V}} \dif x \comma \end{equation}
	and we conclude by taking the sum of~\eqref{eq:lemma:L1:2} and~\eqref{eq:lemma:L1:3}.
	
	\emph{Step 2 ($q > 1$)} Assume now that~$q > 1$. Define
	\[ f \coloneqq \max\set{\rho, \vartheta e^{-V}} \comma \quad g  \coloneqq \max\set{\bar \rho, \vartheta e^{-V}} \fstop \]
	Note that the case~$q=1$ implies
	\begin{equation} \label{eq:prop:Lq:1} \int_\Omega \max\set{f, \tilde\vartheta e^{-V}} \dif x \le \int_\Omega \max\set{g, \tilde \vartheta e^{-V}} \dif x \end{equation}
	\emph{for every~$\tilde \vartheta > 0$}. After multiplying~\eqref{eq:prop:Lq:1} by~${\tilde \vartheta}^{q-2}$, integrating w.r.t.~$\tilde \vartheta$ from~$0$ to some~$\Theta > 0$, and changing the order of integration with Tonelli's theorem, we find
	\begin{multline*} \int_\Omega \left( \int_0^{\min\set{fe^V,\Theta}} {\tilde \vartheta}^{q-2} \dif {\tilde \vartheta} \right) f \dif x + \int_\Omega \left( \int_{\min\set{fe^V,\Theta}}^{\Theta} {\tilde \vartheta}^{q-1} \dif {\tilde \vartheta} \right) e^{-V} \dif x \\
		\le \int_\Omega \left( \int_0^{\min\set{ge^V,\Theta}} {\tilde \vartheta}^{q-2} \dif {\tilde \vartheta} \right) g \dif x + \int_\Omega \left( \int_{\min\set{ge^V,\Theta}}^{\Theta} {\tilde \vartheta}^{q-1} \dif {\tilde \vartheta} \right) e^{-V} \dif x \comma \end{multline*}
	whence
	\begin{multline*} \frac{1}{q-1} \int_\Omega \min\set{f e^V, \Theta}^{q-1} f \dif x - \frac{1}{q} \int_\Omega \min\set{f e^V, \Theta}^{q} e^{-V} \dif x \\
		\le \frac{1}{q-1} \int_\Omega \min\set{g e^V, \Theta}^{q-1} g \dif x - \frac{1}{q} \int_\Omega \min\set{g e^V, \Theta}^{q} e^{-V} \dif x \fstop \end{multline*}
	It follows that
	\begin{multline*} \left(\frac{1}{q-1} - \frac{1}{q} \right) \int_{\Omega} \min\set{fe^V, \Theta}^q e^{-V} \dif x  + \frac{1}{q} \int_\Omega \min\set{g e^V, \Theta}^{q} e^{-V} \dif x \\ \le  \frac{1}{q-1} \int_\Omega \min\set{g e^V, \Theta}^{q-1} g \dif x \fstop \end{multline*}
	We now let~$\Theta \to \infty$ and deduce from the monotone convergence theorem that
	\[  \left(\frac{1}{q-1} - \frac{1}{q} \right) \int_\Omega f^q e^{(q-1)V} \dif x + \frac{1}{q} \int_\Omega g^q e^{{(q-1)}V} \dif x \le \frac{1}{q-1} \int_\Omega g^q e^{(q-1)V} \dif x \fstop \]
	Eventually, we can rearrange, and, noted that~$\left(\frac{1}{q-1} - \frac{1}{q} \right) > 0$, simplify to finally obtain~\eqref{eq:prop:Lq:0}.
\end{proof}

\subsection{Convergence w.r.t~$Wb_2$}
In this section, we prove convergence w.r.t.~$Wb_2$ of the measures built with the scheme~\eqref{eq:jkoTrFun0}. The argument is standard. In fact, we shall give a short proof that relies on the `refined version of Ascoli-Arzel\`a theorem'~\cite[Proposition~3.3.1]{AmbrosioGigliSavare08}.

\begin{proposition} \label{prop:convergence}
	As~$\tau \to 0$, up to subsequences, the maps~$\bigl(t\mapsto (\mu_t^\tau)_\Omega\bigr)_\tau$ converge pointwise w.r.t.~$Wb_2$ to a curve~$t \mapsto \rho_t \dif x$ of absolutely continuous measures, continuous w.r.t.~$Wb_2$.
\end{proposition}

Once again, we first need a lemma.

\begin{lemma} \label{lemma:boundSumT}
	Let~$t \ge 0$ and~$ \tau >  0$. Then
	\begin{equation} \label{eq:lemmaBoundSumT}
		\tau \int_\Omega \rho_t^\tau \log \rho_t^\tau \dif x + \sum_{i=0}^{\lfloor t/\tau \rfloor -1 } \TrFun^2\bigl(\mu_{i\tau}^\tau,\mu_{(i+1)\tau}^\tau\bigr) \le \const \, \tau (1+t+\tau) \fstop
	\end{equation}
	As a consequence,
	\begin{equation} \label{eq:lemmaBoundSumT2}
		Wb_2\bigl((\mu_s^\tau)_\Omega,(\mu_t^\tau)_\Omega\bigr) \le \TrDis \bigl(\mu_s^\tau,\mu_t^\tau\bigr) \le \const \sqrt{(t-s+\tau)(1+t+\tau)} \comma \qquad s \in [0,t] \fstop
	\end{equation}
\end{lemma}

\begin{proof}
	We use~\eqref{eq:existence:inequality} to write
	\begin{equation*} 
		\sum_{i=0}^{\lfloor t/\tau \rfloor -1 } \frac{\TrFun^2\bigl(\mu_{i\tau}^\tau,\mu_{(i+1)\tau}^\tau\bigr)}{4\tau} \le \E(\rho_0) - \E(\rho^\tau_t) + (\mu^\tau_t)_\Omega(\Psi) - (\mu_0)_\Omega(\Psi) + \const \tau \sum_{i=0}^{\lfloor t/\tau \rfloor} \norm{(\mu^\tau_{i\tau})_\Omega} \comma
	\end{equation*}
	and conclude~\eqref{eq:lemmaBoundSumT} by using \Cref{rmk:massbound}.
	
	The first inequality in~\eqref{eq:lemmaBoundSumT2} follows from~\eqref{eq:distIneq}. As for the second one, since~$\TrDis$ is a pseudometric, and by the Cauchy--Schwarz inequality and~\eqref{eq:distIneq}, we have the chain of inequalities
	\begin{align*}
		\TrDis(\mu_s^\tau,\mu_t^\tau) &\le \sum_{i=\lfloor s/\tau \rfloor}^{\lfloor t/\tau \rfloor -1 } \TrDis(\mu_{i\tau}^\tau,\mu_{(i+1)\tau}^\tau) \le \sum_{i=\lfloor s/\tau \rfloor}^{\lfloor t/\tau \rfloor -1 } \TrFun\bigl(\mu_{i\tau}^\tau,\mu_{(i+1)\tau}^\tau\bigr) \\
		&\le \sqrt{\frac{t-s+\tau}{\tau}} \sqrt{\sum_{i=\lfloor s/\tau \rfloor}^{\lfloor t/\tau \rfloor -1 } \TrFun^2\bigl(\mu_{i\tau}^\tau,\mu_{(i+1)\tau}^\tau\bigr)} \fstop
	\end{align*}
	We combine the latter with~\eqref{eq:lemmaBoundSumT} to infer~\eqref{eq:lemmaBoundSumT2}.
\end{proof}

\begin{proof}[Proof of \Cref{prop:convergence}]
	Fix~$t > 0$. We know from \Cref{lemma:boundSumT} that, for every~$s \in [0,t]$ and~$\tau \in (0,1)$, we have
	\[
	(\mu_s^\tau)_\Omega \in K_t \coloneqq \set{ \rho \dif x \, : \, \int_\Omega \rho \log \rho \dif x \le c \, (2+t)} \comma
	\]
	where~$c$ is the constant in \eqref{eq:lemmaBoundSumT}. We claim that~$K_t$ is \emph{compact} in~$(\mathcal M_2(\Omega), Wb_2)$. By identifying an absolutely continuous measure with its density,~$K_t$ can be seen as a subset of~$L^1(\Omega)$. This set is closed and convex, as well as weakly sequentially compact by the Dunford--Pettis theorem. From~\cite[Proposition 2.7]{FigalliGigli10} we know that weak convergence in~$L^1(\Omega)$ implies convergence w.r.t.~$Wb_2$; hence the claim is true.
	
	Furthermore, for every~$r,s \in [0,t]$, we have
	\[
	\limsup_{\tau \to 0} Wb_2\bigl((\mu_r^\tau)_\Omega,(\mu_s^\tau)_\Omega\bigr) \stackrel{\eqref{eq:lemmaBoundSumT2}}{\le} \const \sqrt{\abs{s-r}(1+t)} \fstop
	\]
	
	All the hypotheses of~\cite[Proposition~3.3.1]{AmbrosioGigliSavare08} are satisfied; thus, we conclude the existence of a subsequence of~$\bigl(s \mapsto (\mu^\tau_s)_\Omega\bigr)_\tau$ that converges, pointwise in~$[0,t]$ w.r.t.~$Wb_2$, to a continuous curve of measures. Each limit measure lies in~$K_t$; hence it is absolutely continuous. With a diagonal argument, we find a single subsequence that converges pointwise on the whole half-line~$[0,\infty)$.
\end{proof}

\subsection{Solution to the Fokker--Planck equation with Dirichlet boundary conditions} 
We are now going to conclude the proof of \Cref{Theorem_1.1} by showing that the limit curve is, in fact, a solution to the linear Fokker--Planck equation with the desired boundary conditions.

\begin{proposition} \label{prop:FP}
	If the sequence~$\bigl(t \mapsto (\mu_t^\tau)_\Omega\bigr)_\tau$ converges, pointwise w.r.t.~$Wb_2$ as~$\tau \to 0$, to~$t \mapsto \rho_t  \dif x$, then~$\rho^\tau \to_\tau \rho$ also in~$L^1_\mathrm{loc} \bigl( (0,\infty) ; L^q(\Omega) \bigr)$ for every~$q \in [1,\frac{d}{d-1})$. The curve~$t \mapsto \rho_t \dif x$ solves the linear Fokker--Planck equation in the sense of \Cref{sec:defFokker}, and the map~$t \mapsto \left(\sqrt{\rho_t e^V} - e^{\Psi/2}\right)$ belongs to~$L^2_\mathrm{loc}\bigl([0,\infty ); W^{1,2}_0(\Omega)\bigr)$.
\end{proposition}

Like in the proofs of \cite[Theorem~3.5]{FigalliGigli10} and \cite[Theorem~4.1]{Morales18}, the key to  \Cref{prop:FP} is to first determine (see \Cref{lemma:FP}) that the measures constructed with~\eqref{eq:jkoTrFun0} already solve approximately the Fokker--Planck equation. In order to prove that the limit curve has the desired properties and that convergence holds in~$L^1_\mathrm{loc}\bigl( (0,\infty); L^q(\Omega) \bigr)$ (\Cref{lemma:improvedconv}), two further preliminary lemmas turn out to be particularly useful. Both provide quantitative bounds at the discrete level: one (\Cref{lemma:sobolev}) for~$\sqrt{\rho^\tau e^V}$ in~$L^2_\mathrm{loc}\bigl( (0,\infty); W^{1,2}(\Omega) \bigr)$; the other (\Cref{lemma:lebesgueBound}) for~$\rho^\tau$ in~$L^\infty_\mathrm{loc}\bigl((0,\infty) ; L^q(\Omega) \bigr)$, for suitable values of~$q$. In turn, these bounds are deduced from \Cref{prop:sobolevReg} and \Cref{prop:Lq}.

\begin{lemma}[Sobolev bound] \label{lemma:sobolev}
	If $\tau \le t$, then,
	\begin{equation} \int_\tau^t \norm{ \sqrt{\rho^\tau_r e^V}}_{W^{1,2}}^2 \dif r \le \const (1+t) \fstop \end{equation}
\end{lemma}

\begin{proof}
	Let~$r \ge \tau$. By~\eqref{eq:reg2}, we have 
	\[ \norm{\nabla \sqrt{\rho^\tau_r e^V}}_{L^2}^2 \le \const \frac{\TrFun^2\left(\mu_{\lfloor  r/\tau \rfloor \tau}^\tau , \mu_{\lfloor r/\tau \rfloor \tau - \tau}^\tau \right)}{\tau^2} \fstop \]
	Thus,
	\[ \int_\tau^t \norm{\nabla \sqrt{\rho^\tau_r e^V}}_{L^2}^2 \dif r
	\le \const \sum_{i=0}^{\lfloor t/\tau \rfloor -1} \frac{\TrFun^2\left(\mu^\tau_{(i+1)\tau}, \mu^\tau_{i\tau}\right)}{\tau} \comma \]
	which, using \Cref{lemma:boundSumT}, can be easily reduced to the desired inequality.
\end{proof}

\begin{lemma}[Lebesgue bound] \label{lemma:lebesgueBound}
	Let~$q \in [1, \infty)$ be such that~$q(d-2) \le d$. If~$\tau < t$, then
	\begin{equation} \label{eq:lebesgueBound}
		\norm{\rho^\tau_t}_{L^q} \le \const_{q} e^{\const \tau} \frac{1+t}{t-\tau} \fstop
	\end{equation}
\end{lemma}

\begin{proof} 
	For every~$r \in [0,t]$, \Cref{prop:Lq} gives
	\begin{align*}
		\norm{\rho^\tau_t}_{L^q} &\le \const_{q} \left(\int_{\Omega} \max\set{\rho_t^\tau e^V,\vartheta_0}^{q} e^{-V} \dif x\right)^{1/q} \\
		&\le \const_{q} \left( \int_{\Omega} \max\set{\rho_{r}^\tau e^V,\vartheta_0}^q e^{-V} \dif x \right)^{1/q} 
		\le \const_{q}\left(1+\norm{\rho_r^\tau}_{L^q}\right) \comma
	\end{align*}
	and if, additionally,~$r \ge \tau$, then~\eqref{eq:reg2bis} yields
	\[
	\norm{\rho^\tau_t}_{L^q} \le \const_{q} \left(e^{\const \tau} + \norm{\nabla \sqrt{\rho^\tau_r e^V}}_{L^2}^2 + \norm{\rho_r^\tau}_{L^1} \right) \fstop
	\]
	After integrating w.r.t.~$r$ from~$\tau$ to~$t$, \Cref{lemma:sobolev} and \Cref{rmk:massbound} imply~\eqref{eq:lebesgueBound}.
\end{proof}

\begin{lemma}[Approximate Fokker--Planck] \label{lemma:FP}
	Let~$\omega \Subset \Omega$ be open, let~$\varphi \in C^2_0(\omega)$, and let~$s,t$ be such that~$0 \le s \le t$. Then,~$\rho^\tau,\rho^\tau \nabla V \in L^1_\mathrm{loc}\bigl( ( \tau,\infty ); L^1(\omega) \bigr)$, and
	\begin{multline} \label{eq:lemma:FP:00}
		\abs{\int_\Omega (\rho^\tau_t - \rho^\tau_s) \varphi \dif x - \int_{\lfloor \frac{s}{\tau} \rfloor \tau+\tau}^{\lfloor \frac{t}{\tau} \rfloor\tau + \tau} \int_\Omega (\Delta \varphi - \langle \nabla \varphi, \nabla V \rangle) \rho_r^\tau \dif x \dif r } \\ \le \const_{\omega} \, \tau(1+t+\tau) \norm{\varphi}_{C^2_0(\omega)}\fstop 
	\end{multline}
	Moreover, for~$\epsilon > 0$, the inequality
	\begin{equation}\label{eq:lemma:FP:01}
		\norm{\rho^\tau_t - \rho^\tau_s}_{(C_0^2(\omega))^*} \le \const_{\omega, \epsilon} (t-s+\tau)
	\end{equation}
	holds whenever~$0<2\tau \le \epsilon \le s \le t \le 1/\epsilon$.
\end{lemma}

\begin{remark}
	In~\eqref{eq:lemma:FP:01}, we identify~$\rho^\tau_t - \rho^\tau_s$ with the continuous linear functional
	\[
	C_0^2(\omega) \ni \varphi \longrightarrow \int_\omega (\rho^\tau_t - \rho^\tau_s) \varphi \dif x \fstop
	\]
\end{remark}

\begin{proof}[Proof of \Cref{lemma:FP}] \emph{Step 1 (integrability).} From \Cref{rmk:massbound}, it follows trivially that~$\rho^\tau \in L^1_\mathrm{loc}\bigl([0,\infty); L^1(\Omega)\bigr)$. We shall prove that the function~$\rho^\tau \nabla V $ belongs to $ L^1_\mathrm{loc}\bigl( ( \tau,\infty ); L^1(\omega) \bigr)$ for every~$\omega \Subset \Omega$ open. Fix~$ a,b >0$ with~$\tau < a \le b$. Let~$p$ be as in \Cref{sec:functions}. Its conjugate exponent~$p'$ satisfies~$p' \in [1,\infty)$ and~$p'(d-2) \le d$.
	By H\"older's inequality and~\Cref{lemma:lebesgueBound}, we have
	\begin{equation}  \label{eq:lemma:FP:3}
		\begin{aligned}
			\int_a^b \norm{\rho_r^\tau \nabla V}_{L^1} \dif r &\le \norm{\nabla V}_{L^p(\omega)} \int_a^b \norm{\rho_r^\tau}_{L^{p'}} \dif r \stackrel{\eqref{eq:lebesgueBound}}{\le} \const_{p} \norm{\nabla V}_{L^p(\omega)} e^{\const \tau} \int_a^b \frac{1+r}{r-\tau}  \dif r \\
			&\le \const_{p} \norm{\nabla V}_{L^p(\omega)} e^{\const \tau} \frac{1+b}{a-\tau} (b-a) \le \const_{\omega} e^{\const \tau} \frac{1+b}{a-\tau} (b-a) \fstop
		\end{aligned}
	\end{equation}
	The last passage is due to the fact that both~$p$ and~$\norm{\nabla V}_{L^p(\omega)}$ can be seen as functions of~$V$ and~$\omega$.
	
	\emph{Step 2 (inequality~\eqref{eq:lemma:FP:00}).} 
	Let~$i \in \N_0$, and choose~$\gamma^i \in  \Opt_\TrFun\bigl(\mu^\tau_{(i+1)\tau},\mu^\tau_{i\tau}\bigr)$ and~$S_i\colon \Omega \to \ovom$ as in~\eqref{eq:onestep:transportMap2}. By the triangle inequality and the fact that~$\rho_r^\tau = \rho^\tau_{(i+1)\tau}$ when~$r \in \bigl[(i+1)\tau,(i+2)\tau \bigr)$, we have
	\begin{multline*}
		\abs{\int_\Omega (\rho^\tau_{(i+1)\tau} - \rho^\tau_{i\tau}) \varphi \dif x - \int_{(i+1)\tau}^{(i+2)\tau} \int_\Omega (\Delta \varphi - \langle \nabla \varphi, \nabla V \rangle) \rho_r^\tau \dif x \dif r } \\ 
		\le \underbrace{\abs{\int_\Omega \left(\varphi - \varphi \circ S_i-\tau \Delta \varphi + \tau\langle \nabla \varphi, \nabla V \rangle \right) \rho^\tau_{(i+1)\tau} \dif x}}_{\eqqcolon I_1^i} \\
		+ \underbrace{\abs{\int_\Omega \bigl( (\varphi \circ S_i) \rho_{(i+1)\tau}^\tau - \varphi \rho^\tau_{i\tau}\bigr) \dif x}}_{\eqqcolon I_2^i} \fstop
	\end{multline*}
	Using~$\eqref{eq:onestep:transportMap2}$, we rewrite~$I_1^i$ as
	\[ I_1^i = \abs{\int_\Omega \left( \varphi - \varphi \circ S_i + \langle \nabla \varphi, S_i - \Id \rangle \right) \rho_{(i+1)\tau}^\tau \dif x } \comma \]
	and then, by means of Taylor's theorem with remainder in Lagrange form, we establish the upper bound
	\[ I_1^i \le \const \norm{\varphi}_{C^2_0(\omega)} \int_\Omega \abs{S_i-\Id}^2 \, \rho_{(i+1)\tau}^\tau \dif x \le \const \norm{\varphi}_{C^2_0(\omega)} \TrFun^2\left(\mu^\tau_{(i+1)\tau}, \mu^\tau_{i\tau}\right) \fstop \]
	By Condition~\ref{(2)} in \Cref{def:TrDis} and the fact that~$\varphi$ is supported in the closure of~$\omega$, we have
	\begin{align*} I_2^i &= \abs{\int_\ovom \varphi(y) \dif \pi^2_\# (\gamma_\Omega^\ovom - \gamma_\ovom^\Omega) } = \abs{\int_\ovom \varphi(y) \dif \pi^2_\# (\gamma_\Omega^\omega - \gamma_\ovom^\omega) } \le \norm{\varphi}_{L^\infty(\omega)} \norm{\gamma_{\partial \Omega}^\omega} \\
		&\le \const_\omega \norm{\varphi}_{L^\infty(\omega)} \int_{\partial \Omega \times \omega} \abs{x-y}^2 \dif \gamma(x,y) \le \const_\omega \norm{\varphi}_{L^\infty(\omega)} \TrFun^2\left(\mu^\tau_{(i+1)\tau}, \mu^\tau_{i\tau}\right) \comma \end{align*}
	where~$\const_\omega$ actually only depends on the (strictly positive) distance of~$\omega$ from~$\partial \Omega$. Taking the sum over~$i$, we obtain
	\begin{multline*} \abs{\int_\Omega (\rho^\tau_t - \rho^\tau_s) \varphi \dif x - \int_{\lfloor \frac{s}{\tau} \rfloor \tau+\tau}^{\lfloor \frac{t}{\tau} \rfloor \tau+\tau} \int_\Omega \rho_r^\tau(\Delta \varphi - \langle \nabla \varphi, \nabla V \rangle) \dif x \dif r } \le \sum_{i = \lfloor s/\tau \rfloor}^{\lfloor t/\tau \rfloor -1} (I_1^i + I_2^i) \\
		\le \const_\omega \norm{\varphi}_{C^2_0(\omega)} \sum_{i = 0}^{\lfloor t/\tau \rfloor -1} \TrFun^2\left(\mu^\tau_{(i+1)\tau}, \mu^\tau_{i\tau}\right) \fstop \end{multline*}
	At this point,~\eqref{eq:lemma:FP:00} follows from the last estimate and \Cref{lemma:boundSumT}.
	
	\emph{Step 3 (inequality~\eqref{eq:lemma:FP:01}).} Assume that~$2\tau \le \epsilon \le s \le t \le 1/\epsilon$. From~\eqref{eq:lemma:FP:00}, we obtain
	\[ \abs{\int_\Omega (\rho_t^\tau - \rho_s^\tau) \varphi \dif x} \le \const_{\omega, \epsilon} \, \tau \norm{\varphi}_{C^2_0(\omega)} 
	+ \underbrace{\int_{\lfloor \frac{s}{\tau} \rfloor \tau+\tau}^{\lfloor \frac{t}{\tau} \rfloor \tau+\tau} \norm{\rho_r^\tau (\Delta \varphi - \langle \nabla \varphi, \nabla V \rangle)}_{L^1} \dif r}_{\eqqcolon I_3} \fstop \]
	Taking into account \Cref{rmk:massbound} and the estimate~\eqref{eq:lemma:FP:3} of Step 1,
	\begin{align*}
		I_3 &\le \norm{\varphi}_{C^2_0(\omega)} \int_{\lfloor \frac{s}{\tau} \rfloor \tau+\tau}^{\lfloor \frac{t}{\tau} \rfloor \tau+\tau} \left(\norm{\rho_r^\tau}_{L^1}+\norm{\rho_r^\tau \nabla V}_{L^1} \right) \dif r \\
		&\le \const_{\omega} e^{\const \tau} \norm{\varphi}_{C^2_0(\omega)} (t-s+\tau)\left(1 +  \frac{1+t+\tau}{\lfloor s / \tau \rfloor \tau } \right) \\
		&\le \const_{\omega,\epsilon} \norm{\varphi}_{C^2_0(\omega)} (t-s+\tau) \fstop
	\end{align*}
	The inequality~\eqref{eq:lemma:FP:01} easily follows.
\end{proof}

\begin{lemma}[Improved convergence] \label{lemma:improvedconv}
	Assume that the sequence~$\bigl(t \mapsto (\mu_t^\tau)_\Omega\bigr)_\tau$ converges pointwise w.r.t.~$Wb_2$ as~$\tau \to 0$ to a limit~$t \mapsto \rho_t \dif x$. Then, for every~$q \in [1,\frac{d}{d-1})$, the sequence~$(\rho^\tau)_\tau$ converges to~$\rho$ in~$L^1_\mathrm{loc}\bigl((0,\infty); L^q(\Omega) \bigr)$.
\end{lemma}

\begin{proof}
	\emph{Step 1.} Fix~$\epsilon \in (0,1)$ and an open set~$\omega \Subset \Omega$ with $C^1$-regular boundary. As a first step, we shall prove strong convergence of~$(\rho^\tau)_\tau$ in~$L^1\bigl( \epsilon,\epsilon^{-1}; L^{q}(\omega) \bigr)$. The idea is to use a variant of the Aubin--Lions lemma by M.~Dreher and A.~J\"ungel~\cite{DreherJungel12}.
	Consider the Banach spaces
	\[ X \coloneqq W^{1,1}(\omega) \comma \quad B \coloneqq L^{q}(\omega) \comma \quad Y \coloneqq \bigl( C^2_0(\omega) \bigr)^* \comma \] 
	and note that the embeddings~$X \hookrightarrow B$ and~$B \hookrightarrow Y$ are respectively compact (by the Rellich--Kondrachov theorem~\cite[Theorem 9.16]{Brezis11}) and continuous. Inequality~\eqref{eq:lemma:FP:01} in \Cref{lemma:FP} provides one of the two bounds needed to apply \cite[Theorem~1]{DreherJungel12}. The other one, namely
	\[
	\limsup_{\tau \to 0} \norm{\rho^\tau}_{L^1\bigl((\epsilon, \epsilon^{-1}) ; W^{1,1}(\omega) \bigr)} < \infty \comma
	\]
	can be derived from our previous lemmas. Indeed, \Cref{rmk:massbound} provides the bound on the~$L^1\bigl(\epsilon, \epsilon^{-1};L^1(\omega))$ norm, and we have
	\begin{align*}
		\norm{\nabla \rho^\tau_t}_{L^1(\omega)} &\le \const \norm{\sqrt{\rho^\tau_t} \, \nabla\sqrt{ \rho^\tau_t e^V}}_{L^1(\omega)} + \norm{\rho^\tau_t \nabla V}_{L^1(\omega)} \\
		&\le \const \sqrt{\norm{\rho^\tau_t}_{L^1}} \norm{\nabla \sqrt{\rho_t^\tau e^V}}_{L^2} + \norm{\rho_t^\tau}_{L^{p'}(\omega)} \norm{\nabla V}_{L^p(\omega)} \comma
	\end{align*}
	where~$p = p(\omega)$ is given by \Cref{sec:functions}. When~$\tau \le \epsilon$, \Cref{rmk:massbound} and \Cref{lemma:sobolev} yield
	\begin{align*}
		\int_{\epsilon}^{\frac{1}{\epsilon}} \sqrt{\norm{\rho^\tau_t}_{L^1}} \norm{\nabla \sqrt{\rho_t^\tau e^V}}_{L^2} \dif t \le \sqrt{\int_{\epsilon}^{\frac{1}{\epsilon}}  \norm{\rho^\tau_t}_{L^1} \dif t} \sqrt{\int_{\epsilon}^{\frac{1}{\epsilon}}  \norm{\nabla \sqrt{\rho_t^\tau e^V}}_{L^2}^2 \dif t} \le \const_{\epsilon} \fstop
	\end{align*}
	Moreover, since~$p' \in [1,\infty)$ and~$p'(d-2) \le d$, we can apply \Cref{lemma:lebesgueBound} to bound~$\norm{\rho^\tau_t}_{L^{p'}(\omega)}$. To be precise, there is still a small obstruction to applying Dreher and J\"ungel's theorem: it requires~$\rho^\tau$ to be constant on equally sized subintervals of the time domain, i.e.,~$ (\epsilon, \epsilon^{-1})$; instead, here,~$\tau$ and $ (\epsilon^{-1} - \epsilon)$ may even be incommensurable. Nonetheless, it is not difficult to check that the proof in~\cite{DreherJungel12} can be adapted.\footnote{The adaptation is the following. In place of~\cite[Inequality~(7)]{DreherJungel12}, we write, in our notation:\[
		\sum_{i \, : \, \epsilon < i\tau < \epsilon^{-1}} \norm{\rho^\tau_{i\tau} - \rho^\tau_{(i-1)\tau}}_{Y}	\stackrel{\eqref{eq:lemma:FP:01}}{\le} \const_{\omega,\epsilon} \tau \left( \lceil 1/(\epsilon\tau) - 1 \rceil - \lfloor \epsilon/\tau \rfloor  \right) \le \const_{\omega,\epsilon} (\epsilon^{-1}-\epsilon+\tau) \fstop
		\]} In the end, we obtain the convergence of~$\bigl(\rho^\tau \bigr)_\tau$,~\emph{along a subsequence~$(\tau_k)_{k \in \N_0}$}, to some function~$f \colon (\epsilon,\epsilon^{-1}) \times \omega \to \R_+$ in~$L^1\bigl( \epsilon,\epsilon^{-1}; L^{q}(\omega) \bigr)$. Up to extracting a further subsequence, we can also require that convergence holds in~$L^{q}(\omega)$ for~$\Leb^1_{ (\epsilon,\epsilon^{-1}) }$-a.e.~$t$. For any such~$t$, and for any~$\varphi \in C_c(\omega)$, we thus have
	\[
	\int_\omega \varphi f_t \dif x = \lim_{k \to \infty} \int_\omega \varphi \rho^{\tau_k}_t \dif x = \int_\omega \varphi \rho_t \dif x  \comma
	\]
	where the last identity follows from the convergence w.r.t.~$Wb_2$ and~\cite[Proposition 2.7]{FigalliGigli10}. Therefore,~$f_t(x) = \rho_t(x)$ for~$\Leb^{d+1}_{ (\epsilon,\epsilon^{-1})\times \omega}$-a.e.~$(t,x)$, and, \emph{a posteriori}, there was no need to extract subsequences.
	
	\emph{Step 2.} Secondly, we prove that, for every~$\epsilon \in (0,1)$, the sequence~$(\rho^\tau)_\tau$ is Cauchy in the complete space~$L^1 \bigl( \epsilon,\epsilon^{-1} ; L^{q}(\Omega) \bigr)$. Pick an open subset~$\omega \Subset \Omega$ and cover it with a \emph{finite} number of open balls~$\set{A_i}_i$, all compactly contained in~$\Omega$. Additionally choose~$\beta \in (q,\infty)$ with~$\beta(d-2)\le d$. We have
	\[
	\norm{\cdot}_{L^1\bigl(\epsilon, \epsilon^{-1} ; L^{q}(\Omega) \bigr)} \le \sum_i \norm{\cdot}_{L^1\bigl(\epsilon, \epsilon^{-1} ; L^{q}(A_i) \bigr)} + \norm{\cdot}_{L^1\bigl(\epsilon, \epsilon^{-1} ; L^{q}(\Omega \setminus \omega) \bigr)} \comma
	\]
	and, by~H\"older's inequality,
	\[
	\norm{\cdot}_{L^1\bigl(\epsilon, \epsilon^{-1} ; L^{q}(\Omega \setminus \omega) \bigr)} \le \abs{\Omega \setminus \omega}^{\frac{1}{q}-\frac{1}{\beta}} \norm{\cdot}_{L^1\bigl(\epsilon, \epsilon^{-1} ; L^{\beta}(\Omega) \bigr)} \fstop
	\]
	Hence, by Step~1,
	\begin{equation*}
		\limsup_{\tau_1,\tau_2 \to 0} \norm{\rho^{\tau_1} - \rho^{\tau_2}}_{L^1\bigl(\epsilon, \epsilon^{-1} ; L^{q}(\Omega) \bigr)} 
		\le 2\abs{\Omega \setminus \omega}^{\frac{1}{q}-\frac{1}{\beta}} \limsup_{\tau \to 0} \norm{\rho^\tau}_{L^1\bigl(\epsilon, \epsilon^{-1} ; L^{\beta}(\Omega) \bigr)} \fstop
	\end{equation*}	
	Recall \Cref{lemma:lebesgueBound}: we have
	\[
	\limsup_{\tau \to 0} \norm{\rho^\tau}_{L^1\bigl(\epsilon, \epsilon^{-1} ; L^{\beta}(\Omega) \bigr)} \le \const_{\beta} \int_\epsilon^{\epsilon^{-1}} \left(1+\frac{1}{t} \right) \dif t  \le \const_{\beta, \epsilon} \fstop
	\]
	We conclude, by arbitrariness of~$\omega$, the desired Cauchy property.
	
	By Step~1, the limit of~$(\rho^\tau)_\tau$ in~$L^1\bigl(\epsilon, \epsilon^{-1}; L^{q}(\Omega) \bigr)$ must coincide~$\Leb^{d+1}_{(\epsilon,\epsilon^{-1}) \times  \omega}$-a.e. with~$\rho$ for every~$ \omega \Subset \Omega$ open; hence, this limit is precisely~$\rho$ on~$\Omega$. 
\end{proof}

\begin{proof}[Proof of \Cref{prop:FP}]
	Convergence in~$L^1_\mathrm{loc} \bigl( (0,\infty); L^q(\Omega))$ was proven in the previous lemma. Thus, we shall only prove the properties of the limit curve.
	
	\emph{Step 1 (continuity).} Continuity in duality with~$C_c(\Omega)$ follows from \Cref{prop:convergence} and~\cite[Proposition 2.7]{FigalliGigli10}.
	
	\emph{Step 2 (identity~\eqref{eq:fp} for~$s > 0$).} Let~$0 < s \le t$ and let~$\varphi \in C_c^2(\Omega)$. Thanks to the convergences
	\[ \rho_s^\tau \dif x \stackrel{Wb_2}{\to_\tau} \rho_s \dif x \quad \text{and} \quad \rho_t^\tau \dif x \stackrel{Wb_2}{\to_\tau} \rho_t \dif x \comma \]
	we have (see~\cite[Proposition 2.7]{FigalliGigli10})
	\[ \int_\Omega (\rho_t^\tau-\rho_s^\tau) \varphi \dif x \to_\tau \int_\Omega (\rho_t-\rho_s) \varphi \dif x \fstop \]
	Moreover, since every~$p$ as in \Cref{sec:functions} has a conjugate exponent~$p'$ that satisfies~$p' (d-1) < d$, \Cref{lemma:improvedconv} yields
	\[ \int_{\lfloor \frac{s}{\tau} \rfloor \tau+\tau}^{\lfloor \frac{t}{\tau} \rfloor\tau + \tau} \int_\Omega \rho_r^\tau(\Delta \varphi - \langle \nabla \varphi, \nabla V \rangle) \dif x \dif r \to_\tau \int_{s}^{t} \int_\Omega \rho_r (\Delta \varphi - \langle \nabla \varphi, \nabla V \rangle) \dif x \dif r \fstop \]
	Thus,~\eqref{eq:fp} is true by \Cref{lemma:FP}.
	
	\emph{Step 3 (Sobolev regularity and boundary condition).} In analogy with \Cref{rmk:g}, we define
	\[
	g_r^\tau \coloneqq \sqrt{\rho_r^\tau e^V}-e^{\Psi/2} \comma \quad g_r^{\tau,{(\kappa)}} \coloneqq (g_r^\tau-\kappa)_+-(g_r^\tau+\kappa)_- \comma \qquad \tau,\kappa > 0 \comma r \ge 0 \comma
	\]
	and
	\[
	g_r \coloneqq \sqrt{\rho_r e^V}-e^{\Psi/2} \comma \quad g_r^{{(\kappa)}} \coloneqq (g_r-\kappa)_+-(g_r+\kappa)_- \comma \qquad \kappa > 0 \comma r \ge 0 \fstop
	\]
	Recall that, if~$\kappa \ge c (e^{c \tau}-1)$ for an appropriate constant~$c$, and if~$r \ge \tau$, then the function~$g_{r}^{\tau,(\kappa)}$ is compactly supported in~$\Omega$. Let us fix one such~$\kappa$ and~$0 < s < t$. \Cref{lemma:sobolev} implies that the  sequence~$\bigl(g^{\tau,(\kappa)}\bigr)_{\tau}$ is eventually norm-bounded in the space~$L^2\bigl( s,t ; W^{1,2}_0(\Omega) \bigr)$. As a consequence, it admits a subsequence~$\bigl(g^{\tau_k,(\kappa)}\bigr)_{k}$ (possibly dependent on~$s,t,\kappa$) that converges weakly in~$L^2\bigl(s,t ; W^{1,2}_0(\Omega) \bigr)$. Using \Cref{lemma:improvedconv} and Mazur's lemma \cite[Corollary~3.8~\&~Exercise~3.4(.1)]{Brezis11}, one can easily show that this limit indeed coincides with~$g^{(\kappa)}$.
	
	By means of the weak semicontinuity of the~norm, the definition of~$g^{\tau,(\kappa)}$, and \Cref{lemma:sobolev}, we find
	\begin{equation*} \int_s^t \norm{g^{(\kappa)}_r}_{W^{1,2}}^2 \dif r \le \liminf_{k \to \infty} \int_s^t \norm{g^{{\tau_k},(\kappa)}_r}_{W^{1,2}}^2 \dif r \le \liminf_{k \to \infty} \int_s^t \norm{g^{\tau_k}_r}_{W^{1,2}}^2 \dif r \le \const(1+t) \comma \end{equation*}
	and, by arbitrariness of~$s$,
	\[
	\int_0^t \norm{g^{(\kappa)}_r}_{W^{1,2}}^2 \dif r \le \const(1+t)
	\]
	for every~$\kappa,t > 0$. We can thus extract a subsequence~$\bigl( g^{(\kappa_l)} \bigr)_l$ (possibly dependent on~$t$) that converges~weakly in~$L^2\bigl(0,t;W^{1,2}_0(\Omega) \bigr)$. As before, one can check that this limit is~$g$; hence~$g \in L^2\bigl(0,t;W^{1,2}_0(\Omega) \bigr)$ with
	\begin{equation}
		\int_0^t \norm{g_r}_{W^{1,2}}^2 \dif r \le \const(1+t)
	\end{equation}
	
	\emph{Step 4 (integrability, and~\eqref{eq:fp} for~$s=0$).} Fix an open set~$\omega \Subset \Omega$. Let~$p = p(\omega) > d$ be as in \Cref{sec:functions} and let~$p'$ be its conjugate exponent. Since~$g \in L^2_\mathrm{loc}\bigl([0,\infty) ; W^{1,2}_0(\Omega) \bigr)$, the Sobolev embedding theorem implies $g \in L^2_\mathrm{loc}\bigl( [0,\infty); L^{2p'}(\Omega) \bigr)$. Given that~$V \in L^\infty(\Omega)$, we obtain~$\rho \in L^1_\mathrm{loc}\bigl([0,\infty) ; L^{p'}(\Omega)\bigr)$. In particular,~$t \mapsto \int_\omega \rho_t \dif x$ and~$t \mapsto \int_{\omega} \abs{\nabla V} \rho_t \dif x$ are both locally integrable on~$[0,\infty)$. Given~$\varphi \in C^2_c(\omega)$, the identity~\eqref{eq:fp} for~$s=0$ thus follows from the one with~$s > 0$ by taking the limit~$s \downarrow 0$: on the one side,
	\[
	\lim_{s \downarrow 0} \int_{\Omega} \rho_s \varphi \dif x = \int_{\Omega} \rho_0 \varphi \dif x
	\]
	by continuity in duality with~$C_c(\Omega)$; on the other,
	\[
	\lim_{s \downarrow 0} \int_{s}^{t} \int_\Omega \rho_r (\Delta \varphi - \langle \nabla \varphi, \nabla V \rangle) \dif x \dif r = \int_{0}^{t} \int_\Omega \rho_r (\Delta \varphi - \langle \nabla \varphi, \nabla V \rangle) \dif x \dif r
	\]
	by the dominated convergence theorem.
\end{proof}

\section{Slope formula in dimension~$d=1$} \label{sec:slope}

In this section, we only work in dimension~$d=1$ and we take~$\Omega = ( 0,1 )$. %
Recall (\Cref{prop:distance}) that, in this setting,~$\TrDis$ is a metric on~$\cone$. Our purpose is to find an explicit formula for the descending slope~$\slope{\Ent}{\TrDis}$ and to derive \Cref{thm:main30} as a corollary. Specifically, the main result of this section is the following.

\begin{proposition} \label{prop:slope}
	Assume that~$V \in W^{1,2}(\Omega)$. Take~$\mu \in \cone$ such that~$\Ent(\mu) < \infty$ and let~$\rho$ be the density of~$\mu_\Omega$. Then,
	\begin{equation} \label{eq:inequalityFisher}
		\slope{\Ent}{\TrDis}^2 (\mu) = \begin{cases}
			\displaystyle 4 \int_\Omega \left( \partial_x \sqrt{\rho e^V} \right)^2 e^{-V} \dif x &\text{if } \sqrt{\rho e^V}-e^{\Psi/2} \in W^{1,2}_0(\Omega) \comma \\
			\infty &\text{otherwise.} 
		\end{cases}
	\end{equation}
\end{proposition}

\begin{remark}
	In the current setting, i.e.,~$\Omega = ( 0,1 )$ and~$V \in W^{1,2}(\Omega)$, the function~$V$ is H\"older continuous; thus it extends to the boundary~$\partial \Omega = \set{0,1}$. When~$\sqrt{\rho e^V} \in W^{1,2}(\Omega)$, the function~$\rho$ belongs to~$W^{1,2}(\Omega)$, is continuous, and extends to the boundary as well.
\end{remark}

\begin{remark} \label{rkm:slopeLSC}
	The functional
	\begin{equation} \label{eq:slopeLSC}
		W^{1,2}(\Omega) \ni f \longmapsto \begin{cases}
			\displaystyle 4 \int_\Omega \left( \partial_x f \right)^2 e^{-V} \dif x &\text{if } f-e^{\Psi/2} \in W^{1,2}_0(\Omega) \comma \\
			\infty &\text{if } f-e^{\Psi/2} \in W^{1,2}(\Omega) \setminus W^{1,2}_0(\Omega) \fstop
		\end{cases}
	\end{equation}
	is particularly well-behaved: it is convex, strongly continuous, weakly lower semicontinuous, and has weakly compact sublevels. As a consequence,~$\slope{\Ent}{\TrDis}$ turns out to be lower semicontinuous w.r.t.~$\TrDis$. Indeed, assume that~$\mu^n \stackrel{\TrDis}{\to} \mu$ and $\sup_n \slope{\Ent}{\TrDis}(\mu^n) < \infty$. Let~$\rho^n$ be the density of~$\mu^n_\Omega$. Then the functions~$f_n \coloneqq \sqrt{\rho^n e^V}$ converge, up to subsequences, weakly in~$W^{1,2}(\Omega)$ and---by the Rellich--Kondrachov theorem~\cite[Theorem~8.8]{Brezis11}---strongly in~$C(\overline \Omega)$ to a function~$f$ such that~$f-e^{\Psi/2} \in W^{1,2}_0(\Omega)$ and
	\[
	4 \int_\Omega (\partial_x f)^2 e^{-V} \dif x \le \liminf_{n \to \infty} \slope{\Ent}{\TrDis}^2(\mu^n) \fstop
	\]
	Additionally,~$\rho^n = f_n^2 e^{-V} \to f^2 e^{-V}$ in~$C(\overline \Omega)$, hence~$\mu_\Omega = f^2 e^{-V} \dif x$ (we use~\eqref{eq:distIneq} and~\cite[Proposition~2.7]{FigalliGigli10}).
\end{remark}

While~\eqref{eq:inequalityFisher} reminds the classical slope of the relative entropy (i.e.,~the relative Fisher information), the crucial difference is in the role of the boundary condition: if~$\rho$ does not satisfy the correct one, the slope is infinite.

We are going to prove the two opposite inequalities in~\eqref{eq:inequalityFisher} separately. Proving~$\ge$ is easier: for the case where~$\sqrt{\rho e^V} - e^{\Psi/2} \in W^{1,2}_0$, it amounts to taking small variations of~$\mu$ in an arbitrary direction; for the other case, it suffices to find appropriate sequences that make the difference quotient diverge. To handle the opposite inequality, we have to bound~$\bigl( \Ent(\mu) - \Ent(\tilde \mu) \bigr)_+$ from above for every sufficiently close measure~$\tilde \mu \in \cone$. Classical proofs (e.g.,~\cite[Theorem 15.25]{AmbrosioBrueSemola21} or~\cite[Theorem 10.4.6]{AmbrosioGigliSavare08}) take advantage of geodesic convexity of the functional, which we do not to have; see \Cref{sec:geod}. One of the perks of geodesic convexity is that it automatically ensures lower semicontinuity of the descending slope, which in turn allows to assume stronger regularity on~$\mu$ and then argue by approximation. To overcome this problem, we combine different ideas on different parts of~$\mu$ and~$\tilde \mu$. Away from the boundary~$\partial \Omega = \set{0,1}$, the transport plans move absolutely continuous measures to absolutely continuous measures. The Jacobian equation (change of variables formula) relates the two densities and makes the computations rather easy. Estimating the contribution of the parts of~$\mu,\tilde \mu$ closest to the boundary is more technical: we need to exploit the boundary condition and the Sobolev regularity of the functions~$\rho$,~$\log \rho$, and~$V$. Note, indeed, that since the boundary condition is positive, also~$\log \rho$ has a square-integrable derivative in a neighborhood of~$\partial \Omega$.

To be in dimension~$d=1$ is necessary for~$\TrDis$ to be a distance, but is also extremely convenient because optimal transport maps are monotone and~$W^{1,2}$-regular functions are H\"older continuous. For these reasons, it seems difficult (but maybe still possible) to adapt our proof of \Cref{prop:slope} for an analogue of \Cref{thm:main30} in higher dimension.

We first prove a variant of the Lebesgue differentiation theorem that is needed for the subsequent proof of \Cref{prop:slope}. We prove \Cref{thm:main30} at the end of the section.

\begin{lemma} \label{lemma:compareaverage}
	Let~$(\gamma^n)_{n \in \N_0}$ be a sequence of nonnegative Borel measures on~$\Omega \times \ovom$ such that~$ \lim_{n \to \infty} \cost(\gamma^n) = 0 $.
	Additionally assume that~$\pi^1_\#  \gamma^n$ is absolutely continuous for every~$n \in  \N_0$, with a density that is uniformly bounded in~$L^\infty(\Omega)$. Then, for every~$f \in L^2(\Omega)$,
	\begin{equation} \lim_{n \to \infty} \int \left( \fint_x^y \bigr( f(z) - f(x) \bigr) \dif z \right)^2 \dif \,  \gamma^n (x,y) = 0 \fstop \end{equation}
\end{lemma}

\begin{proof}
	Denote by~$\rho^n$ the density of~$\pi^1_\#  \gamma^n$. Let~$g \colon \Omega \to \R$ be Lipschitz continuous. For every~$n \in \N_0$, we have
	\begin{align*}
		I_n &\coloneqq \int  \left( \fint_x^y \bigr( f(z) - f(x) \bigr) \dif z \right)^2  \dif   \gamma^n  \\ &\le 3\int \left( \fint_x^y (f-g) \dif z \right)^2  \dif   \gamma^n
		+ 3 \int  \left( \fint_x^y g \dif z - g(x)  \right)^2 \dif   \gamma^n  \\
		&\quad + 3\int_{\Omega} ( g-f)^2 \rho^n \dif x \fstop
	\end{align*}
	Consider the Hardy--Littlewood maximal function of (the extension to~$\R$ of) $f-g$, that is,
	\[ (f-g)^*(x) \coloneqq \sup_{r > 0} \frac{1}{2r} \int_{\max\set{x-r,0}}^{\min\set{x+r,1}} \abs{f-g} \dif z \comma \qquad x \in \R \fstop \]
	By the (strong) Hardy--Littlewood maximal inequality,
	\begin{multline*} \int \left( \fint_x^y (f-g) \dif z \right)^2 \dif \,  \gamma^n \le 4 \int \bigl( (f-g)^*(x) \bigr)^2 \dif \,  \gamma^n = 4 \int_\Omega \bigl( (f-g)^* \bigr)^2 \rho^n \dif x \\
		\le 4 \sup_n \norm{\rho^n}_{L^\infty} \norm{(f-g)^*}_{L^2(\R)}^2 \le \const \sup_n \norm{\rho^n}_{L^\infty} \norm{f-g}_{L^2}^2 \fstop \end{multline*}
	The Lipschitz-continuity of~$g$ gives
	\begin{equation*} \int \left( \fint_x^y g \dif z - g(x)  \right)^2 \dif \,  \gamma^n \le (\Lip g)^2 \int (x-y)^2 \dif \,  \gamma^n \le (\Lip g)^2 \cost( \gamma^n) \comma \end{equation*}
	and, moreover, we have
	\[ \int_\Omega (g-f)^2 \rho^n \dif x \le \norm{\rho^n}_{L^\infty} \norm{f-g}_{L^2}^2 \fstop \]
	In conclusion,
	\[ I_n \le \const \sup_n \norm{\rho^n}_{L^\infty} \norm{f-g}^2_{L^2} + 3(\Lip g)^2 \cost( \gamma^n) \fstop \]
	After passing to the limit superior in~$n$, we conclude by arbitrariness of~$g$.
\end{proof}

\begin{proof}[Proof of \Cref{prop:slope}]
	We omit the subscript~$_\TrDis$ in~$\slope{\Ent}{\TrDis}$ throughout the proof.
	
	\emph{Step 1 (inequality~$\ge$, finite case).}  Assume  that~$\sqrt{\rho e^V} - e^{\Psi/2} \in W^{1,2}_0$; hence, in particular,~$\rho \in L^\infty(\Omega)$. Let~$w \colon \Omega \to \R$  be~$C^\infty$-regular with compact support (and not identically equal to~$0$), and, for~$\epsilon > 0$, define~$R_\epsilon(x) \coloneqq x + \epsilon w(x)$. Set~$\mu^\epsilon \coloneqq (R_\epsilon)_\# \mu$ and~$\gamma^\epsilon \coloneqq (\Id,R_\epsilon)_\# \mu$. When~$\epsilon$ is sufficiently small,~$\mu^\epsilon \in \cone$ and~$\gamma^\epsilon \in \Adm_{\TrDis}(\mu,\mu^\epsilon)$. Therefore, arguing as in the proof of \Cref{lem:sobReg},
	\[
	\lim_{\epsilon \to 0^+} \frac{\Ent(\mu) - \Ent(\mu^\epsilon)}{\epsilon} = \int_\Omega (\partial_x w-w \, \partial_x V )\rho \dif x \fstop
	\] Thus,
	\begin{align*} \label{eq:prop:slope:1}
		\begin{split} \int_\Omega (\partial_x w-w \, \partial_x V )\rho \dif x  \le \slopesmall{\Ent}{} (\mu) \liminf_{\epsilon \downarrow 0} \frac{\sqrt{\cost(\gamma^\epsilon)}}{\epsilon} \le \slopesmall{\Ent}{}(\mu) \norm{w}_{L^2(\rho)} \comma
		\end{split}
	\end{align*}
	and we conclude that
	\[
	\int_\Omega \abs{\partial_x \sqrt{\rho e^V}}^2 e^{-V} \dif x \le \frac{1}{4} \slopesmall{\Ent}{}^2  (\mu) \fstop
	\]
	
	\emph{Step 2 (inequality~$\ge$, infinite case).} The case~$\sqrt{\rho e^V}\not\in W^{1,2}(\Omega)$ is trivial. Thus, let us assume now that~$\sqrt{\rho e^V}\in W^{1,2}(\Omega)$ with $\Tr \rho \neq \Tr e^{\Psi - V}$. Without loss of generality, we may consider the case where~$\rho(0) \neq e^{\Psi(0)-V(0)}$. If~$\rho(0) > e^{\Psi(0)-V(0)}$, for~$\epsilon > 0$ define 
	\begin{align*} \mu^\epsilon &\coloneqq \mu - \epsilon\mu_{ ( 0,\epsilon^2 )} + \left( \epsilon \int_0^{\epsilon^2} \rho \dif x \right) \delta_0 \in \cone \comma \\ \gamma^\epsilon &\coloneqq \epsilon \mu_{( 0,\epsilon^2 )} \otimes \delta_0 + (\Id,\Id)_\# (\mu_\Omega - \epsilon \mu_{(0,\epsilon^2)}) \in \Adm_{\TrDis}(\mu,\mu^\epsilon) \fstop \end{align*}
	Since all the functions involved are continuous up to the boundary, we get
	\begin{align*} \Ent(\mu) - \Ent(\mu^\epsilon) &= \int_0^{\epsilon^2} \left(\rho \log \rho - (1-\epsilon)\rho \log \bigl((1-\epsilon)\rho \bigr) + \epsilon \bigl(V-1-\Psi(0)\bigr)\rho \right) \dif x \\
		&\sim_{\epsilon \downarrow 0} \epsilon^3 \bigl(\log \rho(0) + V(0) - \Psi(0)\bigr) \rho(0) \fstop \end{align*}
	On the other hand,
	\[ \TrDis(\mu,\mu^\epsilon) \le  \sqrt{\cost(\gamma^\epsilon)} =  \sqrt{\epsilon \int_0^{\epsilon^2} x^2 \rho \dif x} \le \sqrt{\epsilon^{5} \int_0^{\epsilon^2} \rho \dif x} \sim_{\epsilon \downarrow 0} \epsilon^\frac{7}{2} \sqrt{\rho(0)} \comma \]
	from which we find
	\begin{align*} \slopesmall{\Ent}{}(\mu) &\ge \limsup_{\epsilon \downarrow 0} \frac{\Ent(\mu) - \Ent(\mu^\epsilon)}{\TrDis(\mu,\mu^\epsilon)} \\
		&\ge \underbrace{\sqrt{\rho(0)} \bigl(\log \rho(0)+V(0)-\Psi(0)\bigr)}_{>0} \limsup_{\epsilon \downarrow 0} \epsilon^{-\frac{1}{2}} = \infty \fstop \end{align*}
	
	If, instead,~$\rho(0) < e^{\Psi(0) - V(0)}$, we consider, for~$\epsilon > 0$,
	\begin{equation*} \mu^\epsilon \coloneqq \mu + \epsilon \Leb^1_{( 0,\epsilon^2 )} - \epsilon^3 \delta_0 \in \cone \comma \quad
		\gamma^\epsilon \coloneqq \epsilon \delta_0 \otimes \Leb^1_{( 0,\epsilon^2 )}+(\Id,\Id)_\# \mu_\Omega \in \Adm_{\TrDis}(\mu,\mu^\epsilon) \fstop \end{equation*}
	and conclude with similar computations as before.
	
	\emph{Step 3 (preliminaries for~$\le$).} We suppose again that~$\sqrt{\rho e^V}-e^{\Psi/2} \in W^{1,2}_0(\Omega)$. In particular, there exist~$\bar \lambda, \bar \epsilon > 0$ such that
	\[
	\rho|_{[0,\bar \epsilon] \cup [1-\bar \epsilon,1]} > \bar \lambda \fstop
	\]
	Let us take a sequence~$(\mu^n)_{n \in \N_0}$ that converges to~$\mu$ w.r.t.~$\TrDis$, with~$\Ent(\mu^n) < \Ent(\mu)$ for every~$n$. We aim to prove that
	\[
	\limsup_{n \to \infty} \frac{\Ent(\mu) - \Ent(\mu^n)}{\TrDis(\mu,\mu^n)} \le 2\sqrt{\int_\Omega \left(\partial_x \sqrt{\rho e^V} \right)^2 e^{-V} \dif x} \fstop
	\]
	
	For every~$n \in \N_0$, we write:
	\begin{itemize}
		\item $\rho^n$ for the density of~$\mu^n_\Omega$;
		\item $\gamma^n$ for some (arbitrarily chosen) $\TrDis$-optimal transport plan between~$\mu$ and~$\mu^n$ such that the diagonal~$\Delta$ of~$\partial \Omega \times \partial \Omega$ (i.e.,~the set with the two points~$(0,0)$ and~$(1,1)$) is~$\gamma^n$-negligible;
		\item $T_n,S_n$ for maps such that~$(\gamma^n)_\Omega^\ovom = (\Id,T_n)_\# \mu_\Omega$ and~$(\gamma^n)_\ovom^\Omega = (S_n,\Id)_\# \mu^n_\Omega$. We can and will assume that these two maps are nondecreasing, hence~$\Leb^1_\Omega$-a.e.~differentiable;
		\item $a_n,b_n \in \ovom = [0,1]$ for the infimum and supremum of the set~$T_n^{-1}(\Omega)$, respectively. Note that, since~$T_n$ is monotone,~$T_n^{-1}(\Omega)$ is an interval. Conventionally, we set~$a_n=1$ and~$b_n = 0$ if~$T_n^{-1}(\Omega) = \emptyset$.
	\end{itemize}
	Observe that, since~$(0,a_n)\subseteq T_n^{-1}(\set{0,1})$, we have
	\[
	\TrDis^2(\mu,\mu^n) \ge \int_0^{a_n} \min\set{x,1-x}^2 \rho \dif x \ge \bar \lambda \int_0^{\min\set{a_n,\bar \epsilon}} x^2 \dif x = \frac{\bar \lambda}{3} \min\set{a_n,\bar \epsilon}^3 \fstop
	\]
	In particular,
	\begin{equation} \label{eq:limsupanbn}
		\limsup_{n \to \infty} \frac{a_n^3}{\TrDis^2(\mu,\mu^n)} < \infty \text{ and, similarly,~}\limsup_{n \to \infty} \frac{(1-b_n)^3}{\TrDis^2(\mu,\mu^n)} < \infty \semicolon
	\end{equation}
	thus, up to taking subsequences, we may and will assume that~$a_n < \bar \epsilon < 1-\bar \epsilon < b_n$ for every~$n$. In particular, $(\gamma^n)_\Omega^\Omega \neq 0$ and~$\Leb^1_{(0,a_n)\cup(b_n,1)} \ll \mu_{(0,a_n)\cup(b_n,1)}$.  
	Furthermore, since~$\gamma^n$ is $W_2$-optimal between its marginals (cf. \Cref{prop:trPlans}), it is concentrated on a monotone set~$\Gamma_n$. This implies that~$\gamma(0,1)$ and~$\gamma(1,0)$ equal~$0$ as soon as~$\gamma_\Omega^{\Omega} \neq 0$. Combining this observation with the fact that~$\Delta$ is~$\gamma$-negligible, we infer that~$\gamma_{\partial \Omega}^{\partial \Omega} = 0$. By the same argument,~$T|_{(b_n,1)} \equiv 1$ and~$T|_{(0,a_n)} \equiv 0$.
	
	Another assumption that we can and will make is
	\begin{equation} \label{eq:prop:slope:1bis} \rho^n|_{S_n^{-1}(\partial \Omega)} \le \Lambda \coloneqq \left(\sup_{\partial \Omega} e^\Psi\right) \cdot \left( \sup_\Omega e^{-V} \right) \fstop \end{equation}
	Indeed, if this is not the case, we can consider the new measures
	\begin{align*} \tilde \gamma^n &\coloneqq \gamma^n - (S_n,\Id)_\# \left( \rho^n|_{S_n^{-1}(\partial \Omega)} - \Lambda \right)_+ \Leb^1_\Omega \comma \\
		\tilde \mu^n &\coloneqq \mu - \pi^1_\# (\tilde \gamma^n) + \pi^2_\# (\tilde \gamma^n) \in \cone \comma \end{align*}
	and notice that~$\tilde \gamma^n \in \Adm_{\TrDis}(\mu, \tilde \mu^n)$. We have
	\begin{multline*} \Ent(\tilde \mu^n) - \Ent(\mu^n) = \int_{S_n^{-1}(\partial \Omega) \cap \set{\rho^n > \Lambda} } \Lambda(\log \Lambda + V - 1 - \Psi \circ S_n) \dif x \\- \int_{S_n^{-1}(\partial \Omega) \cap \set{\rho^n > \Lambda} } \rho^n(\log \rho^n + V - 1 - \Psi \circ S_n) \dif x \comma \end{multline*}
	and, because of the definition of~$\Lambda$, we obtain~$\Ent(\tilde \mu^n) \le \Ent(\mu^n)$. At the same time,~$ \TrDis(\mu, \tilde \mu^n) \le \TrDis(\mu, \mu^n) $
	because~$\tilde \gamma^n \le \gamma^n$. This concludes the proof of the claim that we can assume~\eqref{eq:prop:slope:1bis}.
	
	\emph{Step 4 (inequality~$\le$).} By \Cref{prop:trPlans},~$(\gamma^n)_\Omega^\Omega$ is a $W_2$-optimal transport plan between its marginals~$\rho \Leb^1_{T_n^{-1}(\Omega)}$ and~$\rho^n \Leb^1_{S_n^{-1}(\Omega)}$, and it is induced by the map~$T_n$. Hence, by~\cite[Theorem~7.3]{AmbrosioBrueSemola21}, the Jacobian equation
	\begin{equation} \label{eq:prop:slope:2}  \left(\rho^n|_{S_n^{-1}(\Omega)} \circ T_n \right) \cdot \partial_x T_n = \rho \end{equation}
	holds~$\rho \Leb^1_{T_n^{-1}(\Omega)}$-a.e.
	Consequently, we have the chain of identities
	\begin{align} \label{eq:prop:slope:4}
		\begin{split}
			\int_{S_n^{-1}(\Omega)}  (\log \rho^n + V  -1 ) \rho^n \dif x &= \int (\log \rho^n + V -1) \dif \pi^2_\# (\gamma^n)^{\Omega}_\Omega  \\
			&= \int_{T_n^{-1}(\Omega)} \bigl((\log \rho^n + V -1 ) \circ T_n \bigr) \rho \dif x \\
			& \stackrel{\eqref{eq:prop:slope:2}}{=} \int_{T_n^{-1}(\Omega)} \left( \log \rho - \log (\partial_x T_n) + V \circ T_n - 1\right) \rho \dif x \fstop
		\end{split}
	\end{align}
	Thus, we can decompose the difference~$\Ent(\mu) - \Ent(\mu^n)$ as
	\begin{align}  \label{eq:prop:slope:5}
		\begin{split}
			\Ent(\mu) - \Ent(\mu^n) %
			&\stackrel{\eqref{eq:prop:slope:4}}{=} \int_{T_n^{-1}(\Omega)} \bigl(\log (\partial_x T_n) + V - V \circ T_n\bigr) \rho \dif x + (\mu-\mu^n)_{\partial \Omega}(\Psi) \\
			&\quad+ \int_{T_n^{-1}(\partial \Omega)} (\log \rho + V - 1) \rho \dif x - \int_{S_n^{-1}(\partial \Omega)} (\log \rho^n + V - 1) \rho^n \dif x \fstop
		\end{split}
	\end{align}
	Let us focus on the integral on~$T_n^{-1}(\Omega)$. By making the estimate~$\log (\partial_x T_n) \le \partial_x T_n-1$ and using the properties of the Riemann--Stieltjes integral, we obtain
	\begin{align} \label{eq:prop:slope:6} \begin{split}
			&\int_{T_n^{-1}(\Omega)} \log (\partial_x T_n) \rho \dif x \le \int_{T_n^{-1}(\Omega)}  (\partial_x T_n -1) \rho \dif x = \int_{a_n}^{b_n} (\partial_x T_n) \rho \dif x - \int_{a_n}^{b_n} \rho \dif x \\
			&\quad \le \lim_{\epsilon \downarrow 0} \int_{a_n+\epsilon}^{b_n-\epsilon} \rho \dif T_n - b_n\rho(b_n) + a_n\rho(a_n) + \int_{a_n}^{b_n} x \partial_x \rho \dif x
			\\
			&\quad = (T(b_n^-)-b_n)\rho(b_n) - (T(a_n^+)-a_n) \rho(a_n)  - \int_{a_n}^{b_n} (T_n-\Id) \partial_x \rho \dif x \comma
		\end{split}
	\end{align}
	where we employ the notation~$T(a_n^+) \coloneqq \lim_{\epsilon \downarrow 0} T(a_n+\epsilon)$, and similarly with~$T(b_n^-)$.
	
	Let~$f \coloneqq \partial_x V$. By the fundamental theorem of calculus,
	\[
	\int_{T_n^{-1}(\Omega)} (V-V\circ T_n)\rho \dif x = \int_{a_n}^{b_n} \left( \int_{T_n(x)}^x f(z) \dif z \right) \rho \dif x \fstop
	\]
	By adding and subtracting~$f(x)$, we get
	\begin{equation} \label{eq:prop:slope:7}
		\begin{aligned}
			&\int_{T_n^{-1}(\Omega)} (V-V\circ T_n)\rho \dif x \\
			&\quad = \int_{a_n}^{b_n} f(x) \left( \int_{T_n(x)}^x  \dif z \right) \rho \dif x +  \int_{a_n}^{b_n} \left( \int_{T_n(x)}^x \bigl( f(z)-f(x)\bigr) \dif z \right) \rho \dif x \\
			&\quad= -\int_{a_n}^{b_n} (T_n-\Id) \rho \, f \dif x + \int_{a_n}^{b_n} \left(\int_{T_n(x)}^x \bigl(f(z)-f(x)\bigr) \dif z\right) \rho \dif x \fstop
	\end{aligned} \end{equation}
	At this point, we observe that, by H\"older's inequality and \Cref{lemma:compareaverage} (applied to the restriction~$(\gamma^n)_\Omega^\Omega$), the last double integral is negligible, i.e., it is of the order $o_n\bigl(\TrDis(\mu,\mu^n)\bigr)$.
	
	To handle the rest of~\eqref{eq:prop:slope:5}, we exploit the convexity of~$l \mapsto l \log l$ and write
	\begin{equation} \label{eq:prop:slope:8} %
		-\int_{S_n^{-1}(\partial \Omega)} (\log \rho^n + V - 1) \rho^n \dif x \le -\int_{S_n^{-1}(\partial \Omega)} (\log \rho + V) \rho^n \dif x + \int_{S_n^{-1}(\partial \Omega) \cap \set{\rho^n > 0} } \! \rho \dif x \fstop %
	\end{equation}
	By Condition~\ref{(3)} in \Cref{def:TrDis} and the boundary condition of~$\rho$,
	\begin{equation} \label{eq:prop:slope:9}
		(\mu-\mu^n)_{\partial \Omega}(\Psi) = \int (\log \rho + V) \dif \, \left( \pi^1_\# (\gamma^n)_{\partial \Omega}^\ovom - \pi^2_\# (\gamma^n)_\ovom^{\partial \Omega} \right) \fstop
	\end{equation}
	
	In summary, recalling that~$(\gamma^n)_{\partial \Omega}^{\partial \Omega} = 0$, from~\eqref{eq:prop:slope:5},~\eqref{eq:prop:slope:6},~\eqref{eq:prop:slope:7},~\eqref{eq:prop:slope:8}, and~\eqref{eq:prop:slope:9} follows the inequality
	\begin{align} \label{eq:prop:slope:10}
		\begin{split}
			\Ent(\mu) - \Ent(\mu^n) &\le o_{n} \left(\TrDis(\mu,\mu^n) \right) \underbrace{-\int_{a_n}^{b_n} (T_n-\Id) (\partial_x \rho+\rho\partial_x V) \dif x}_{\eqqcolon L^n_1} \\
			&\quad \underbrace{+ \int (\log \rho + V) \dif \left( \pi^1_\# \bigl( \gamma^n - (\gamma^n)_{\Omega}^\Omega\bigr) - \pi^2_\# \bigl( \gamma^n - (\gamma^n)_{\Omega}^\Omega \bigr) \right)}_{\eqqcolon L^n_2} \\
			&\quad \underbrace{+ \bigl(T(b_n^-)-b_n\bigr)\rho(b_n)+\int_{S_n^{-1}(1)\cap \set{\rho^n > 0}} \rho \dif x - \int_{T_n^{-1}(1)} \rho \dif x}_{\eqqcolon L^n_3} \\
			&\quad \underbrace{-  \bigl(T(a_n^+)-a_n\bigr)\rho(a_n) + \int_{S_n^{-1}(0)\cap \set{\rho^n > 0}} \rho \dif x - \int_{T_n^{-1}(0)} \rho \dif x}_{\eqqcolon L^n_4} \fstop
		\end{split}
	\end{align}
	
	We claim that the last three lines in~\eqref{eq:prop:slope:10}, i.e.,~$L^n_2$,~$L^n_3$ and~$L^n_4$, are bounded from above by negligible quantities, of the order~$o_n\left(\TrDis(\mu,\mu^n)\right)$. Let us start with~$L_3^n$. Since every left-neighborhood of~$1$ is \emph{not}~$\mu_\Omega$-negligible,
	\[
	\sup \set{x \in \Omega \, \colon \, (x,T_n(x)) \in \Gamma_n} =1 \comma 
	\]
	which, together with the monotonicity of~$\Gamma_n$, implies
	\begin{equation}
		\label{eq:essinf}
		T_n(1^-) \le \mu^n_\Omega\essinf S^{-1}(1) \fstop
	\end{equation}
	We now distinguish two cases: either~$b_n < 1$ or~$b_n = 1$. If~$b_n < 1$, given that~$T_n|_{(b_n,1)} \equiv 1$, the set~$S^{-1}(1)$ is~$\mu^n_\Omega$-negligible by~\eqref{eq:essinf}. Thus
	\begin{align*} L^n_3 &\le \int_{b_n}^1 \bigl( \rho(b_n)-\rho(x) \bigr) \dif x = -\int_{b_n}^1 \left(\int_{b_n}^x \partial_x \rho \dif z\right) \dif x \\
		&\le \sqrt{\int_{b_n}^1 \abs{x-b_n}^2 \dif x } \sqrt{ \int_{b_n}^1 \left( \fint_{b_n}^x \partial_x \rho \dif z \right)^2 \dif x } \\
		&\stackrel{\eqref{eq:limsupanbn}}{=} O_n\bigl(\TrDis(\mu,\mu^n)\bigr) \sqrt{ \int_{b_n}^1 \left( \fint_{b_n}^x \partial_x \rho \dif z \right)^2 \dif x} \fstop
	\end{align*}
	Knowing that~$\rho \in W^{1,2}(\Omega)$ and that~$b_n \to_n 1$, it can be easily proven with Hardy's inequality that the last square root tends to~$0$ as~$n \to \infty$.
	
	Assume now that~$b_n = 1$. This time, Inequality~\eqref{eq:essinf} yields
	\[ L^n_3 \le (T_n(1^-)-1)\rho(1) + \int_{T_n(1^-)}^1 \rho \dif x = \int_{T_n(1^-)}^1 \bigl( \rho(x)-\rho(1) \bigr) \dif x \fstop \]
	We conclude as in the case~$b_n < 1$, because the computations that led to~\eqref{eq:limsupanbn} can be easily adapted to show that~$(1-T_n(1^-))^3 = O_n\bigl( \TrDis^2(\mu,\mu^n)\bigr)$. Indeed, the monotonicity of~$T_n$ gives
	\[
	\TrDis^2(\mu,\mu^n) \ge \int_{T_n(1^-)}^1 \bigl(x-T_n(x)\bigr)^2 \rho(x) \dif x \ge \bar \lambda \int_{\max\set{1-\bar \epsilon,T_n(1^-)}}^1 \bigl(x-T_n(1^-)\bigr)^2 \dif x \fstop
	\]
	The proof for~$L^n_4$ is similar to that for~$L^n_3$.
	
	Let us now deal with the term~$L^n_2$:
	\[
	L^n_2 = \int \bigl(\log \rho(x)+V(x)-\log \rho(y)-V(y)\bigr) \dif \, \bigl((\gamma^n)_\Omega^{\partial \Omega}+(\gamma^n)_{\partial \Omega}^\Omega \bigr) \fstop \]
	Define the square-integrable function
	\[ g \coloneqq \begin{cases} 
		\frac{\partial_x \rho}{\rho}+\partial_x V &\text{on } (0,\bar \epsilon) \cup (1-\bar \epsilon,1) \comma \\
		0 &\text{otherwise.}
	\end{cases} \]
	Since~$\gamma_\Omega^{\set{1}}$ is concentrated on~$(b_n,1) \times \set{1}$, and~$\gamma_{\set{1}}^\Omega$ is concentraded on~$\set{1} \times (T_n(1^-),1)$, as soon as~$n$ is large enough for~$b_n$ and~$T_n(1^-)$ to be greater than~$1-\bar \epsilon$, we have the equality
	\begin{equation*}
		\bigl(\log \rho(x)+V(x)-\log \rho(y)-V(y)\bigr) = \int_y^x g \dif z \quad \text{for }\bigl((\gamma^n)_\Omega^{\set{1}} + (\gamma^n)^\Omega_{\set{1}}\bigr)\text{-a.e.~}(x,y) \fstop
	\end{equation*} Moreover,
	\[
	\int \left( \int_y^x g \dif z \right) \dif \,  (\gamma^n)_\Omega^{\set{1}} \le \TrDis(\mu,\mu^n) \sqrt{ \int_{b_n}^1 \left( \fint_x^1 g \dif z \right)^2 \underbrace{\rho}_{\le \norm{\rho}_{L^\infty}} \dif x } \comma
	\]
	and
	\[
	\int \left( \int_y^x g \dif z \right) \dif \,  (\gamma^n)^\Omega_{\set{1}} \le \TrDis(\mu,\mu^n) \sqrt{ \int^1_{T_n(1^-)} \left( \fint_x^1 g \dif z \right)^2 \underbrace{\rho^n|_{S_n^{-1}(1)}}_{\le \Lambda} \dif x } \fstop
	\]
	In both cases, since~$b_n$ and~$T_n(1^-)$ tend to~$1$ as~$n \to \infty$, and~$g \in L^2(\Omega)$, the square roots are infinitesimal. The same argument can be easily applied at~$0$ (i.e.~for the integrals w.r.t.~$(\gamma^n)_\Omega^{\set{0}}$ and~$(\gamma^n)^\Omega_{\set{0}}$), and this brings us to the conclusion that~$L_2^n$ is negligible.
	
	In the end,~\eqref{eq:prop:slope:10} reduces to
	\begin{align*}
		\Ent(\mu) - \Ent(\mu^n) &\le -\int_{a_n}^{b_n} (T_n-\Id)(\partial_x \rho+\rho \, \partial_x V) \dif x + o_{n}\left( \TrDis(\mu,\mu^n) \right) \\
		&\le \TrDis(\mu,\mu^n)\sqrt{\int_\Omega \left(\frac{\partial_x \rho}{\sqrt{\rho}} + \sqrt{\rho} \, \partial_x V \right)^2 \dif x + o_{n}(1) } \comma
	\end{align*}
	which is precisely the statement that we wanted to prove.
\end{proof}

\begin{corollary}[\Cref{thm:main30}] \label{cor:main3}
	Assume that~$V \in W^{1,2}(\Omega)$. Let~$ \mu \in \mathcal M_2(\Omega)$. Then,
	\begin{equation} \label{eq:main31}
		\slope{\hat \E}{Wb_2}^2 ( \mu)
		=
		\begin{cases}
			\displaystyle 4 \int_0^1 \left( \partial_x \sqrt{\rho e^V} \right)^2 e^{-V} \dif x &\text{if }   \mu = \rho \dif x \\ &\text{ and } \sqrt{\rho e^V}-1 \in W^{1,2}_0(\Omega) \comma \\
			\infty &\text{otherwise,}
		\end{cases}
	\end{equation}
	where~$\hat \E$ is defined as
	\begin{equation}
		\mathcal M_2(\Omega) \ni \mu \stackrel{\hat \E}{\longmapsto} \begin{cases}
			\E(\rho) &\text{if }  \mu = \rho \dif x \comma \\
			\infty &\text{otherwise.}
		\end{cases}
	\end{equation}
	Additionally,~$\slope{\hat \E}{Wb_2}$ is lower semicontinuous w.r.t.~$Wb_2$.
\end{corollary}

\begin{proof}
	We may assume that~$ \mu = \rho \dif x$ for some~$\rho \in L^1_{+}(\Omega)$, and that~$\E(\rho) < \infty$. In particular,~$ \mu$ is finite and we can fix some~$\tilde \mu \in \cone$ such that~$\tilde \mu_\Omega = \mu$
	
	\emph{Step 1 (inequality~$\le$).} Let~$( \mu^n)_{n \in \N_0} \subseteq \mathcal M_2(\Omega)$ be such that~$Wb_2( \mu^n, \mu) \to_n 0$ (and~$\mu^n \neq \mu$). We want to prove that the limit superior
	\[ \limsup_{n \to \infty} \frac{\bigl(\hat \E ( \mu) - \hat \E ( \mu^n) \bigr)_+}{Wb_2( \mu,  \mu^n)}
	\]
	is bounded from above by the right-hand side of~\eqref{eq:main31}. To this aim, we may assume that the limit superior is actually a limit and that~$\hat \E( \mu^n) \le \hat \E( \mu) = \E(\rho)$ for every~$n \in \N_0$. In particular, each measure~$ \mu^n$ is finite and has a density~$\rho^n$. By \Cref{lemma:proj}, for every~$n \in \N_0$,
	\[
	\inf_{\tilde \nu \in \cone} \set{ \TrDis(\tilde \mu, \tilde \nu) \, : \, \tilde \nu_\Omega = \mu^n} = Wb_2( \mu,  \mu^n) \comma
	\]
	which ensures the existence of~$\tilde \mu^n \in \cone$ such that~$\tilde \mu^n_\Omega = \mu^n$ and
	\begin{equation} \label{eq:main32}
		\lim_{n \to \infty} \frac{\TrDis(\tilde \mu, \tilde \mu^n)}{Wb_2( \mu, \mu^n)} = 1 \comma \text{ as well as, consequently, } \lim_{n \to \infty} \TrDis(\tilde \mu,\tilde \mu^n) = 0 \fstop
	\end{equation}
	By~\eqref{eq:main32} and \Cref{prop:slope} (with~$\Psi \equiv 0$), we conclude that
	\[
	\lim_{n \to \infty} \frac{\bigl(\hat \E ( \mu) - \hat \E ( \mu^n) \bigr)_+}{Wb_2( \mu,  \mu^n)} \le \limsup_{n \to \infty} \frac{\bigl(\E (\rho) - \E ( \rho^n) \bigr)_+}{\TrDis(\tilde \mu, \tilde \mu^n)} \le \text{RHS of~\eqref{eq:main31}.}
	\]
	
	\emph{Step 2 (inequality~$\ge$).} By \Cref{prop:slope} (with~$\Psi \equiv 0$), we know that there exists a sequence~$(\tilde \mu^n)_{n \in \N_0} \subseteq \cone$ such that~$ \TrDis(\tilde \mu^n,\tilde \mu) \to_n 0$ (with~$\tilde \mu^n \neq \tilde \mu$) and
	\[
	\lim_{n \to \infty} \frac{\bigl(\hat \E( \mu) - \hat \E(\tilde \mu^n_\Omega)\bigr)_+}{\TrDis(\tilde \mu, \tilde \mu^n)} = \text{RHS of~\eqref{eq:main31}.}
	\]
	If this number is~$0$, then there is nothing to prove. Otherwise, we may assume that~$\mu \neq \tilde \mu_\Omega^n$ for every~$n$, and we conclude by using~\eqref{eq:distIneq}.
	
	\emph{Step 3 (semicontinuity).} The lower semicontinuity is proven as in \Cref{rkm:slopeLSC}: if~$\mu^n \stackrel{Wb_2}{\to} \mu$ and~$\sup_n \slope{\hat \E}{Wb_2}(\mu^n) < \infty$, then, up to subsequences,~$\left(\sqrt{\rho^n e^V} \right)_{n}$ converges weakly in~$W^{1,2}(\Omega)$ and (strongly) in~$C(\overline \Omega)$, the limit is~$\sqrt{\rho e^V}$ by~\cite[Proposition~2.7]{FigalliGigli10}, and~$\sqrt{\rho e^V}-1 \in W^{1,2}_0(\Omega)$. We conclude by the weak semicontinuity of the functional in~\eqref{eq:slopeLSC}.
\end{proof}

\section[Proof of Theorem 1.2]{Proof of~\Cref{Theorem_1.5}} \label{sec:main20}

As in \Cref{sec:slope}, throughout this section we restrict to the case where~$\Omega = ( 0,1 ) \subseteq \R^1$. Fix~$\mu_0 \in \cone$ such that its restriction to~$( 0,1 )$ is absolutely continuous with density equal to~$\rho_0$. Recall the scheme~\eqref{eq:jkoTrDis0}: for every~$\tau > 0$ and~$n \in \N_0$, we iteratively choose
\begin{equation} \label{eq:jko2} \mu^\tau_{(n+1)\tau} \in \argmin_{\mu \in \cone} \, \left(\Ent(\mu) + \frac{\TrDis^2(\mu,\mu_{n\tau})}{2\tau} \right) \fstop \end{equation}
These sequences of measures are extended to maps~$t \mapsto \mu^\tau_t$, constant on the intervals~$\bigl[n\tau,(n+1)\tau \bigr)$ for every~$n \in \N_0$.

The purpose of this section is to prove \Cref{Theorem_1.5}. Observe the following fact: Statement~\ref{st:main203} follows directly from Statements~\ref{st:main201}-\ref{st:main202}. Indeed, given the sequence of maps~$( t \mapsto \mu_t^\tau )_\tau$ that converges to~$t \mapsto \mu_t$ pointwise w.r.t.~$\TrDis$, we infer from~\eqref{eq:distIneq} that~$\bigl(t \mapsto (\mu_t^\tau)_\Omega \bigr)_\tau$ converges to~$t \mapsto (\mu_t)_\Omega$ pointwise w.r.t.~$Wb_2$. Since the approximating maps are precisely the same as those built with~\eqref{eq:jkoTrFun0}, we can apply~\Cref{prop:FP} to conclude Statement~\ref{st:main203}. The proof of \Cref{thm:main30} is thus split into only three parts. 

\subsection{Equivalence of the schemes}

Let us fix a measure~$\bar \mu \in \cone$ such that its restriction to~$\Omega = ( 0,1 )$ is absolutely continuous. Denote by~$\bar \rho$ the density of this restriction and assume that~$\E(\bar \rho) < \infty$.

\begin{proposition} \label{prop:equivalenceOnestep}
	If~$2\tau\abs{\Psi(1)-\Psi(0)} < 1$, then~$\mu \in \cone$ is a minimizer of
	\begin{equation} \label{eq:prop:onestepTrDis:00} \Ent(\cdot) + \frac{\TrDis^2(\cdot, \bar \mu)}{2\tau} \colon \cone \to \R \cup \set{\infty} \end{equation}
	if and only if it is a minimizer of
	\begin{equation} \label{eq:prop:onestepTrDis:01} \Ent(\cdot) + \frac{\TrFun^2(\cdot, \bar \mu)}{2\tau} \colon \cone \to \R \cup \set{\infty} \fstop \end{equation}
	In particular, there exists one single such~$\mu$; see \Cref{prop:existence} and \Cref{prop:uniq}.
\end{proposition}

\begin{proof}
	Let~$\gfun$ be the function in~\eqref{eq:prop:onestepTrDis:00} and~$\hfun$ be that in~\eqref{eq:prop:onestepTrDis:01}. Recall that~$\TrDis \le \TrFun$, which implies that~$\gfun \le \hfun$. Let~$\mu \in \cone$, let~$\gamma \in \Opt_\TrDis(\mu,\bar \mu)$ be such that the diagonal~$\Delta$ of~$\partial \Omega \times \partial \Omega$ is~$\gamma$-negligible, and define
	\[ \tilde \mu \coloneqq \mu - \pi^1_\# \gamma_{\partial \Omega}^{\partial \Omega} + \pi^2_\# \gamma_{\partial \Omega}^{\partial \Omega} \in \cone \comma \quad \tilde \gamma \coloneqq \gamma - \gamma_{\partial \Omega}^{\partial \Omega} \in \Adm_{\TrFun}(\tilde \mu,\bar \mu) \fstop \]
	We have
	\begin{multline} \label{eq:prop:onestepTrDis:1} \hfun(\tilde \mu) \le \Ent(\tilde \mu) + \frac{\cost{(\tilde \gamma)}}{2\tau} = \gfun(\mu) + \bigl( \pi^2_\# \gamma_{\partial \Omega}^{\partial \Omega} - \pi^1_\# \gamma_{\partial \Omega}^{\partial \Omega} \bigr)(\Psi) - \frac{\cost(\gamma_{\partial \Omega}^{\partial \Omega})}{2\tau} \\
		= \gfun(\mu) + \bigl(\Psi(1)-\Psi(0)\bigr)\bigl( \gamma(0,1) - \gamma(1,0) \bigr) - \frac{\gamma(0,1) + \gamma(1,0)}{2\tau} \le \gfun(\mu) \comma \end{multline}
	where, in the last inequality, we used the assumption on~$\tau$.
	
	\emph{Step 1.} It follows from~\eqref{eq:prop:onestepTrDis:1} that~$\inf \hfun \le \gfun \le \hfun$. This is enough to conclude that every minimizer of~$\hfun$ is a minimizer of~$\gfun$ too.
	
	\emph{Step 2.}  Assume now that~$\mu$ is a minimizer of~$\gfun$. Again by~\eqref{eq:prop:onestepTrDis:1},
	\[ \gfun(\mu) \le \gfun(\tilde \mu) \le \hfun(\tilde \mu) \le \gfun(\mu) \fstop \]
	Therefore, it must be true that~$\gfun(\mu) = \hfun(\tilde \mu)$ and that all inequalities in~\eqref{eq:prop:onestepTrDis:1} are equalities. This can only happen if~$\gamma_{(\partial \Omega \times \partial \Omega) \setminus \Delta} = \gamma_{\partial \Omega}^{\partial \Omega}$ has zero mass, which implies~$\mu = \tilde \mu$. It is now easy to conclude from~$ \gfun \le \hfun$ and~$\gfun(\mu)=\hfun(\mu)$ that~$\mu$ is a minimizer of~$\hfun$.
\end{proof}

\subsection{Convergence}

\begin{proposition} \label{prop:convergence1d}
	As~$\tau \to 0$, up to subsequences, the maps~$(t\mapsto \mu_t^\tau)_\tau$ converge pointwise w.r.t.~$\TrDis$ to a curve~$t \mapsto \mu_t$, continuous w.r.t~$\TrDis$. The restrictions~$(\mu_t)_\Omega$ are absolutely continuous.
\end{proposition}

\begin{lemma} \label{lemma:TVbound}
	For every~$t \ge 0$ and~$\tau > 0$ such that~$2\tau\abs{\Psi(1)-\Psi(0)} < 1$, we have the upper bound
	\begin{equation}
		\norm{\mu_t^\tau} \le \const(1+t+\tau) \fstop
	\end{equation}
\end{lemma}

\begin{proof}
	Let~$t \ge 0$ be fixed. We already know from \Cref{rmk:massbound} that~$\norm{(\mu_t^\tau)_\Omega} \le \const$. By applying \Cref{lemma:boundaryVSinterior} with~$\Phi(x) \coloneqq 1-x$, we find
	\[ \mu^\tau_{(i+1)\tau}(0)-\mu^\tau_{i\tau}(0) \le \int (1-x) \dif \, \bigl(\mu^\tau_{i\tau}-\mu^\tau_{(i+1)\tau}\bigr)_\Omega + \const \, \tau + \frac{\TrFun^2\bigl(\mu^\tau_{(i+1)\tau},\mu^\tau_{i\tau}\bigr)}{4\tau} \comma \]
	for every~$i \in \N_0$. By summing over~$i \in \set{0,1,\dotsc,\lfloor t/\tau \rfloor -1}$ and using \Cref{lemma:boundSumT}, 
	\begin{align*} \mu^\tau_t(0) - \mu_0(0) &\le \int (1-x) \dif \, (\mu_0-\mu^\tau_{t})_\Omega + \const (1+t+\tau) \le \const (1+t+\tau) \fstop \end{align*}
	Thus, the sequence~$\bigl(\mu_t^\tau(0)\bigr)_\tau$ is bounded from above as~$\tau \to 0$. By suitably choosing~$\Phi$, we can find a similar bound from below and bounds for~$\mu_t^\tau(1)$.
\end{proof}

\begin{proof}[Proof of \Cref{prop:convergence1d}]
	We can assume that~$\tau < 1$ and that~$2\tau \abs{\Psi(1)-\Psi(0)} < 1$. The proof goes as in \Cref{prop:convergence}: for a fixed~$t \ge 0$, we need to prove that
	\begin{equation} \label{eq:almostHold}
		\limsup_{\tau \to 0} \TrDis(\mu_s^\tau,\mu_t^\tau) \le \const \sqrt{\abs{r-s}(1+t)} \comma \qquad r,s \in [0,t] \comma
	\end{equation}
	and that
	\[
	\tilde K_t \coloneqq \set{\mu \in \cone \, : \, \norm{\mu} \le c_1 (2+t) \comma \text{ and } \mu_\Omega = \rho \dif x \text{ with } \int_\Omega \rho \log \rho \dif x \le c_2 (2+t)}
	\]
	is compact in~$(\cone,\TrDis)$, where the constants~$c_1$ and~$c_2$ are given by \Cref{lemma:TVbound} and \Cref{lemma:boundSumT}, respectively.
	
	The inequality~\eqref{eq:almostHold} follows from~\eqref{eq:lemmaBoundSumT2}.
	If~$( \mu^n)_{n \in \N_0}$ is a sequence in~$\tilde K_t$, thanks to the bound on the total mass, we can extract a (not relabeled) subsequence that converges weakly to some~$ \mu \in \cone$. Let~$ \rho^n$ be the density of~$\mu_\Omega^n$ for every~$n \in \N_0$. We exploit the bound on the integral~$\int_\Omega  \rho^n \log  \rho^n$ to extract a further subsequence such that~$( \rho^n)_{n \in \N_0}$ converges weakly in~$L^1(\Omega)$ to some~$ \rho$. We have~$ \mu_\Omega =  \rho \dif x$, as well as~$\norm{ \mu} \le c_1(2+t)$ and~$\int_\Omega  \rho \log  \rho \dif x \le c_2 (2+t)$; hence~$ \mu \in \tilde K_t$. The convergence~$ \mu^n \to_n  \mu$ holds also w.r.t.~$\TrDis$ thanks to~\Cref{lemma:convergCrit}.
\end{proof}

\subsection{Curve of maximal slope}
\begin{proposition} \label{prop:edi}
	Assume that~$V \in W^{1,2}(\Omega)$. If the sequence~$(t \mapsto \mu^\tau_t)_\tau$ converges pointwise w.r.t.~$\TrDis$ to a curve~$t \mapsto \mu_t$, then the latter is a curve of maximal slope for the functional~$\Ent$ in the metric space~$(\cone, \TrDis)$.
\end{proposition}

To prove this proposition, we employ the classical~\cite[Theorem~2.3.1]{AmbrosioGigliSavare08}, but we also crucially need the results of~\Cref{sec:slope}. In particular, we rely on the explicit formula for the slope of~\Cref{prop:slope} and on the consequent semicontinuity observed in~\Cref{rkm:slopeLSC}.

\begin{proof}
	Consider the subspace~$\widetilde \cone \coloneqq \set{\mu \in \cone \, : \, \Ent(\mu) \le \Ent(\mu_0)}$. Note that, since~$\Ent$ is~$\TrDis$-lower semicontinuous (\Cref{prop:lscIbis}),~$t \mapsto \mu_t$ entirely lies in~$\widetilde{\cone}$. Moreover,~$\slope{\Ent}{\TrDis}$ coincides with~$\slope{(\Ent|_{\widetilde{\cone}})}{\TrDis}$ on~$\widetilde{\cone}$. Therefore, it suffices to prove that~$t \mapsto \mu_t$ is a curve of maximal slope in~$\widetilde{\cone}$.
	
	We invoke \cite[Theorem~2.3.1]{AmbrosioGigliSavare08}. Let us check the assumptions. Firstly, the space~$(\widetilde \cone, \TrDis)$ is complete by \Cref{lemma:completeSublev}. %
	Secondly,~\cite[(2.3.2)]{AmbrosioGigliSavare08} is satisfied because the slope~$\slope{\Ent}{\TrDis}$ is~$\TrDis$-lower semicontinuous; see \Cref{rkm:slopeLSC} and \cite[Remark~2.3.2]{AmbrosioGigliSavare08}. Thirdly, \cite[Assumptions 2.1a,b]{AmbrosioGigliSavare08} follow from \Cref{prop:lscIbis} and \Cref{prop:equivalenceOnestep}. Finally, to prove \cite[(2.3.3)]{AmbrosioGigliSavare08}, let us pick a sequence~$(\mu^n)_{n \in \N_0} \subseteq \widetilde{\cone}$ that converges to some~$\mu$ w.r.t.~$\TrDis$ and such that~$\sup_n \slope{\Ent}{\TrDis}(\mu^n) < \infty$. We will show that~$\Ent(\mu^n) \to \Ent(\mu)$. Note that it is enough to prove this convergence \emph{up to subsequences}. Let~$\rho^n,\rho$ be the densities of~$\mu^n_\Omega,\mu_\Omega$, respectively.
	Since~$\sup_n \slope{\Ent}{\TrDis}(\mu^n) < \infty$, up to subsequences, the functions~$\left(\sqrt{\rho^n e^V} \right)_n$ converge in~$C(\overline \Omega)$ to~$\sqrt{\rho e^V}$. Since~$V$ is bounded, we also have the convergence~$\rho^n \to \rho$ in~$C(\overline \Omega)$. We write
	\begin{align*}
		\abs{\Ent(\mu^n) - \Ent(\mu)} &= \abs{\E(\mu^n) - \E(\mu) + (\mu^n-\mu)_{\partial \Omega}(\Psi)} \\
		&\le \abs{\E(\mu^n) - \E(\mu) - (\mu^n-\mu)_{\Omega}(\Psi)} + \abs{\mu^n(\Psi) - \mu(\Psi)}
	\end{align*}
	Thanks to the uniform convergence~$\rho^n \to \rho$, we have
	\[
	\abs{\E(\mu^n) - \E(\mu) - (\mu_n-\mu)_{\Omega}(\Psi)} \to 0 \fstop
	\] Additionally, by \Cref{lemma:boundaryVSinteriorTrDis},
	\[
	\abs{\mu^{n}(\Psi) - \mu(\Psi)} \le  \const \TrDis(\mu^{n}, \mu) \sqrt{ \norm{\mu^{n}_\Omega} + \norm{\mu_\Omega} + \TrDis^2(\mu^{n}, \mu) } \comma
	\]
	from which we conclude, because~$\sup_{n} \norm{\mu^{n}_\Omega} \le \sup_n \norm{\rho^n}_{L^\infty} < \infty$.
\end{proof}

\begin{remark}
	To be precise, \cite[Theorem~2.3.1]{AmbrosioGigliSavare08} applies to the limit of the maps~$t \mapsto \tilde \mu_t^\tau \coloneqq \mu_{\lceil t/\tau \rceil \tau}$
	(as opposed to~$\mu_t^\tau = \mu_{\lfloor t/\tau \rfloor \tau}$). It can be easily checked that the distance~$\TrDis(\mu_t^\tau,\tilde \mu_t^\tau)$ converges to~$0$ locally uniformly in time; see~\eqref{eq:lemmaBoundSumT2}.
\end{remark}

\appendix

\section{Additional properties of~$\TrDis$} \label{sec:appendix}

\subsection{$\TrDis$ is not a distance when~$d \ge 2$}

We are going to prove that, when~$d \ge 2$, the property
\[
\TrDis(\mu,\nu) = 0 \quad \Longrightarrow \quad \mu = \nu
\]
in general breaks down. In fact, when applying~$\TrDis$ to two measures~$\mu,\nu \in \cone$ the information about~$\mu_{\partial \Omega}$ and~$\nu_{\partial \Omega}$ is completely lost, as soon as~$\partial \Omega$ is connected and ``not too irregular''. A similar result is~\cite[Theorem 2.2]{Mainini12} by~E.~Mainini.

\begin{proposition} \label{prop:informloss}
	If~$\alpha \colon [0,1] \to \partial \Omega$ is~$\left(\frac{1}{2}+\epsilon\right)$-H\"older continuous for some~$\epsilon > 0$, then
	\begin{equation} \label{eq:informloss0}
		\TrDis\bigl(\delta_{\alpha(0)}-\delta_{\alpha(1)},0\bigr) = 0 \fstop
	\end{equation}
	Consequently: Assume that~$\partial \Omega$ is~$C^{0,\frac{1}{2}+}$-path-connected, meaning that for every pair of points~$x,y \in \partial \Omega$ there exist~$\epsilon > 0$ and a $\left(\frac{1}{2}+\epsilon\right)$-H\"older curve~$\alpha \colon [0,1] \to \partial \Omega$ with~$\alpha(0) = x$ and~$\alpha(1) = y$; then, 	for every~$\mu,\nu \in \cone$, we have
	\begin{equation} \label{eq:informloss} \TrDis(\mu,\nu) = Wb_2(\mu_\Omega,\nu_\Omega) \fstop \end{equation}
\end{proposition}

\begin{proof}
	\emph{Step 1.} Let~$\alpha \colon [0,1] \to \partial \Omega$ be~$\left(\frac{1}{2}+\epsilon\right)$-H\"older continuous for some~$\epsilon > 0$. For~$n \in \N_1$, consider the points
	\[ x_i \coloneqq \alpha(i/n), \qquad i \in \set{ 0,1,\dotsc,n} \comma \]
	and the measure
	\[ \gamma^n \coloneqq \sum_{i=0}^{n-1} \delta_{(x_{i},x_{i+1})} \fstop \]
	It is easy to check that~$\gamma^n \in \Adm_{\TrDis}\bigl(\delta_{\alpha(0)} - \delta_{\alpha(1)},0\bigr)$; moreover,
	\begin{equation*}  \cost(\gamma^n) =  \sum_{i=0}^{n-1} \abs{x_{i}-x_{i+1}}^2 \le \const_\alpha \sum_{i=0}^{n-1} n^{-1-2\epsilon} = \const_\alpha n^{-2\epsilon} \comma
	\end{equation*}
	where the inequality follows from the H\"older continuity of~$\alpha$. We conclude~\eqref{eq:informloss0} by letting~$n \to \infty$.
	
	\emph{Step 2.} Assume now that~$\partial \Omega$ is~$C^{0,\frac{1}{2}+}$-path-connected. Fix a finite signed Borel measure~$\eta$ on~$\partial \Omega$ with~$\eta(\partial \Omega) = 0$, that is,~$\norm{\eta_+} = \norm{\eta_-} \eqqcolon \lambda$. We shall prove that~$\TrDis(\eta,0) = 0$.
	Fix~$\epsilon_1,\epsilon_2 > 0$ and let~$X = \set{x_1,x_2,\dotsc,x_N} \subseteq \partial \Omega$ be a~$\epsilon_1$-covering for~$\partial \Omega$, meaning that there exists a function~$P \colon \partial \Omega \to X$ such that~$\abs{x-P(x)} \le \epsilon_1$ for every~$x \in \partial \Omega$. We pick one such~$P$ that is also Borel measurable (we can by~\cite[Theorem~18.19]{AliprantisBorder06}). From the previous Step, for every~$i,j \in \set{1,2,\dotsc,N}$, we get~$\gamma_{i,j}$ (nonnegative and concentrated on~$\partial \Omega \times \partial \Omega$) such that
	\[ \pi^1_\# \gamma_{i,j} - \pi^2_\# \gamma_{i,j} = \delta_{x_i} - \delta_{x_j} \quad \text{and} \quad \cost(\gamma_{i,j}) \le \epsilon_2 \fstop \]
	We define
	\[ \gamma \coloneqq (\Id,P)_\# \eta_+ + (P,\Id)_\# \eta_- + \frac{1}{\lambda}\sum_{i,j = 1}^N \eta_+\bigl(P^{-1}(x_i)\bigr) \eta_-\bigl(P^{-1}(x_j)\bigr) \gamma_{i,j} \fstop \]
	The~$\TrDis$-admissibility of~$\gamma$, i.e.,~$\gamma \in \Adm_{\TrDis}(\eta,0)$, is straightforward. Furthermore,
	\begin{align*}
		\cost(\gamma) &= \int \abs{\Id-P}^2 \dif \, (\eta_+ + \eta_-)
		+ \frac{1}{\lambda} \sum_{i,j = 1}^N \eta_+(P^{-1}(x_i)) \eta_-(P^{-1}(x_j)) \cost(\gamma_{i,j}) \\
		&\le 2\lambda\epsilon_1^2 + \lambda \epsilon_2 \comma
	\end{align*}
	which brings us to the conclusion that~$\TrDis(\eta,0) = 0$ by arbitrariness of~$\epsilon_1,\epsilon_2$.
	
	\emph{Step 3.} Let us assume again that~$\partial \Omega$ is~$C^{0,\frac{1}{2}+}$-path-connected, and fix~$\mu,\nu \in \cone$ and~$\epsilon_3 > 0$. Let~$\gamma$ be a $Wb_2$-optimal transport plan between~$\mu_\Omega$ and~$\nu_\Omega$, and set~$\tilde \mu \coloneqq \pi^1_\# \gamma + (\nu - \pi^2_\# \gamma)_{\partial \Omega}$. It is easy to check that~$\tilde \mu \in \cone$ and that~$\mu_\Omega = \tilde \mu_\Omega$. Therefore, the previous Step is applicable to~$\eta \coloneqq \mu_{\partial \Omega} - \tilde \mu_{\partial \Omega}$, and produces~$\gamma_\eta$ on~$\partial \Omega \times \partial \Omega$ such that
	\[ \pi^1_\# \gamma_\eta - \pi^2_\# \gamma_\eta = \eta \quad \text{and} \quad \cost(\gamma_\eta) \le \epsilon_3 \fstop \]
	The measure~$\gamma' \coloneqq \gamma + \gamma_\eta$ is $\TrDis$-admissible between~$\mu$ and~$\nu$. Therefore,
	\begin{equation*}
		\TrDis^2(\mu,\nu) \le \cost(\gamma') \le \cost(\gamma) + \epsilon_3 = Wb_2^2(\mu_\Omega, \nu_\Omega)+\epsilon_3 \comma
	\end{equation*}
	which yields one of the two inequalities in~\eqref{eq:informloss} by arbitrariness of~$\epsilon_3$. The other inequality is~\eqref{eq:distIneq}.
\end{proof}

\subsection{(Lack of) completeness}
We prove here two claims from \Cref{sec:distance}: in the setting where~$\Omega$ is a finite union of intervals, the metric space~$(\cone,\TrDis)$ is \emph{not} complete, but the sublevels of~$\Ent$ are.

\begin{proposition} \label{prop:notComplete}
	Assume that~$d=1$ and that $\Omega$ is a finite union of intervals. Then the metric space~$(\cone,\TrDis)$ is not complete.
\end{proposition}

\begin{proof}
	Without loss of generality, we may assume that~$(0,1)$ is a connected component of~$\Omega$,~i.e.,~$(0,1) \subseteq \Omega$ and~$\set{0,1} \subseteq \partial \Omega$.
	
	Consider the sequence
	\[
	\mu^n \coloneqq \frac{1}{x} \Leb^1_{(2^{-n},1)} - \delta_0 \int_{2^{-n}}^1 \frac{1}{x} \dif x \in \cone \comma \qquad n \in \N_1 \fstop
	\]
	For every~$n$, there exists the admissible transport plan
	\[
	\gamma^n \coloneqq \delta_0 \otimes \left( \frac{1}{x} \Leb^1_{(2^{-n-1},2^{-n})} \right) + (\Id, \Id)_\# \left(\frac{1}{x} \Leb^1_{(2^{-n},1)} \right) \in \Adm_\TrDis(\mu^n,\mu^{n+1}) \comma
	\]
	which yields
	\[
	\sum_{n=1}^\infty \TrDis(\mu^n,\mu^{n+1}) \le \sum_{n=1}^\infty \sqrt{\int_{2^{-n-1}}^{2^{-n}} \frac{x^2}{x} \dif x} = \sum_{n=1}^\infty \sqrt{\frac{3}{8}} \, 2^{-n} = \sqrt{\frac{3}{8}} \semicolon
	\]
	hence~$(\mu^n)_n$ is Cauchy.
	
	Assume now that~$\mu^n \stackrel{\TrDis}{\to}_n \mu$ for some~$\mu \in \cone$ and, for every~$n \in \N_1$, fix~$\tilde \gamma^n \in \Opt_\TrDis(\mu^n,\mu)$. Also fix~$\epsilon>0$. We have
	\[
	\TrDis^2(\mu^n,\mu) = \int \abs{x-y}^2 \dif \tilde \gamma^n(x,y) \ge \epsilon^2 \tilde \gamma^n\bigl([\epsilon,1-\epsilon] \times \partial \Omega \bigr) \comma
	\]
	and, using the conditions in \Cref{def:TrDis},
	\begin{align*}
		\norm{\mu_\Omega} &\ge \tilde \gamma^n\bigl([\epsilon,1-\epsilon] \times \Omega \bigr) = \mu^n\bigl([\epsilon,1-\epsilon]\bigr)-\tilde \gamma^n\bigl([\epsilon,1-\epsilon] \times \partial \Omega \bigr) \\
		&\ge \mu^n\bigl([\epsilon,1-\epsilon] \bigr) - \frac{\TrDis^2(\mu^n,\mu)}{\epsilon^2} \fstop
	\end{align*}
	Passing to the limit~$n \to \infty$, we find
	\[ \norm{\mu_\Omega} \ge \int_\epsilon^{1-\epsilon} \frac{1}{x} \dif x \]
	from which, by arbitrariness of~$\epsilon$, it follows that the total mass of~$\mu_\Omega$ is infinite, contradicting the finiteness required in \Cref{def:TrDis}.
\end{proof}

\begin{proposition} \label{lemma:completeSublev}
	Assume that~$d=1$ and that $\Omega$ is a finite union of intervals. Then the sublevels of~$\Ent$ in~$\cone$ are complete w.r.t.~$\TrDis$.
\end{proposition}

\begin{proof}
	Take a Cauchy sequence~$(\mu^n)_{n \in \N_0} \subseteq \cone$ for~$\TrDis$ in a sublevel of~$\Ent$, that is,~$\Ent(\mu^n) \le M$ for some~$M \in \R$, for every~$n \in \N_0$. Thanks to \Cref{lemma:boundaryVSinteriorTrDis}, for every~$n \in \N_0$ we have
	\begin{align*}
		M &\ge \Ent(\mu^n) \ge \int_{\Omega} \rho^n \log \rho^n \dif x - \bigl(\norm{V}_{L^\infty}+1\bigr) \norm{\mu^n_\Omega} + \mu^n_{\partial \Omega}(\Psi) \\
		&\ge \int_{\Omega} \rho^n \log \rho^n \dif x - \bigl(\norm{V}_{L^\infty}+1\bigr) \norm{\mu^n_\Omega} + \mu^0(\Psi)-\mu^n_\Omega(\Psi) \\
		&\quad - \const \TrDis(\mu^n,\mu^0) \sqrt{\norm{\mu^n_\Omega} + \norm{\mu^0_\Omega}+\TrDis^2(\mu^n,\mu^0)} \comma
	\end{align*}
	and, since~$\TrDis(\mu^n,\mu^0)$ is bounded, the family~$(\rho^n)_{n \in \N_0}$ is uniformly integrable. Let~$(\rho^{n_k})_{k \in \N_0}$ be a subsequence that converges to some~$\rho$ weakly in~$L^1(\Omega)$. For each of the finitely many~$\bar x \in \partial \Omega$, let~$\Phi_{\bar x}$ be a Lipschitz continuous function such that
	\[
	\Phi_{\bar x}(\bar x) = 1 \quad \text{and} \quad \Phi_{\bar x}(x) = 0 \text{ if } x \in \partial \Omega \setminus \set{\bar x} \fstop
	\]
	Again by \Cref{lemma:boundaryVSinteriorTrDis}, for every~$\bar x \in \partial \Omega$ and~$n,m \in \N_0$, we have
	\begin{align*}
		\abs{\mu^n(\bar x) - \mu^m(\bar x)} &\le \abs{\mu^n_\Omega(\Phi_{\bar x}) - \mu^m_\Omega(\Phi_{\bar x})} \\ 
		&\quad + \const_{\Phi_{\bar x}} \TrDis(\mu^n,\mu^m) \sqrt{\norm{\mu^n_\Omega} + \norm{\mu^m_\Omega}+\TrDis^2(\mu^n,\mu^m)} \\
		&= \abs{\int_\Omega \Phi_{\bar x} \cdot (\rho^n - \rho^m) \dif x} \\
		&\quad+ \const_{\Phi_{\bar x}} \TrDis(\mu^n,\mu^m) \sqrt{\norm{\rho^n}_{L^1} + \norm{\rho^m}_{L^1}+\TrDis^2(\mu^n,\mu^m)} \comma
	\end{align*}
	which implies that~$(\mu^{n_k}(\bar x))_{k \in \N_0}$ is a Cauchy sequence in~$\R$, thus convergent to some number~$l_{\bar x}$. Define
	\[
	\mu \coloneqq \rho \dif x + \sum_{\bar x \in \partial \Omega} l_{\bar x} \delta_{\bar x}  \fstop
	\]
	It is easy to check that~$\mu^{n_k} \to_k \mu$ weakly; therefore, by \Cref{lemma:convergCrit}, also w.r.t.~$\TrDis$. The limit~$\mu$ also lies in the sublevel, i.e.,~$\Ent(\mu) \le M$, by \Cref{prop:lscIbis}.
\end{proof}

\subsection{If~$\Omega$ is an interval,~$\TrDis$ is geodesic, but~$\Ent$ is not geodesically convex} \label{sec:geod}
We prove that~$(\cone,\TrDis)$ is geodesic when~$\Omega = (0,1)$, by using the analogous well-known property of the classical $2$-Wasserstein distance. However, as we expect in light of~\cite[Remark~3.4]{FigalliGigli10},~$\Ent$ is \emph{not} geodesically~$\lambda$-convex for any~$\lambda$. We provide a short proof by adapting the aforementioned remark.

\begin{proposition} \label{prop:geodesic}
	If~$\Omega = (0,1)$, then~$(\cone,\TrDis)$ it is a geodesic metric space.
\end{proposition}

\begin{proof}
	We already know from \Cref{prop:distance} that~$(\cone,\TrDis)$ is a metric space.
	
	For any two measures~$\mu_0,\mu_1 \in \cone$, we need to find a curve~$t \mapsto \mu_t$ such that
	\begin{equation} \label{eq:geodesicIneq}
		\TrDis(\mu_s,\mu_t) \le (t-s) \TrDis(\mu_0,\mu_1) \comma \qquad 0 \le s \le t \le 1 \fstop
	\end{equation}
	The opposite inequality follows from the triangle inequality and~\eqref{eq:geodesicIneq} itself. Indeed,
	\begin{align*}
		\TrDis(\mu_0,\mu_1) &\le \TrDis(\mu_0,\mu_s) + \TrDis(\mu_s,\mu_t) + \TrDis(\mu_t,\mu_1) \\
		&\stackrel{\eqref{eq:geodesicIneq}}{\le} (s+t-s+1-t)\TrDis(\mu_0,\mu_1) = \TrDis(\mu_0,\mu_1) \comma
	\end{align*}
	and, in order for the inequalities to be equalities, the identity~$\TrDis(\mu_s,\mu_t) = (t-s)\TrDis(\mu_0,\mu_1)$ must be true.
	
	Take~$\gamma \in \Opt_\TrDis(\mu_0,\mu_1)$. By~\Cref{prop:trPlans},~$\gamma$ is optimal, between its marginals, for the classical~$2$-Wasserstein distance. Since the set~$\ovom = [0,1]$, endowed with the Euclidean metric, is geodesic, the classical theory of optimal transport (see, e.g.,~\cite[Theorem 10.6]{AmbrosioBrueSemola21}) ensures the existence of a curve (geodesic)~$t \mapsto \nu_t$ of nonnegative measures on~$\ovom$ with constant total mass, such that
	\begin{equation} \label{eq:geodesicW2}
		W_2(\nu_s,\nu_t) \le (t-s)W_2(\pi^1_\# \gamma, \pi^2_\# \gamma) = (t-s)\sqrt{\cost( \gamma)} = (t-s) \TrDis(\mu_0,\mu_1)
	\end{equation}
	for~$0 \le s \le t \le 1$. After noticing that~$\nu_1-\nu_0 = \mu_1 - \mu_0$ by Condition~\ref{(3)} in \Cref{def:TrDis}, we define
	\[
	\mu_t \coloneqq \mu_0+\nu_t-\nu_0 \comma \qquad t \in (0,1) \fstop
	\]
	We claim that this is the sought curve. Firstly, since
	\[
	(\mu_t)_\Omega = (\mu_0)_{\Omega} + (\nu_t)_\Omega - (\nu_0)_{\Omega} = (\nu_t)_\Omega \ge 0
	\]
	by Condition~\ref{(1)} in \Cref{def:TrDis}, and since~$\nu_0(\ovom) = \nu_t(\ovom)$, we can be sure that~$\mu_t \in \cone$ for every~$t$. Secondly, every $W_2$-optimal transport plan~$ \gamma_{st}$ between~$\nu_s$ and~$\nu_t$ is $\TrDis$-admissible between~$\mu_s$ and~$\mu_t$. Hence,
	\[
	\TrDis(\mu_s,\mu_t) \le \sqrt{\cost( \gamma_{st})}=W_2(\nu_s,\nu_t) \stackrel{\eqref{eq:geodesicW2}}{\le} (t-s) \TrDis(\mu_0,\mu_1) \fstop \qedhere
	\]
\end{proof}

\begin{proposition} \label{prop:notconv}
	Let~$\Omega = (0,1)$. The functional~$\Ent$ is \emph{not} geodesically $\lambda$-convex on the metric space~$(\cone,\TrDis)$ for any~$\lambda \in \R$.
\end{proposition}

\begin{proof}
	Consider the curve
	\[
	t \longmapsto \mu_t \coloneqq \begin{cases}
		\frac{1}{t} \Leb^1_{(0,t)} - \delta_0 &\text{if } t \in (0,1] \comma \\
		0 &\text{if } t = 0 \fstop
	\end{cases}
	\]
	Clearly,~$\mu_t \in \cone$ for every~$t \in [0,1]$. We claim that this curve is a geodesic, that~$\Ent(\mu_0) < \infty$, and that~$\lim_{t \to 0} \Ent(\mu_t) = \infty$,
	which would conclude the proof. The second claim, namely~$\Ent(\mu_0) < \infty$, is obvious. The third claim is true because
	\[
	\Ent(\mu_t) = - \log t + \fint_0^t V \dif x - \Psi(0) \comma \qquad t \in (0,1] \comma
	\]
	and, since~$V \in L^\infty(0,1)$, the right-hand side tends to~$\infty$ as~$t \to 0$. To prove the first claim, fix~$0 \le s < t \le 1$ and define
	\[
	\gamma_{st} \coloneqq \left( \Id, \frac{s}{t} \Id \right)_\# \mu_t \in \Adm_\TrDis(\mu_t,\mu_s) \comma
	\]
	which gives
	\begin{equation} \label{eq:notGeoC1}
		\TrDis^2(\mu_s,\mu_t) \le \cost(\gamma_{st}) = \int \abs{x-\frac{s}{t}x}^2 \dif \mu_t = \frac{(t-s)^2}{3} \fstop
	\end{equation}
	Conversely, for every~$\gamma \in \Opt_{\TrDis}(\mu_1,\mu_0)$, Condition~\ref{(3)} in \Cref{def:TrDis} implies
	\[
	\gamma(1,1) + \gamma(1,0) + \gamma(\set{1}\times \Omega) = \gamma(1,1) + \gamma(0,1) + \gamma(\Omega \times \set{1}) \comma
	\]
	and, since~$\gamma(\set{1} \times \Omega) = 0$ by Condition~\ref{(2)} in \Cref{def:TrDis}, we have~$\gamma(1,0) \ge \gamma(\Omega \times \set{1})$. Therefore,
	\begin{align*}
		\TrDis^2(\mu_1,\mu_0) = \cost(\gamma) &\ge \cost\bigl(\gamma_\Omega^{\set{0}}\bigr) + \int \abs{x-1}^2 \dif \pi^1_\# \gamma_\Omega^{\set{1}} + \gamma(1,0) \\
		&\ge \cost\bigl(\gamma_\Omega^{\set{0}}\bigr) + \int \bigl(\abs{x-1}^2 + 1\bigr) \dif \pi^1_\# \gamma_\Omega^{\set{1}}  \ge \int x^2 \dif \pi^1_\# \gamma_\Omega^{\partial \Omega} \fstop
	\end{align*}
	By Conditions~\ref{(1)} and~\ref{(2)} in \Cref{def:TrDis},
	\[
	\int x^2 \dif \pi^1_\# \gamma_\Omega^{\partial \Omega} = \int x^2 \dif \pi^1_\# \gamma_\Omega^\ovom = \int_0^1 x^2 \dif x = \frac{1}{3} \semicolon
	\]
	hence
	\[
	\TrDis^2(\mu_s,\mu_t) \stackrel{\eqref{eq:notGeoC1}}{\le} \frac{(t-s)^2}{3} \le (t-s)^2 \TrDis^2(\mu_1,\mu_0) \comma
	\]
	and this concludes the proof.
\end{proof}


\begin{thebibliography}{25}
	\providecommand{\natexlab}[1]{#1}
	\providecommand{\url}[1]{\texttt{#1}}
	\expandafter\ifx\csname urlstyle\endcsname\relax
	\providecommand{\doi}[1]{doi: #1}\else
	\providecommand{\doi}{doi: \begingroup \urlstyle{rm}\Url}\fi
	
	\bibitem[Aliprantis and Border(2006)]{AliprantisBorder06}
	C.~D. Aliprantis and K.~C. Border.
	\newblock \emph{Infinite dimensional analysis}.
	\newblock Springer, Berlin, third edition, 2006.
	\newblock A hitchhiker's guide.
	
	\bibitem[Ambrosio and Gigli(2013)]{AmbrosioGigli13}
	L.~Ambrosio and N.~Gigli.
	\newblock A user's guide to optimal transport.
	\newblock In \emph{Modelling and optimisation of flows on networks}, volume
	2062 of \emph{Lecture Notes in Math.}, pages 1--155. Springer, Heidelberg,
	2013.
	\newblock \doi{10.1007/978-3-642-32160-3\_1}.
	
	\bibitem[Ambrosio et~al.(2008)Ambrosio, Gigli, and
	Savar\'{e}]{AmbrosioGigliSavare08}
	L.~Ambrosio, N.~Gigli, and G.~Savar\'{e}.
	\newblock \emph{Gradient flows in metric spaces and in the space of probability
		measures}.
	\newblock Lectures in Mathematics ETH Z\"{u}rich. Birkh\"{a}user Verlag, Basel,
	second edition, 2008.
	
	\bibitem[Ambrosio et~al.(2021)Ambrosio, Bru\'{e}, and
	Semola]{AmbrosioBrueSemola21}
	L.~Ambrosio, E.~Bru\'{e}, and D.~Semola.
	\newblock \emph{Lectures on optimal transport}, volume 130 of \emph{Unitext}.
	\newblock Springer, Cham, 2021.
	\newblock \doi{10.1007/978-3-030-72162-6}.
	\newblock La Matematica per il 3+2.
	
	\bibitem[Benamou and Brenier(2000)]{BenamouBrenier00}
	J.-D. Benamou and Y.~Brenier.
	\newblock A computational fluid mechanics solution to the {M}onge-{K}antorovich
	mass transfer problem.
	\newblock \emph{Numer. Math.}, 84\penalty0 (3):\penalty0 375--393, 2000.
	\newblock \doi{10.1007/s002110050002}.
	
	\bibitem[Bogachev(2007)]{Bogachev07}
	V.~I. Bogachev.
	\newblock \emph{Measure theory. {V}ol. {I}, {II}}.
	\newblock Springer-Verlag, Berlin, 2007.
	\newblock \doi{10.1007/978-3-540-34514-5}.
	
	\bibitem[Bormann et~al.(2025)Bormann, Monsaingeon, Renger, and {von
		Renesse}]{BormannMonsaingeonRengerVonRenesse25}
	M.~Bormann, L.~Monsaingeon, D.~R.~M. Renger, and M.~{von Renesse}.
	\newblock A gradient flow that is none: Heat flow with {W}entzell boundary
	condition.
	\newblock \emph{arXiv preprint}, 2025.
	\newblock \doi{10.48550/arXiv.2506.22093}.
	
	\bibitem[Brenier(1987)]{Brenier87}
	Y.~Brenier.
	\newblock D\'{e}composition polaire et r\'{e}arrangement monotone des champs de
	vecteurs.
	\newblock \emph{C. R. Acad. Sci. Paris S\'{e}r. I Math.}, 305\penalty0
	(19):\penalty0 805--808, 1987.
	
	\bibitem[Brezis(2011)]{Brezis11}
	H.~Brezis.
	\newblock \emph{Functional analysis, {S}obolev spaces and partial differential
		equations}.
	\newblock Universitext. Springer, New York, 2011.
	
	\bibitem[Casteras et~al.(2025)Casteras, Monsaingeon, and
	Santambrogio]{CasterasMonsaingeonSantambrogio24}
	J.-B. Casteras, L.~Monsaingeon, and F.~Santambrogio.
	\newblock Sticky-reflecting diffusion as a {W}asserstein gradient flow.
	\newblock \emph{J. Math. Pures Appl. (9)}, 199:\penalty0 Paper No. 103721, 32,
	2025.
	\newblock \doi{10.1016/j.matpur.2025.103721}.
	
	\bibitem[Cox(2020)]{Cox20}
	G.~Cox.
	\newblock Almost sure uniqueness of a global minimum without convexity.
	\newblock \emph{Ann. Statist.}, 48\penalty0 (1):\penalty0 584--606, 2020.
	\newblock \doi{10.1214/19-AOS1829}.
	
	\bibitem[Daneri and Savar\'e(2008)]{DaneriSavare08}
	S.~Daneri and G.~Savar\'e.
	\newblock Eulerian calculus for the displacement convexity in the {W}asserstein
	distance.
	\newblock \emph{SIAM J. Math. Anal.}, 40\penalty0 (3):\penalty0 1104--1122,
	2008.
	\newblock \doi{10.1137/08071346X}.
	
	\bibitem[De~Giorgi(1993)]{DeGiorgi93}
	E.~De~Giorgi.
	\newblock New problems on minimizing movements.
	\newblock In \emph{Boundary value problems for partial differential equations
		and applications}, volume~29 of \emph{RMA Res. Notes Appl. Math.}, pages
	81--98. Masson, Paris, 1993.
	
	\bibitem[Dreher and J\"{u}ngel(2012)]{DreherJungel12}
	M.~Dreher and A.~J\"{u}ngel.
	\newblock Compact families of piecewise constant functions in {$L^p(0,T;B)$}.
	\newblock \emph{Nonlinear Anal.}, 75\penalty0 (6):\penalty0 3072--3077, 2012.
	\newblock \doi{10.1016/j.na.2011.12.004}.
	
	\bibitem[Erbar and Meglioli(2025)]{ErbarMeglioli24}
	M.~Erbar and G.~Meglioli.
	\newblock Gradient flow for a class of diffusion equations with dirichlet
	boundary data.
	\newblock \emph{arXiv preprint}, 2025.
	\newblock \doi{10.48550/arXiv.2408.05987}.
	
	\bibitem[Figalli and Gigli(2010)]{FigalliGigli10}
	A.~Figalli and N.~Gigli.
	\newblock A new transportation distance between non-negative measures, with
	applications to gradients flows with {D}irichlet boundary conditions.
	\newblock \emph{J. Math. Pures Appl. (9)}, 94\penalty0 (2):\penalty0 107--130,
	2010.
	\newblock \doi{10.1016/j.matpur.2009.11.005}.
	
	\bibitem[Jordan et~al.(1998)Jordan, Kinderlehrer, and
	Otto]{JordanKinderlehrerOtto98}
	R.~Jordan, D.~Kinderlehrer, and F.~Otto.
	\newblock The variational formulation of the {F}okker-{P}lanck equation.
	\newblock \emph{SIAM J. Math. Anal.}, 29\penalty0 (1):\penalty0 1--17, 1998.
	\newblock \doi{10.1137/S0036141096303359}.
	
	\bibitem[Kim et~al.(2025)Kim, Koo, and Seo]{KimKooSeo23}
	D.~Kim, D.~Koo, and G.~Seo.
	\newblock A gradient flow for the porous medium equations with dirichlet
	boundary conditions.
	\newblock \emph{arXiv preprint}, 2025.
	\newblock \doi{10.48550/arXiv.2212.06092}.
	
	\bibitem[Mainini(2011)]{Mainini12}
	E.~Mainini.
	\newblock A description of transport cost for signed measures.
	\newblock \emph{Zap. Nauchn. Sem. S.-Peterburg. Otdel. Mat. Inst. Steklov.
		(POMI)}, 390:\penalty0 147--181, 308--309, 2011.
	\newblock \doi{10.1007/s10958-012-0718-2}.
	
	\bibitem[Monsaingeon(2021)]{Monsaingeon21}
	L.~Monsaingeon.
	\newblock A new transportation distance with bulk/interface interactions and
	flux penalization.
	\newblock \emph{Calc. Var. Partial Differential Equations}, 60\penalty0
	(3):\penalty0 Paper No. 101, 49, 2021.
	\newblock \doi{10.1007/s00526-021-01946-2}.
	
	\bibitem[Morales(2018)]{Morales18}
	J.~Morales.
	\newblock A new family of transportation costs with applications to
	reaction-diffusion and parabolic equations with boundary conditions.
	\newblock \emph{J. Math. Pures Appl. (9)}, 112:\penalty0 41--88, 2018.
	\newblock \doi{10.1016/j.matpur.2017.12.001}.
	
	\bibitem[Otto(1999)]{Otto99}
	F.~Otto.
	\newblock Evolution of microstructure in unstable porous media flow: a
	relaxational approach.
	\newblock \emph{Comm. Pure Appl. Math.}, 52\penalty0 (7):\penalty0 873--915,
	1999.
	\newblock \doi{10.1002/(SICI)1097-0312(199907)52:7<873::AID-CPA5>3.3.CO;2-K}.
	
	\bibitem[Otto(2001)]{Otto01}
	F.~Otto.
	\newblock The geometry of dissipative evolution equations: the porous medium
	equation.
	\newblock \emph{Comm. Partial Differential Equations}, 26\penalty0
	(1-2):\penalty0 101--174, 2001.
	\newblock \doi{10.1081/PDE-100002243}.
	
	\bibitem[Profeta and Sturm(2020)]{ProfetaSturm20}
	A.~Profeta and K.-T. Sturm.
	\newblock Heat flow with {D}irichlet boundary conditions via optimal transport
	and gluing of metric measure spaces.
	\newblock \emph{Calc. Var. Partial Differential Equations}, 59\penalty0
	(4):\penalty0 Paper No. 117, 34, 2020.
	\newblock \doi{10.1007/s00526-020-01774-w}.
	
	\bibitem[Santambrogio(2017)]{Santambrogio17}
	F.~Santambrogio.
	\newblock \{{E}uclidean, metric, and {W}asserstein\} gradient flows: an
	overview.
	\newblock \emph{Bull. Math. Sci.}, 7\penalty0 (1):\penalty0 87--154, 2017.
	\newblock \doi{10.1007/s13373-017-0101-1}.
	
\end{thebibliography}
\end{document}